\documentclass[12pt,reqno]{amsart} 
\usepackage{amssymb}
\usepackage{mathrsfs}
\usepackage{enumerate}
\usepackage[all]{xy}
\usepackage[usenames,dvipsnames]{color}
\usepackage[colorlinks=true, citecolor=blue, linkcolor=blue]{hyperref}
\definecolor{darkgreen}{rgb}{0,0.5,0}
\definecolor{bluegreen}{rgb}{0,0.2,0.8}
\definecolor{darkred}{rgb}{0.8,0,0}
\definecolor{newercolor}{rgb}{0.2,0,1}
\definecolor{darkyellow}{rgb}{0.7,0.7,0}
\definecolor{darkorange}{rgb}{0.8,0.4,0}

\newcommand{\xxx}[2][k]{#2_1\times\cdots\times #2_{#1}}
\newcommand{\Xxx}[2][k]{\prod_{i=1}^{#1}#2_i}

\renewenvironment{enumerate}[1][]
{\begin{enumerat}[#1]\setlength{\itemsep}{6pt}}{\end{enumerat}}

\newenvironment{enuma}{\begin{enumerate}[{\rm(a) }]}{\end{enumerate}}
\newenvironment{enumi}{\begin{enumerate}[{\rm(i) }]}{\end{enumerate}}

\renewenvironment{itemize}
{\begin{itemiz}\setlength{\itemsep}{6pt} \setlength{\itemindent}{-20pt} }
{\end{itemiz}}

\newcommand{\8}[1]{c_{#1}}
\numberwithin{table}{section}

\newcommand{\boldd}[1]{{\mathversion{bold}\textbf{#1}}}

\newlength{\short}
\newcommand{\4}[1]{\widebar{#1}}
\newcommand{\5}[1]{\widehat{#1}}

\newcommand{\9}[1]{{}^{#1}\!}

\def\pair[#1,#2]{[\hskip-1.5pt[#1,#2]\hskip-1.5pt]}

\SelectTips{cm}{10} \UseTips   %% use computer modern tips in xy

\let\oldcirc=\circ
\renewcommand{\circ}{\mathchoice
    {\mathbin{\scriptstyle\oldcirc}}{\mathbin{\scriptstyle\oldcirc}}
    {\mathbin{\scriptscriptstyle\oldcirc}}
    {\mathbin{\scriptscriptstyle\oldcirc}}}

\def\beq#1\eeq{\begin{equation*}#1\end{equation*}}
\def\beqq#1\eeqq{\begin{equation}#1\end{equation}}

\numberwithin{equation}{section}

\newtheorem{Thm}{Theorem}[section]
\newtheorem{Prop}[Thm]{Proposition}

\newtheorem{Lem}[Thm]{Lemma}
\newtheorem{Defi}[Thm]{Definition}

\newtheorem{Hyp}[Thm]{Hypotheses}
\newtheorem{Not}[Thm]{Notation}

\newtheorem{Th}{Theorem}
\newtheorem{Cr}[Th]{Corollary}

\newcommand{\widebar}[1]
      {\overset{{\mskip3mu\leaders\hrule height0.4pt\hfill\mskip3mu}}{#1}
      \vphantom{#1}}

\newcounter{let} 
\loop\stepcounter{let}
\expandafter\edef\csname cal\alph{let}\endcsname%
{\noexpand\mathcal{\Alph{let}}}
\ifnum\thelet<26\repeat

\loop\stepcounter{let}
\expandafter\edef\csname scr\alph{let}\endcsname%
{\noexpand\mathscr{\Alph{let}}}
\ifnum\thelet<26\repeat

\newcommand{\tdef}[2][]{\expandafter\newcommand\csname#2\endcsname%
{#1\textup{#2}}}
\tdef{Iso}   \tdef{Aut}    \tdef{Out}    \tdef{Inn}    \tdef{Hom}
\tdef{End}   \tdef{Inj}    \tdef{map}    \tdef{Ker}    \tdef{Ob}
\tdef{Mor}   \tdef{Res}    \tdef{Id}     \tdef{Fr}     \tdef{Spin} 
\tdef{rk}    \tdef{conj}   \tdef{incl}   \tdef{proj}   \tdef{diag} 
\tdef{trf}   \tdef{Sol}    \tdef{He}     \tdef{Sz}     \tdef{cj}
\tdef{Rep}   \tdef{pr}    \tdef{Inndiag} \tdef{Outdiag}  \tdef{expt}
\tdef{supp}  \tdef{Isom}   \tdef{ord}    \tdef{Coker}   \tdef{Tr}
\tdef[_]{typ} \tdef[^]{op} \tdef[^]{ab}   \tdef{lcm}  \tdef{McL}
\tdef{restr}  \tdef{Comp}   \tdef{HS}

\newcommand{\fdef}[1]{\expandafter\newcommand\csname#1\endcsname%
{\mathfrak{#1}}}
\fdef{X}  \fdef{red}  \fdef{foc}  \fdef{hyp}  \fdef{Lie} \fdef{Y}

\newcommand{\bbdef}[1]{\expandafter\newcommand% 
\csname#1\endcsname{\mathbb{#1}}}
\bbdef{C} \bbdef{F} \bbdef{R} \bbdef{Z} \bbdef{N} \bbdef{Q} \bbdef{K}

\newcommand{\itdef}[1]{\expandafter\newcommand\csname#1\endcsname%
{\textit{#1}}}
\itdef{PSL}  \itdef{PSU}  \itdef{SL}  \itdef{SU}  \itdef{GL} \itdef{GU}
\itdef{Sp}   \itdef{PSp} \itdef{PSO} \itdef{SO}   \itdef{SD} \itdef{PGU} 
\itdef{PGL}  \itdef{Co}  \itdef{Fi}  \itdef{GO}

\newcommand{\POmega}{\textit{P}\varOmega}

\newcommand{\sminus}{\smallsetminus}
\newcommand{\lie}[3]{\def\test{#2}\def\tst{G}\ifx\test\tst{{}^{#1}#2_{#3}}
\else{{}^{#1}\!#2_{#3}}\fi}
\renewcommand{\*}{\,\lower6pt\hbox{\Large{\textup{*}}}\,}
\newcommand{\syl}[2]{\textup{Syl}_{#1}(#2)}
\newcommand{\sylp}[1]{\syl{p}{#1}}

\renewcommand{\Im}{\textup{Im}}
\newcommand{\autf}{\Aut_{\calf}}

\newcommand{\outf}{\Out_{\calf}}

\newcommand{\homf}{\Hom_{\calf}}
\newcommand{\isof}{\Iso_{\calf}}
\newcommand{\defeq}{\overset{\textup{def}}{=}}

\newcommand{\mxfoura}[8]{\left(\begin{smallmatrix}#1&#2&#3&#4\\#5&#6&#7&#8}
\newcommand{\mxfourb}[8]{\\#1&#2&#3&#4\\#5&#6&#7&#8\end{smallmatrix}\right)}

\let\emptyset=\varnothing
\renewcommand{\:}{\colon}
\newcommand{\pcom}{{}^\wedge_p}

\newcommand{\nsg}{\trianglelefteq}

\let\nnsg=\ntrianglelefteq

\newcommand{\til}[1]{\widetilde{#1}}
\let\too=\longrightarrow
\let\xto=\xrightarrow

\newcommand{\gen}[1]{{\langle}#1{\rangle}}
\newcommand{\Gen}[1]{{\bigl\langle}#1{\bigr\rangle}}

\newcommand{\longleft}[1]{\;{\leftarrow%
\count255=0 \loop \mathrel{\mkern-6mu}%
    \relbar\advance\count255 by1\ifnum\count255<#1\repeat}\;}
\newcommand{\longright}[1]{\;{\count255=0 \loop \relbar\mathrel{\mkern-6mu}%
    \advance\count255 by1\ifnum\count255<#1\repeat\rightarrow}\;}
\newcommand{\Right}[2]{\overset{#2}{\longright#1}}
\newcommand{\RIGHT}[3]{\mathrel{\mathop{\kern0pt\longright#1}
        \limits^{#2}_{#3}}}

\newcommand{\LEFT}[3]{\mathrel{\mathop{\kern0pt\longleft#1}\limits^{#2}_{#3}}
}
\newcommand{\dRIGHT}[3]{\mathrel{%
   \mathop{\vcenter{\baselineskip=0pt\hbox{$\kern0pt\longright#1$}%
   \hbox{$\kern0pt\longright#1$}}}\limits^{#2}_{#3}}}
\newcommand{\LRIGHT}[3]{\mathrel{%
   \mathop{\vcenter{\baselineskip=0pt\hbox{$\kern0pt\longleft#1$}%
   \hbox{$\kern0pt\longright#1$}}}\limits^{#2}_{#3}}}
\newcommand{\RLEFT}[3]{\mathrel{%
   \mathop{\vcenter{\baselineskip=0pt\hbox{$\kern0pt\longright#1$}%
   \hbox{$\kern0pt\longleft#1$}}}\limits^{#2}_{#3}}}
\newcommand{\onto}[1]{\;{\count255=0 \loop \relbar\mathrel{\mkern-6mu}%
    \advance\count255 by1
    \ifnum\count255<#1 \repeat \twoheadrightarrow}\;}

\newcommand{\m}{morphism}

\begin{document}

\title{Realizability and tameness of fusion systems}

\author{Carles Broto}
\address{Departament de Matem\`atiques, Edifici Cc, Universitat Aut\`onoma de 
Barcelona, 08193 Cerdanyola del Vall\`es (Barcelona), Spain.
\newline\indent Centre de 
Recerca Matem\`atica, Edifici Cc, Campus de Bellaterra, 08193 Cerdanyola del 
Vall\`es (Barcelona), Spain.}
\email{carles.broto@uab.cat}

\thanks{C. Broto is partially supported by MICINN grant 
PID2020-116481GB-I00 and AGAUR grant 2021-SGR-01015}.

\author{Jesper M. M\o{}ller}
\address{Matematisk Institut, Universitetsparken 5, DK--2100 K\o{}benhavn, 
Denmark}
\email{moller@math.ku.dk}
\thanks{J. M\o{}ller was partially supported by the Danish National Research 
Foundation through the Copenhagen Centre for Geometry and Topology (DNRF151)}

\author{Bob Oliver}
\address{Universit\'e Sorbonne Paris Nord, LAGA, UMR 7539 du CNRS, 
99, Av. J.-B. Cl\'ement, 93430 Villetaneuse, France.}
\email{bobol@math.univ-paris13.fr}
\thanks{B. Oliver is partially supported by UMR 7539 of the CNRS}

\author{Albert Ruiz}
\address{Departament de Matem\`atiques, Edifici Cc, Universitat Aut\`onoma de 
Barcelona, 08193 Cerdanyola del Vall\`es (Barcelona), Spain.
\newline\indent Centre de 
Recerca Matem\`atica, Edifici Cc, Campus de Bellaterra, 08193 Cerdanyola del 
Vall\`es (Barcelona), Spain.}
\email{albert.ruiz@uab.cat}

\thanks{A. Ruiz is partially supported by MICINN grant 
PID2020-116481GB-I00 and AGAUR grant 2021-SGR-01015}.

%%\date{}


\begin{abstract}

A saturated fusion system over a finite $p$-group $S$ is a category whose 
objects are the subgroups of $S$ and whose morphisms are injective 
homomorphisms between the subgroups satisfying certain axioms. A fusion 
system over $S$ is realized by a finite group $G$ if $S$ is a Sylow 
$p$-subgroup of $G$ and morphisms in the category are those induced by 
conjugation in $G$. 
One recurrent question in this subject is to find criteria as to 
whether a given saturated fusion system is realizable or not.

One main result in this paper is that a saturated fusion 
system is realizable if all of its components (in the sense of Aschbacher) are 
realizable. Another result is that all realizable fusion systems are tame: 
a finer condition on realizable fusion systems that involves describing 
automorphisms of a fusion system in terms of those of some group that 
realizes it. Stated in this way, these results depend on the classification 
of finite simple groups, but we also give more precise formulations whose 
proof is independent of the classification.

\end{abstract}

\maketitle

\bigskip

\section*{Introduction}

Let $p$ be a prime. The \emph{fusion system} of a finite group $G$ over a 
Sylow $p$-subgroup $S$ of $G$ is the category $\calf_S(G)$ whose 
objects are the subgroups of $S$ and whose morphisms are the homomorphisms 
between subgroups induced by conjugation in $G$, thus encoding 
$G$-conjugacy relations among subgroups and elements of $S$. 
With this as starting point and also motivated by questions in 
representation theory, Puig defined the concept of abstract fusion 
systems (see \cite{Puig} and Definition \ref{d:sfs}) and showed that they 
behave in many ways like finite groups.

By analogy with finite groups, a \emph{component} $\calc$ of a fusion 
system $\calf$ is a subnormal fusion subsystem that is quasisimple (i.e., 
$O^p(\calc)=\calc$ and $\calc/Z(\calc)$ is simple). The basic properties of 
components were shown by Aschbacher \cite[Theorem 6]{A-gfit} (see also 
Lemma \ref{l:A2-Th6} below). 

A fusion system $\calf$ over a finite $p$-group $S$ is \emph{realized} by a 
finite group $G$ if $S\in\sylp{G}$ and $\calf\cong\calf_S(G)$, 
and is \emph{realizable} if it is realized by some finite group. 
One of our main theorems is the following:

\begin{Th} \label{ThA}
Let $p$ be a prime, let $\calf$ be a saturated fusion system over a finite 
$p$-group, and let $\cale\nsg\calf$ be a normal fusion subsystem that 
contains all components of $\calf$. If $\cale$ is realizable, then $\calf$ 
is also realizable.
\end{Th}

The following is an immediate corollary to Theorem \ref{ThA}:

\begin{Cr} \label{CrB} 
Let $p$ be a prime, and let $\calf$ be a saturated fusion system over a 
finite $p$-group. If all components of $\calf$ are realizable, then $\calf$ 
is realizable.
\end{Cr}

Corollary \ref{CrB} is just the special case of Theorem \ref{ThA} where 
$\cale$ is the generalized Fitting subsystem of $\calf$: the central 
product of the components of $\calf$ and $O_p(\calf)$. Note, however, that 
a fusion system can be realizable even when some of its components are not.

For each component $\calc$ of $\calf$, $\calc/Z(\calc)$ is simple, and is a 
composition factor of $\calf$ (see \cite[\S\,II.10]{AKO}). Hence one 
consequence of Corollary \ref{CrB} is that $\calf$ is realizable if all of 
its composition factors are realizable. However, the converse of this is 
not true either: $\calf$ can be realizable without all of its composition 
factors being realizable.

In order to prove Theorem \ref{ThA}, we need to work with linking 
systems and tameness. 
The concept of \emph{linking systems} associated to fusion systems was 
first proposed by Benson in \cite{Benson} and in unpublished notes, and was 
developed in detail by Broto, Levi, and Oliver \cite{BLO2}. See Definition 
\ref{d:Linking} for precise definitions. This was 
originally motivated by questions involving classifying spaces of fusion 
systems and of the finite groups that they realize, but also turns out to 
be important when studying many of the purely algebraic properties of 
fusion systems. 

A fusion system $\calf$ is \emph{tamely realized} by $G$ if it is 
realized by $G$, and in addition, the natural homomorphism from $\Out(G)$ 
to $\Out(\call_S^c(G))$ is split surjective (Definitions \ref{d:kappaG} and 
\ref{d:tame}). Here, $\call_S^c(G)$ is the linking system associated to $G$ 
and to $\calf$. We say that $\calf$ is \emph{tame} if it is tamely realized 
by some finite group.

Tameness was originally defined in \cite[\S2]{AOV1}, motivated by questions 
of realizability and extensions of fusion systems, and that 
is how it is used here in the proof of Theorem \ref{ThA}. In this way, it 
also plays a role in Aschbacher's program for classifying simple fusion 
systems over $2$-groups and reproving certain parts of the classification 
of finite simple groups. See \cite[\S2.4]{AO} for more detail. 

Tameness can also be interpreted topologically. For a finite group $G$, let 
$BG\pcom$ be the classifying space of $G$ completed at $p$ in the sense of 
Bousfield and Kan, and let $\Out(BG\pcom)$ be the set of homotopy classes 
of self homotopy equivalences of $BG\pcom$. Then for $S\in\sylp{G}$, the 
fusion system $\calf_S(G)$ is tamely realized by $G$ if and only if the 
natural map from $\Out(G)$ to $\Out(BG\pcom)$ is split surjective. We refer 
to \cite[Theorem B]{BLO1}, \cite[Lemma 8.2]{BLO2}, and \cite[Lemma 
1.14]{AOV1} for the proof that $\Out(\call_S^c(G))\cong\Out(BG\pcom)$.

We can now state our second main theorem. 

\begin{Th} \label{ThC}
For each prime $p$, every realizable fusion system over a finite $p$-group 
is tame. 
\end{Th}

One of the original motivations for defining tameness in 
\cite{AOV1} was the hope that it might provide a new way to construct 
exotic fusion systems; i.e., fusion systems not realized by any finite 
group. By \cite[Theorem B]{AOV1}, if $\calf$ is a reduced fusion system 
that is not tame, then there is an extension of $\calf$ whose reduction is 
isomorphic to $\calf$ and is exotic. However, Theorem \ref{ThC} tells us 
that this procedure does not give us any new exotic examples, since if 
$\calf$ is not tame, then it is itself exotic.

A saturated fusion system $\calf$ is \emph{reduced} if $O_p(\calf)=1$ and 
$O^p(\calf)=\calf=O^{p'}(\calf)$ (see Definitions \ref{d:subgroups} and 
\ref{d:reduced}). The \emph{reduction} $\red(\calf)$ of an arbitrary 
saturated fusion system $\calf$ is the fusion system obtained by taking 
$C_\calf(O_p(\calf))/Z(O_p(\calf))$, and then alternately taking $O^p(-)$ 
or $O^{p'}(-)$ until the sequence becomes constant. By \cite[Theorem 
A]{AOV1}, $\calf$ is tame if $\red(\calf)$ is tame. So one immediate 
consequence of Theorem \ref{ThC} is:

\begin{Cr} \label{CrD}
If $\calf$ is a saturated fusion system over a finite $p$-group $S$, and 
$\red(\calf)$ is realizable, then $\calf$ is also realizable.
\end{Cr}

The proofs of Theorems \ref{ThA} and \ref{ThC} as formulated above, as well 
as those of Corollaries \ref{CrB} and \ref{CrD}, require the classification 
of finite simple groups. But they will be reformulated in Section 
\ref{s:tame-real} in a a way so as to be independent of the classification. 
Our main theorem there, Theorem \ref{t:E.real}, is independent of the 
classification and includes Theorems \ref{ThA} and \ref{ThC} as special 
cases (the latter is reformulated as Theorem \ref{t:all.tame}).

The first two sections of the paper contain mostly background material: 
some basic definitions and properties of fusion and linking systems are in 
Section \ref{s:defs}, and those of automorphism groups and tameness in 
Section \ref{s:aut}. We then deal with products in Section \ref{s:prod} and 
components of fusion systems in Section \ref{s:comp}. Theorems \ref{ThA} 
and \ref{ThC}, as well as some other applications, are shown in Section 
\ref{s:tame-real}, as Theorems \ref{t:E.real} and \ref{t:all.tame}.

The authors would like to thank the Universitat Aut\`onoma de Barcelona and 
the University of Copenhagen for their hospitality while the four authors 
met during early stages of this work; and also the French CNRS and 
BigBlueButton for helping us to meet virtually at frequent intervals to 
discuss this work during the covid-19 pandemic. We would especially like to 
thank the referee, whose very careful reading of the paper and many 
suggestions helped us to greatly improve it.

\smallskip
\noindent{\textbf{Notation:}}
The notation used in this paper is mostly standard, with a few exceptions. 
Composition of functions and functors is always from right to left. Also, 
$C_n$ denotes a (multiplicative) cyclic group of order $n$. When $G$ is a 
(multiplicative) group, $1\in G$ always denotes its identity element.

When $f\:\calc\too\cald$ is a functor, then for objects $c,c'$ in 
$\calc$, we let $f_{c,c'}$ be the induced map from $\Mor_\calc(c,c')$ to 
$\Mor_\cald(f(c),f(c'))$, and also set $f_c=f_{c,c}$ for short. 

When $G$ is a group, we indicate conjugation by setting 
$\9gx=c_g(x)=gxg^{-1}$ and $\9gH=c_g(H)=gHg^{-1}$ for $g,x\in G$ and $H\le 
G$. Also, for $P,Q\le G$, we let $\Hom_G(P,Q)$ be the set of (injective) 
homomorphisms from $P$ to $Q$ induced by conjugation in $G$, and set 
$\Aut_G(P)=\Hom_G(P,P)$.

Throughout the paper, $p$ will always be a fixed prime.

\bigskip

\section{Fusion systems and linking systems}
\label{s:defs}

This is a background section intended to provide the reader with the 
necessary basic definitions and properties of fusion and linking systems 
that will be used throughout the paper. Fusion systems and saturation were 
originally introduced by Puig, first in unpublished notes,  and then in 
\cite{Puig}. Abstract linking systems were defined in \cite{BLO2}. 
As general references for the subject we refer to \cite{AKO} and 
\cite{Craven}. 

\smallskip

\subsection{Fusion systems}

For a prime $p$, a \emph{fusion system} over a finite $p$-group $S$ is a 
category whose objects are the subgroups of $S$, and whose morphisms are 
injective homomorphisms between subgroups such that for each $P,Q\le S$:
\begin{itemize} 
	\item $\homf(P,Q)\supseteq \Hom_S(P,Q)$; and 
	\item for each $\varphi\in\homf(P,Q)$, 
	$\varphi^{-1}\in\homf(\varphi(P),P)$.
\end{itemize}
Here, $\homf(P,Q)$ denotes the set of morphisms in $\calf$ from 
$P$ to $Q$. We also write $\isof(P,Q)$ for 
the set of isomorphisms, $\autf(P)=\isof(P,P)$, and 
$\outf(P)=\autf(P)/\Inn(P)$. For $P\le S$ and $g\in S$, we set 
	\[ P^\calf = \{\varphi(P)\,|\,\varphi\in\homf(P,S)\} 
	\qquad\textup{and}\qquad 
	g^\calf = \{\varphi(g)\,|\,\varphi\in\homf(\gen{g},S)\}  \]
(the sets of subgroups and elements \emph{$\calf$-conjugate} to $P$ and 
to $g$).

The following version of the definition of a saturated fusion system is the 
most convenient one to use here. (See Definitions I.2.2 and I.2.4 and 
Proposition I.2.5 in \cite{AKO}.)

\begin{Defi} \label{d:sfs}
	Let $\calf$ be a fusion system over a finite $p$-group $S$.
	\begin{enuma}
	\item A subgroup $P\le S$ is \emph{fully normalized} \emph{(fully 
	centralized)} in $\calf$ if $|N_S(P)|\ge|N_S(Q)|$ 
	($|C_S(P)|\ge|C_S(Q)|$) for each $Q\in P^\calf$. 

	\item A subgroup $P\le S$ is \emph{fully automized} in $\calf$ if 
	$\Aut_S(P)\in\sylp{\autf(P)}$.
	
	\item A subgroup $P\le S$ is \emph{receptive} in $\calf$ if each 
	isomorphism $\varphi\in\isof(Q,P)$ in $\calf$ extends to a morphism 
	$\4\varphi\in\homf(N_\varphi,S)$, where 
		\[ N_\varphi = \{ g\in N_S(Q) \,|\, \varphi c_g \varphi^{-1} \in 
			\Aut_S(P) \}. \]

	\item The fusion system $\calf$ is \emph{saturated} if it satisfies the 
	following two conditions:
	\begin{enumerate}[\rm(I) ]
	\item \emph{(Sylow axiom)} each subgroup $P\le S$ fully normalized 
	in $\calf$ is also fully automized and fully centralized; and 
	\item \emph{(extension axiom)} each subgroup $P\le S$ fully 
	centralized in $\calf$ is also receptive.
	\end{enumerate}

	\end{enuma}
\end{Defi}

The above definition is motivated by fusion systems of finite groups. 
When $G$ is a finite group and $S\in\sylp{G}$, the \emph{$p$-fusion system} 
of $G$ is the category $\calf_S(G)$ whose objects are the subgroups of $S$, 
and where $\Mor_{\calf_S(G)}(P,Q)=\Hom_G(P,Q)$ for each $P,Q\le S$. For a 
proof that $\calf_S(G)$ is saturated, see, e.g., \cite[Lemma I.1.2]{AKO}. 
In general, a saturated fusion system $\calf$ over a finite $p$-group $S$ 
will be called \emph{realizable} if $\calf=\calf_S(G)$ for some finite 
group $G$ with $S\in\sylp{G}$, and will be called \emph{exotic} otherwise. 

The following lemma lists relations between some of these 
conditions that hold for \emph{all} fusion systems, not just those that are 
saturated. 

\begin{Lem}[{\cite[Lemma I.2.6]{AKO}}] 
If $\calf$ is a fusion system over a finite $p$-group $S$, then each 
receptive subgroup of $S$ is fully centralized, and each subgroup that is 
fully automized and receptive is fully normalized. 
\end{Lem}

We next list some of the terminology used to describe certain 
subgroups in a fusion system.

\begin{Defi} \label{d:subgroups}
Let $\calf$ be a fusion system over a finite $p$-group $S$. For a subgroup 
$P\le S$, 
\begin{enuma} 

\item $P$ is \emph{$\calf$-centric} if $C_S(Q)\le Q$ for each $Q\in 
P^\calf$;

\item $P$ is \emph{$\calf$-radical} if $O_p(\outf(P))=1$; 

\item $P$ is \emph{$\calf$-quasicentric} if for each $Q\in P^\calf$ which 
is fully centralized in $\calf$, the centralizer fusion system $C_\calf(Q)$ 
(see Definition \ref{d:CentralizerNormalizer}(b)) is the fusion 
system of the group $C_S(Q)$; 

\item $P$ is \emph{weakly closed} in $\calf$ if $P^\calf=\{P\}$;

\item $P$ is \emph{strongly closed} in $\calf$ if for each $x\in P$, 
$x^\calf\subseteq P$;

\item \label{d:subgroups:g} $P$ is \emph{normal} in $\calf$ ($P\nsg\calf$) 
if each $\varphi\in\homf(Q,R)$ (for $Q,R\le S$) extends to a morphism 
$\4\varphi\in\homf(PQ,PR)$ such that $\4\varphi(P)=P$; and 

\item \label{d:subgroups:h} $P$ is \emph{central} in $\calf$ if each 
$\varphi\in\homf(Q,R)$ (for $Q,R\le S$) extends to a morphism 
$\4\varphi\in\homf(PQ,PR)$ such that $\4\varphi|_P=\Id_P$.

\end{enuma}
Let $\calf^{cr}\subseteq\calf^c\subseteq\calf^q$ denote the sets of 
$\calf$-centric $\calf$-radical, $\calf$-centric, and $\calf$-quasicentric 
subgroups of $S$, respectively, or (depending on the context) the full 
subcategories of $\calf$ with those objects.
Let $O_p(\calf)\ge Z(\calf)$ denote the (unique) largest normal and 
central subgroups, respectively, in $\calf$.
\end{Defi}

The following result is one of the versions of Alperin's fusion 
theorem for fusion systems. 

\begin{Thm} \label{t:AFT}
Let $\calf$ be a saturated fusion system over a finite $p$-group $S$. Then 
each morphism in $\calf$ is a composite of restrictions of automorphisms of 
subgroups that are $\calf$-centric, $\calf$-radical, and fully 
normalized in $\calf$. 
\end{Thm}

\begin{proof} This follows from \cite[Theorem I.3.6]{AKO} (the same 
statement but for $\calf$-essential subgroups), together with 
\cite[Proposition I.3.3(a)]{AKO} (all $\calf$-essential subgroups are 
$\calf$-centric and $\calf$-radical). Alternatively, the result as stated 
here is shown directly (without mention of essential subgroups) in 
\cite[Theorem A.10]{BLO2}.
\end{proof}

\begin{Defi}%[{\cite[Definition~I.5.3]{AKO}}]
\label{d:CentralizerNormalizer}
Let $\calf$ be a saturated fusion system over a finite $p$-group $S$, and 
let $Q\leq S$ be a subgroup. 
\begin{enuma}

\item For each $K\le\Aut(Q)$, set $N_S^K(Q)=\{x\in N_S(Q)\,|\,c_x\in K\}$, 
and let $N_\calf^K(Q)$ be the fusion system over $N_S^K(Q)$ in which for 
$P,R\le N_S^K(Q)$, 
	\begin{align*} 
	\qquad
	\Hom_{N_\calf^K(Q)}(P,R) = \bigl\{\varphi\in\homf(P,R) &\,\big|\, 
	\textup{there is}~ \4\varphi\in\homf(PQ,RQ) \\ &~\textup{such that}~ 
	\4\varphi|_P=\varphi,~ \4\varphi(Q)=Q,~ \4\varphi|_Q\in K \bigr\}. 
	\end{align*}

\item Set $N_\calf(Q)=N_\calf^{\Aut(Q)}(Q)$ and 
$C_\calf(Q)=C_\calf^{\{1\}}(Q)$.

\end{enuma}
\end{Defi}

If $Q$ is fully normalized (fully centralized) in $\calf$ then $N_\calf(Q)$ 
($C_\calf(Q)$) is a saturated fusion system (see 
\cite[Proposition~A.6]{BLO2} or \cite[Theorem~I.5.5]{AKO}). There is 
a similar condition (see \cite[Theorem~I.5.5]{AKO}) that implies that 
$N_\calf^K(Q)$ is saturated. Note that $Q\nsg\calf$ if and only if 
$N_\calf(Q)=\calf$, and $Q$ is central in $\calf$ (i.e., $Q\le Z(\calf)$) 
if and only if $C_\calf(Q)=\calf$.

\begin{Lem} \label{l:NG(Q)}
Let $G$ be a finite group with $S\in\sylp{G}$, and set $\calf=\calf_S(G)$.
\begin{enuma}
    \item For each $Q\le S$, $Q$ is fully normalized (fully centralized) if
and only if $N_S(Q)\in\sylp{N_G(Q)}$ ($C_S(Q)\in\sylp{C_G(Q)}$). If
this holds, then $N_\calf(Q)=\calf_{N_S(Q)}(N_G(Q))$
($C_\calf(Q)=\calf_{C_S(Q)}(C_G(Q))$).
    \item In all cases, $O_p(G)\le O_p(\calf)$ and $O_p(Z(G))\le Z(\calf)$.
\end{enuma}
\end{Lem}

\begin{proof} Point (a) is shown in \cite[Proposition I.5.4]{AKO}. In
	particular, when $Q=O_p(G)$, we have
	$N_\calf(Q)=\calf_S(G)=\calf$ and hence $Q\le O_p(\calf)$. Similarly, when $Q=O_p(Z(G))$, it says that $C_\calf(Q)=\calf_S(G)=\calf$ and
	hence that $Q\le Z(\calf)$.
\end{proof}

\smallskip

\subsection{Linking systems} 

Before recalling the definition of linking systems, we need to 
introduce more notation. If $P,Q\leq G$ are subgroups of a finite group 
$G$, the \emph{transporter set $T_G(P,Q)$} is defined by setting 
	\[ T_G(P,Q)=\{g \in G \mid \9gP\leq Q\}. \]
The \emph{transporter category of $G$} is the category $\calt(G)$ whose 
objects are the subgroups of $G$, and whose morphisms sets are the 
transporter sets:
	\[ \Mor_{\calt(G)}(P,Q)=T_G(P,Q). \]
Composition in $\calt(G)$ is given by multiplication in $G$. If $\calh$ is a set 
of subgroups of $G$, then $\calt_\calh(G)\subseteq\calt(G)$ denotes the 
full subcategory with object set $\calh$.

The following definition of linking system taken from 
\cite[Definition III.4.1]{AKO}.

\begin{Defi}\label{d:Linking}
	Let $\calf$ be a fusion system over a finite $p$-group $S$. A \emph{linking system} associated to $\calf$ is a triple $(\call,\delta,\pi)$ where $\call$ is a finite category, and $\delta$ and $\pi$ are a pair of functors
	\[
	\calt_{\Ob(\call)}(S) \Right3{\delta} \call \Right3{\pi} \calf 
	\]
	that satisfy the following conditions:
	\begin{enumerate}[\rm ({A}1)]
	\item $\Ob(\call)$ is a set of subgroups of $S$ closed under 
	$\calf$-conjugacy and overgroups, and contains $\calf^{cr}$. Each 
	object in $\call$ is isomorphic (in $\call$) to one which is fully 
	centralized in $\calf$. 
	\item $\delta$ is the identity on objects, and $\pi$ is the 
	inclusion on objects. For each $P,Q\in\Ob(\call)$ such that $P$ is 
	fully centralized in $\calf$, $C_S(P)$ acts freely on 
	$\Mor_\call(P,Q)$ via $\delta_P$ and right composition, and
		\[ \pi_{P,Q}\colon \Mor_\call(P,Q) \to \Hom_\calf(P,Q) \]
		is the orbit map for this action. 
	\end{enumerate}
	\begin{enumerate}[\rm (A)] \setcounter{enumi}{1}
	\item For each $P,Q\in\Ob(\call)$ and each $g\in T_S(P,Q)$, 
	$\pi_{P,Q}$ sends $\delta_{P,Q}(g)\in\Mor_\call(P,Q)$ to 
	$c_g\in\Hom_\calf(P,Q)$.
		
	\item For all $\psi\in\Mor_\call(P,Q)$ and all $g\in P$, the diagram
		\[
		\xymatrix{ P \ar[r]^\psi \ar[d]_{\delta_P(g)} & Q \ar[d]^{\delta_Q(\pi(\psi)(g))} \\
			P \ar[r]^\psi & Q
		}
		\]
		commutes in $\call$.

	\end{enumerate} 
When the functors $\delta$ and $\pi$ are understood, we refer 
directly to the category $\call$ as a linking system. 
\end{Defi} 
	
A \emph{centric linking system} associated to $\calf$ is a linking system 
$\call$ associated to $\calf$ such that $\Ob(\call)=\calf^c$.

Linking systems associated to a fusion system were originally 
motivated by centric linking systems of finite groups. For a finite group 
$G$, a $p$-subgroup $P\leq G$ is \emph{$p$-centric in $G$} 
if $Z(P)\in\sylp{C_G(P)}$; equivalently, if $C_G(P)=Z(P)\times 
O_{p'}(C_G(P))$. For $S\in\sylp{G}$, the 
\emph{centric linking system of $G$ over $S$} is consists of the category 
$\call^c_S(G)$ whose objects are the subgroups of $S$ which are $p$-centric 
in $G$ and whose morphism sets are given by
	\[ \Mor_{\call^c_S(G)}(P,Q)=T_G(P,Q)/O_{p'}(C_G(P)) 
	\qquad \textup{(all $P,Q\in\Ob(\call_S^c(G))$),} \]
together with functors 
$\calt_{\Ob(\call^c_S(G))}(S)\xto{~\delta~}\call_S^c(G)\xto{~\pi~} 
\calf_S(G)$ defined in the obvious way. 

When $G$ is a finite group and $S\in\sylp{G}$, then $P\leq S$ 
is $\calf$-centric (see Definition \ref{d:subgroups}) if and only if $P$ is 
$p$-centric in $G$ (see \cite[Lemma A.5]{BLO1}). Moreover, $\calf_S(G)$ is 
always saturated (see \cite[Theorem I.2.3]{AKO}), and 
$(\call^c_S(G),\delta,\pi)$ is a centric linking system 
associated to $\calf_S(G)$.

Some of the basic properties of linking systems are listed in the next 
proposition.

\begin{Prop} \label{L-prop}
Let $(\call,\delta,\pi)$ be a linking system associated to a saturated 
fusion system $\calf$ over a finite $p$-group $S$. For each pair of 
subgroups $P\le Q\le S$ with $P,Q\in\Ob(\call)$, set 
$\iota_{P,Q}=\delta_{P,Q}(1)\in\Mor_\call(P,Q)$ (the \emph{inclusion} in 
$\call$ of $P$ into $Q$). Then 
\begin{enuma}  

\item $\delta$ is injective on all morphism sets; and 

\item all morphisms in $\call$ are monomorphisms and epimorphisms in the 
categorical sense.

\end{enuma}
Conditions for the existence of restrictions and extensions of morphisms 
are as follows:
\begin{enuma} \setcounter{enumi}{2}

\item For every morphism $\psi\in\Mor_\call(P,Q)$, and every 
$P_0,Q_0\in\Ob(\call)$ such that $P_0\le{}P$, $Q_0\le{}Q$, and 
$\pi(\psi)(P_0)\le{}Q_0$, there is a unique morphism 
$\psi|_{P_0,Q_0}\in\Mor_\call(P_0,Q_0)$ (the ``restriction'' of $\psi$) 
such that $\psi\circ\iota_{P_0,P}=\iota_{Q_0,Q}\circ\psi|_{P_0,Q_0}$.  

\item Let $P,Q,\widebar{P},\widebar{Q}\in\Ob(\call)$ and 
$\psi\in\Mor_\call(P,Q)$ be such that $P\nsg\widebar{P}$, 
$Q\le\widebar{Q}$, and for each $g\in\widebar{P}$ there is 
$h\in\widebar{Q}$ such that 
$\iota_{Q,\4Q}\circ\psi\circ\delta_P(g)=\delta_{Q,\4Q}(h)\circ\psi$. 
Then there is a unique morphism 
$\widebar{\psi}\in\Mor_\call(\widebar{P},\widebar{Q})$ such that 
$\widebar{\psi}|_{P,Q}=\psi$. 

\end{enuma}
\end{Prop}

\begin{proof} See points (c), (f), (b), and (e), respectively, in 
\cite[Proposition 4]{O-linkext}.
\end{proof}

We note here the existence and uniqueness of linking systems shown by 
Chermak, Oliver, and Glauberman-Lynd. Two linking systems 
$(\call_1,\delta_1,\pi_1)$ and $(\call_2,\delta_2,\pi_2)$ associated to the 
same fusion system $\calf$ are \emph{isomorphic} 
if there is an isomorphism 
of categories $\rho\:\call_1\xto{~\cong~}\call_2$ such that 
$\rho\circ\delta_1=\delta_2$ and $\pi_2\circ\rho=\pi_1$.

\begin{Thm}[{\cite{Chermak,O-Ch,GLynd}}] \label{t:unique.l.s.}
	Let $\calf$ be a saturated fusion system over a finite $p$-group $S$, and 
	let $\calh$ be a set of subgroups of $S$ such that 
	$\calf^{cr}\subseteq\calh\subseteq\calf^q$, and such that $\calh$ is closed 
	under $\calf$-conjugacy and overgroups. Then up to isomorphism, there is a unique linking 
	system $\call^\calh$ associated to $\calf$ with object set $\calh$. 
\end{Thm}

\begin{proof}  The existence and uniqueness of a centric linking system 
associated to $\calf$ was shown by Chermak. See \cite[Main 
theorem]{Chermak} and \cite[Theorem A]{O-Ch} for two versions of his 
original proof, and \cite[Theorem 1.2]{GLynd} for the changes to the proof 
in \cite{O-Ch} needed to make it independent of the classification of 
finite simple groups.

More generally, if $\calf^{cr} \subseteq \calh\subseteq\calf^c$, the 
uniqueness of an $\calh$-linking system follows by the same obstruction 
theory (shown to vanish in \cite[Theorem 3.4]{O-Ch} and \cite[Theorem 
1.1]{GLynd}) as that used in the centric case (by the same argument as in 
the proof of \cite[Proposition 3.1]{BLO2}). For arbitrary 
$\calh\subseteq\calf^q$ containing $\calf^{cr}$, the existence and 
uniqueness now follows from \cite[Proposition III.4.8]{AKO}, applied with 
$\calh\supseteq\calh\cap\calf^c$ in the role of $\5\calh\supseteq\calh$. 
\end{proof}

\smallskip

\subsection{Normal fusion and linking subsystems}

Let $\calf$ be a fusion system over a finite $p$-group $S$. A \emph{fusion 
subsystem} of $\calf$ is a subcategory $\cale\subseteq\calf$ which is 
itself a fusion system over a subgroup $T\le S$ (in particular, 
$\Ob(\cale)$ is the set of subgroups of $T$). We write $\cale\le\calf$ when 
$\cale$ is a fusion subsystem, and also sometimes say that $\cale\le\calf$ 
is a pair of fusion systems over $T\le S$.

\begin{Defi} \label{d:isofus}
Let $\calf$ be a fusion system over a finite $p$-group $S$. 
\begin{enuma}
\item Let $R$ be another finite $p$-group and let $\alpha\colon S 
\too R$ be an isomorphism. We denote by $\9\alpha\calf$ the fusion 
system over $R$ with morphism sets
	$$ \Hom_{\9\alpha\calf}(P,Q) = \alpha 
	\circ\Hom_\calf(\alpha^{-1}(P),\alpha^{-1}(Q))\circ \alpha^{-1}$$
for each pair of subgroups $P, Q\leq R$. 

\item Let $\cale$ be another fusion system over a finite $p$-group $T$. We 
say that $\cale$ and $\calf$ are isomorphic fusion systems if there is an 
isomorphism $\alpha\colon S\too T$ such that $\cale =\9\alpha\calf$. 

\end{enuma}	        
\end{Defi}	

A more general concept of morphism between fusion systems is 
given in \cite[Definition II.2.2]{AKO}. 

Consider now the following definition from \cite[Definition I.6.1]{AKO}.

\begin{Defi} \label{d:Fnormal}
Fix a saturated fusion system $\calf$ over a finite $p$-group $S$.
\begin{enuma}

\item A fusion subsystem $\cale\leq\calf$ over $T\nsg S$ is \emph{weakly 
normal} if $\cale$ is saturated, $T$ is strongly closed in $\calf$, and the 
following conditions hold: \smallskip
\begin{itemize}

\item \emph{(invariance condition)} $\9\alpha\cale=\cale$ for each 
$\alpha\in\Aut_{\calf}(T)$, and

\item \emph{(Frattini condition)} for each $P\leq T$ and each $\varphi \in 
\Hom_{\calf}(P,T)$, there are $\alpha\in\Aut_{\calf}(T)$ and 
$\varphi_0\in\Hom_{\cale}(P,T)$ such that $\varphi=\alpha\circ\varphi_0$.

\end{itemize}

\item A fusion subsystem $\cale\leq\calf$ over $T\nsg S$ is \emph{normal 
($\cale\nsg\calf$)} if $\cale$ is weakly normal in $\calf$ and 
\smallskip
\begin{itemize}

\item \emph{(Extension condition)} each $\alpha\in\Aut_{\cale}(T)$ extends 
to $\4{\alpha}\in\Aut_{\calf}(TC_S(T))$ such that $[\4{\alpha},C_S(T)]\leq 
Z(T)$.

\end{itemize}

\item A saturated fusion system $\calf$ over a finite $p$-group $S$ is 
\emph{simple} if it contains no proper nontrivial normal fusion subsystem.

\end{enuma}
\end{Defi}

It will be convenient to say that ``$\cale\nsg\calf$ is a normal pair 
of fusion systems over $T\nsg S$'' to mean that $\calf$ is a fusion system 
over $S$ and $\cale\nsg\calf$ is a normal subsystem over $T$.

Note that if $\calf$ is a saturated fusion system over a finite 
$p$-subgroup $S$, and $P\leq S$, then  $P \trianglelefteq \calf$ if and 
only if $\calf_P(P) \trianglelefteq \calf$ \cite[(7.9)]{A-gfit}. 

Note also that what are called ``normal fusion subsystems'' 
in \cite[Definition 1.18]{AOV1} are what we are calling 
``weakly normal'' subsystems here.

When $(\call,\delta,\pi)$ is a linking system associated to the fusion 
system $\calf$ over $S$, and $\calf_0\le\calf$ is a fusion subsystem over 
$S_0\le S$, then a \emph{linking subsystem} associated to 
$\calf_0$ is a linking system $(\call_0,\delta_0,\pi_0)$ associated to 
$\calf_0$, where $\call_0$ is a subcategory of $\call$ and 
	\[ \calt_{\Ob(\call_0)}(S_0) \Right4{\delta_0} \call_0 
	\Right4{\pi_0} \calf_0 \]
are the restrictions of $\delta$ and $\pi$. In this situation, we write 
$\call_0\le\call$, and sometimes say that $\call_0\le\call$ is a pair of 
linking systems. Note in particular the special case where $S_0=S$ and 
$\calf_0=\calf$ but $\Ob(\call_0)\subseteq\Ob(\call)$: a pair of 
linking systems with possibly different object sets associated to 
the same fusion system.

\begin{Defi} \label{d:Lnormal}
Fix a pair of saturated fusion systems $\cale\leq\calf$ over finite 
$p$-groups $T\nsg S$ such that $\cale\nsg\calf$, and let $\calm\leq\call$ 
be a pair of associated linking systems. Then, \emph{$\calm$ is normal in 
$\call$ ($\calm\nsg\call$)} if:
\begin{enuma}
\item $\Ob(\call)=\{P\leq S \mid P \cap T \in \Ob(\calm)\}$, and
\item for all $\gamma\in\Aut_{\call}(T)$ and $\psi\in\Mor(\calm)$, 
$\gamma\psi\gamma^{-1}\in\Mor(\calm)$.
\end{enuma}
If $\calm\nsg\call$, then we define 
$\call/\calm=\Aut_\call(T)/\Aut_\calm(T)$.
\end{Defi}

Notice that not every normal pair of fusion systems 
has an associated normal pair of linking systems.

Definition \ref{d:Lnormal} differs from Definition 1.27 in \cite{AOV1} 
in that there is no ``Frattini condition'' in the definition we give here. 
We have omitted it since it follows from the Frattini condition for normal 
fusion subsystems, as shown in the next lemma.

\begin{Lem} \label{l:Lnormal}
If $\calm\nsg\call$ is a normal pair of linking systems associated to 
fusion systems $\cale\nsg\calf$ over finite $p$-groups $T\nsg 
S$, then for all $P,Q \in \Ob(\calm)$ and all $\psi \in \Mor_{\call}(P,Q)$, 
there are morphisms $\gamma \in \Aut_{\call}(T)$ and 
$\psi_0\in\Mor_{\calm}(\gamma(P),Q)$ such that 
$\psi=\psi_0\circ\gamma|_{P,\gamma(P)}$.
\end{Lem}

\begin{proof} Let $\psi\in \Mor_\call(P,Q)$ be as above, and assume first 
that $P$ is fully centralized in $\calf$. Then 
$\pi(\psi)\in\Hom_\calf(P,Q)$, and by the Frattini condition on 
$\cale\nsg\calf$, there are $\alpha\in \Aut_\calf(T)$ and 
$\varphi\in\Hom_\cale(\alpha(P), Q)$ such that $\pi(\psi)= \varphi\circ 
\alpha|_{P,\alpha(P)}$. Choose $\til\varphi\in 
\Mor_\calm(\alpha(P),Q)$ and $\til\alpha\in \Aut_\call(T)$ such that 
$\pi_\calm(\til\varphi)=\varphi$ and $\pi_\call(\til\alpha)=\alpha$. By 
axiom (A2) for the linking system $\call$, there is $z\in C_S(P)$ such that 
$\til\varphi \circ \til\alpha|_{P,\alpha(P)} = \psi \circ 
\delta_P(z)|_{P,P}$. (Here, we take restrictions of morphisms in 
$\calm$ in the sense of Proposition \ref{L-prop}(c).) Set $\gamma = 
\til\alpha\circ \delta_S(z)^{-1}\in \Aut_\call(T)$, and then $\psi = 
\til\varphi\circ \gamma|_{P,\gamma(P)}$.

If $P$ is not fully centralized in $\calf$, then choose $P^*\in P^\calf$ 
that is fully centralized, and fix $\omega\in\Iso_\call(P^*,P)$. Then 
$P^*\le T$ since $T$ is 
strongly closed in $\calf$, and so $P^*\in \Ob(\calm)$. We just showed that 
there are morphisms $\gamma_1,\gamma_2\in\Aut_\call(T)$, 
$\psi_1\in\Mor_\calm(\gamma_1(P^*),Q)$, and 
$\psi_2\in\Mor_\calm(\gamma_2(P^*),P)$ such that 
	\[ \psi\circ\omega = \psi_1\circ\gamma_1|_{P^*,\gamma_1(P^*)} 
	\qquad\textup{and}\qquad \omega = 
	\psi_2\circ\gamma_2|_{P^*,\gamma_2(P^*)} . \]
Also, $\psi_2$ is an isomorphism since $\omega$ is an isomorphism. 
Set $\gamma=\gamma_1\gamma_2^{-1}\in\Aut_\call(T)$. Then 
	\[ \psi = (\psi\circ\omega)\circ\omega^{-1} = 
	\psi_1\circ\gamma|_{\gamma_2(P^*),\gamma_1(P^*)}\circ\psi_2^{-1} = 
	\psi_1\circ(\gamma\psi_2\gamma^{-1})^{-1}
	\circ\gamma|_{P,\gamma(P)} \]
where $\gamma\psi_2\gamma^{-1}\in\Iso_\calm(\gamma_1(P^*),\gamma(P))$ 
by Definition \ref{d:Lnormal}(b).
\end{proof}

\smallskip

\subsection{Fusion subsystems of \texorpdfstring{$p$-}{p-}power index and 
index prime to \texorpdfstring{$p$}{p}}

We recall some more definitions. 

\begin{Defi} \label{d:reduced}
Let $\calf$ be a saturated fusion system over a finite $p$-group $S$.
\begin{enuma} 

\item Set $\foc(\calf)=\{g^{-1}h \,|\, g,h\in S,~ h\in g^\calf \} = 
\bigl\{g^{-1}\alpha(g) \,\big|\, g\in P\le S,~ \alpha\in\autf(P) \bigr\}$ 
(the \emph{focal subgroup} of $\calf$).

\item Set $\hyp(\calf)=\bigl\{g^{-1}\alpha(g) \,\big|\, g\in P\le S,~ 
\alpha\in O^p(\autf(P)) \bigr\}$ (the \emph{hyperfocal subgroup} of 
$\calf$).

\item A saturated fusion subsystem $\cale\le\calf$ over $T\le 
S$ has \emph{$p$-power index} if $T\ge\hyp(\calf)$, and 
$\Aut_\cale(P)\ge O^p(\autf(P))$ for all $P\le T$. The smallest normal 
subsystem of $p$-power index is denoted $O^p(\calf)$.

\item A saturated fusion subsystem $\cale\le\calf$ over $T\le 
S$ has \emph{index prime to $p$} if $T=S$ and $\Aut_\cale(P)\ge 
O^{p'}(\autf(P))$ for all $P\le T$. The smallest normal subsystem of 
index prime to $p$ is denoted $O^{p'}(\calf)$.

\end{enuma}
\end{Defi}

For the existence of the minimal subsystems $O^p(\calf)$ and 
$O^{p'}(\calf)$, see, e.g., Theorems I.7.4 and I.7.7 in \cite{AKO}.

\begin{Lem} \label{l:foc(F)}
\begin{enuma} 

\item If $G$ is a finite group with $S\in\sylp{G}$, then 
$\foc(\calf_S(G))=S\cap[G,G]$ and $\hyp(\calf_S(G))=S\cap O^p(G)$. 

\item If $\calf$ is a saturated fusion system over a finite $p$-group $S$, 
then $O^p(\calf)$ and $O^{p'}(\calf)$ are fusion subsystems over 
$\hyp(\calf)$ and $S$, respectively, and are both normal in $\calf$. Also, 
	\[ O^p(\calf)=\calf ~\iff~ \hyp(\calf)=S ~\iff~ \foc(\calf)=S. \]

\end{enuma}
\end{Lem}

\begin{proof} The first statement in (a) is the focal subgroup theorem (see 
\cite[Theorem 7.3.4]{Gorenstein}), and the second is Puig's 
hyperfocal subgroup theorem \cite[\S1.1]{Puig-hfoc}. 

Point (b) is due to Puig, and is also shown in Theorems I.7.4 and I.7.7 and 
Corollary I.7.5 in \cite{AKO}. 
\end{proof}

\begin{Lem} \label{l:C(Q)<|N(Q)}
Let $\calf$ be a saturated fusion system over a finite $p$-group $S$, fix 
$Q\le S$, and let $K\le\Aut(Q)$ be a subgroup of $p$-power order. Assume 
that $Q$ is fully centralized in $\calf$ and that 
$\Aut_S^K(Q)\in\sylp{\autf^K(Q)}$. Then $N_\calf^K(Q)$ is saturated, 
and $C_\calf(Q)$ is normal of $p$-power index in $N_\calf^K(Q)$. 
\end{Lem}

\begin{proof} By Proposition I.5.2 and Theorem I.5.5 in \cite{AKO}, the 
fusion systems $C_\calf(Q)$ and $N_\calf^K(Q)$ are saturated. So it 
suffices to prove the lemma when $\calf=N_\calf^K(Q)$; i.e., 
when $Q\nsg\calf$ and $K\ge\autf(Q)$. Then $C_\calf(Q)\nsg\calf$ by 
\cite[Proposition 8.8]{Craven}.

We claim that
	\begin{multline} 
	P\le S,~ \textup{$\alpha\in\autf(P)$ of order prime to $p$} 
	~\implies~ \\
	\textup{$\alpha$ extends to $\4\alpha\in\autf(PQ)$ with 
	$\alpha|_Q=\Id_Q$} \label{e:xyz} 
	\end{multline}
Since $Q\nsg\calf$, $\alpha$ extends to $\4\alpha\in\autf(PQ)$, 
and we can arrange that $\4\alpha$ also has order prime 
to $p$. But $\4\alpha|_Q\in K$ by assumption, hence has $p$-power order, 
and so $\4\alpha|_Q=\Id_Q$. 

We next check that $C_S(Q)\ge\hyp(\calf)$. Fix $P\le S$, and 
$\alpha\in\autf(P)$ of order prime to $p$. By \eqref{e:xyz}, $\alpha$ 
extends to $\4\alpha\in\autf(PQ)$ such that $\4\alpha|_Q=\Id_Q$. But 
then $[[\alpha,P],Q]\le[[\4\alpha,PQ],Q]=1$ by the 3-subgroup lemma (see 
\cite[Theorem 2.2.3]{Gorenstein} or \cite[8.7]{A-FGT}) and 
since $[PQ,Q]\le Q$ and $[\4\alpha,Q]=1$. So $[\alpha,P]\le C_S(Q)$. Since 
$\hyp(\calf)$ is generated by such subgroups $[\alpha,P]$, this proves that 
$\hyp(\calf)\le C_S(Q)$. 

It remains to show, for all $P\le C_S(Q)$, that $\Aut_{C_\calf(Q)}(P)$ 
contains $O^p(\autf(P))$. But this follows directly from \eqref{e:xyz}, 
which says that each $\alpha\in\autf(P)$ of order prime to 
$p$ lies in $\Aut_{C_\calf(Q)}(P)$. Thus $C_\calf(Q)$ has $p$-power index 
in $\calf$.
\end{proof}

\smallskip

\subsection{Quotient fusion systems}
We begin with the basic definition and properties.

\begin{Defi} \label{d:F/Q}
Let $\calf$ be a fusion system over a finite $p$-group $S$, and assume 
$Q\nsg S$ is strongly closed in $\calf$. Then $\calf/Q$ is defined to be 
the fusion system over $S/Q$ where
	\[ \Hom_{\calf/Q}(P/Q,R/Q) = \bigl\{\varphi/Q\,\big|\, 
	\varphi\in\homf(P,R) \bigr\} \]
for $P,R\le S$ containing $Q$. Here, $\varphi/Q\in\Hom(P/Q,R/Q)$ sends $gQ$ 
to $\varphi(g)Q$. 
\end{Defi}

Note that by definition, $\calf/Q=N_\calf(Q)/Q$. If $\calf=\calf_S(G)$ for 
some finite group $G$ with $S\in\sylp{G}$, and $H\nsg G$ is such that 
$Q=H\cap S$, then $\calf/Q\cong\calf_{SH/H}(G/H)$ (see \cite[Theorem 
5.20]{Craven}). 

\begin{Lem} \label{l:F/Q}
Let $\calf$ be a saturated fusion system over a finite $p$-group $S$, and 
assume $Q\nsg S$ is strongly closed in $\calf$. Then $\calf/Q$ is 
saturated. If $\cale\nsg\calf$ is a normal fusion subsystem over $T\nsg S$ 
such that $T\ge Q$, then $\cale/Q\nsg\calf/Q$. 
\end{Lem}

\begin{proof} For a proof that $\calf/Q$ is saturated, see 
\cite[Proposition 5.11]{Craven} or \cite[Lemma II.5.4]{AKO}. If 
$\cale\nsg\calf$ over $T\ge Q$, then $\cale/Q$ is saturated since 
$\cale$ is, $T/Q$ is strongly closed in $\calf/Q$, and the invariance and 
Frattini conditions for normality of $\cale/Q\le\calf/Q$ follow immediately 
from those for $\cale\nsg\calf$ (see \cite[Lemma 5.59]{Craven} for 
details). 

It remains to prove the extension condition for $\cale/Q\le\calf/Q$. We 
must show, for each $\varphi\in\Aut_{\cale/Q}(T/Q)$, that $\varphi$ extends 
to some $\4\varphi\in\Aut_{\calf/Q}((T/Q)C_{S/Q}(T/Q))$ such 
that $[\4\varphi,C_{S/Q}(T/Q)]\le Z(T/Q)$. This clearly holds when 
$\varphi\in\Inn(T/Q)\in\sylp{\Aut_{\cale/Q}(T/Q)}$, so it will suffice to 
prove it when $\varphi$ has order prime to $p$. Let $U\le S$ be 
such that $Q\le U$ and $C_{S/Q}(T/Q)=U/Q$. 

By \cite[Theorem 5]{A-gfit} or \cite[Theorem 1]{Henke-prod}, there is a 
(unique) saturated fusion subsystem $\cale S\le\calf$ over $S$ such that 
$\cale$ is normal of $p$-power index in $\cale S$. 
Since $\cale S$ is saturated, so is $\cale S/Q$. So by the extension axiom, 
$\varphi$ extends to some $\4\varphi\in\Aut_{\cale S/Q}(TU/Q)$, and upon 
replacing $\4\varphi$ by $\4\varphi{}^k$ for some $k$, we can arrange that 
$\4\varphi$ have order prime to $p$. Then $\4\varphi=\5\varphi/T$ for some 
$\5\varphi\in\Aut_{\cale S}(TU)$, and we can again arrange that $\5\varphi$ 
have order prime to $p$. 

By definition of the hyperfocal subgroup, $[\5\varphi,TU]\le\hyp(\cale 
S)\le T$, where the last inclusion holds since $O^p(\cale S)\le\cale$. Thus 
$[\4\varphi,TU/Q]\le T/Q$, and so 
$[\4\varphi,C_{S/Q}(T/Q)]=[\4\varphi,U/Q]\le(T\cap U)/Q=Z(T/Q)$. (We 
thank the referee for pointing out this short argument.)
\end{proof}

We refer to \cite[\S\,II.5]{AKO} and \cite[\S\,5.2]{Craven} for some 
of the other properties of these quotient systems. 

The next lemma involves normal fusion subsystems of index prime to 
$p$ (see Definition \ref{d:reduced}).

\begin{Lem} \label{l:Op'(F/Z)}
Let $\calf$ be a saturated fusion system over a finite $p$-group $S$, and 
let $Z\le Z(\calf)$ be a central subgroup. Then 
$O^{p'}(\calf/Z)=O^{p'}(\calf)/Z$.
\end{Lem}

\begin{proof} Set $\calh=\{P\in\calf^c\,|\,P/Z\in(\calf/Z)^c\}$. (Note that 
for $P\in\calf^c$, $P\ge Z(S)\ge Z$.) For each $P\in\calf^c\sminus\calh$ 
that is fully normalized in $\calf$, $P/Z$ is fully normalized in 
$\calf/Z$, and hence $C_{S/Z}(P/Z)\nleq P/Z$. Choose $x\in S\sminus P$ such 
that $xZ\in C_{S/Z}(P/Z)$, and consider the automorphism $c_x\in\Aut_S(P)$. 
Then $c_x\notin\Inn(P)$ since $P\in\calf^c$, and $c_x$ induces the identity on 
$Z$ and on $P/Z$. Since $\{\alpha\in\autf(P)\,|\,[\alpha,P]\le Z\}$ is a 
normal $p$-subgroup of $\autf(P)$ (see \cite[Corollary 5.3.3]{Gorenstein}), 
this proves that $c_x\in O_p(\autf(P))$, and hence that 
	\beqq \textup{for each $P\in\calf^c\sminus\calh$ there is 
	$P^*\in P^\calf$ with $\Out_S(P^*)\cap O_p(\outf(P^*))\ne1$.} 
	\label{e:Op'} \eeqq

By \cite[Theorem I.7.7]{AKO}, there is a finite group $\Gamma$ of order 
prime to $p$, and a map $\theta\:\Mor((\calf/Z)^c)\too\Gamma$ that sends 
composites to products and inclusions to the identity, and is such that 
$\theta(\Aut_{\calf/Z}(S/Z))=\Gamma$ and 
$O^{p'}(\calf/Z)=\gen{\theta^{-1}(1)}$. Let $\calf^\calh\subseteq\calf$ be 
the full subcategory with object set $\calh$, and let $\Phi$ be the natural 
map from $\Mor(\calf^\calh)$ to $\Mor((\calf/Z)^c)$. Set 
$\calf_0=\gen{(\theta\Phi)^{-1}(1)}$: a fusion subsystem of $\calf$ over 
$S$. By \cite[Lemma 1.6]{O-red} and \eqref{e:Op'}, $\calf_0\ge 
O^{p'}(\calf)$ and is saturated. Thus $O^{p'}(\calf)/Z\le\calf_0/Z\le 
O^{p'}(\calf/Z)$.

Conversely, $O^{p'}(\calf)/Z$ has index prime to $p$ 
in $\calf/Z$ since for each $P/Z\le S/Z$, we have 
$\Aut_{O^{p'}(\calf)/Z}(P/Z)\ge O^{p'}(\Aut_{\calf/Z}(P/Z))$. So 
$O^{p'}(\calf)/Z\ge O^{p'}(\calf/Z)$.
\end{proof}

The following construction is needed when we want to look at the image of 
$\cale\nsg\calf$ in $\calf/Q$ but $\cale$ does not contain $Q$.

\begin{Defi} \label{d:ZE/Z}
Let $\cale\le\calf$ be a pair of saturated fusion systems over $T\le 
S$, and let $Z\le Z(\calf)$ be a central subgroup. Define 
$Z\cale\le\calf$ to be the fusion subsystem over $ZT$ where for each 
$P,Q\le ZT$, 
	\[ \Hom_{Z\cale}(P,Q) = \{ \varphi\in\homf(P,Q) \,|\, 
	\varphi|_{P\cap T}\in\Hom_\cale(P\cap T,Q\cap T) \}. \]
\end{Defi}

If $\cale\nsg\calf$, then the above definition is a special case of a 
construction of Aschbacher \cite[Theorem 5]{A-gfit}. But the definition and 
arguments in this very restricted case are much more elementary.

\begin{Lem} \label{l:E/Z<|F/Z}
Let $\cale\le\calf$ be a pair of saturated fusion systems over 
finite $p$-groups $T\le S$. Let $Z\le Z(\calf)$ be a central subgroup. 
Then $Z\cale$ is saturated, and $Z\cale\nsg\calf$ if $\cale\nsg\calf$.
\end{Lem}

\begin{proof} A subgroup $P\le ZT$ is fully normalized or fully centralized 
in $Z\cale$ if and only if $P\cap T$ is fully normalized or fully 
centralized in $\cale$. The saturation axioms for $Z\cale$ follow easily 
from those for $\cale$: note, for example, that 
$\Aut_{Z\cale}(P)\cong\Aut_\cale(P\cap T)$ for $P\le ZT$. So $Z\cale$ is 
saturated. 

If $\cale\nsg\calf$, then the subgroup $ZT$ is strongly closed in $\calf$ 
since for each $x=zt$ (for $z\in Z$ and $t\in T$), each 
$\varphi\in\homf(\gen{x},S)$ extends to $\4\varphi\in\homf(Z\gen{t},S)$, 
and $\varphi(x)=z\4\varphi(t)\in ZT$. 
The extension condition for $Z\cale$ follows directly from that for 
$\cale\nsg\calf$, and the invariance and Frattini conditions for 
$Z\cale$ follow from those conditions applied to $\cale\nsg\calf$ and 
the definition of a central subgroup. 
Thus $Z\cale\nsg\calf$. 
\end{proof}

The following lemma will also be useful. 

\begin{Lem} \label{l:F/Q=E/Q0}
Let $\calf$ be a saturated fusion system over a finite $p$-group $S$, and 
let $\cale\nsg\calf$ be a normal fusion subsystem over $T\nsg S$. Let 
$Q\nsg S$ be a subgroup strongly closed in $\calf$, and assume that 
\begin{enumi} 
\item each morphism in $\calf$ between subgroups of $T$ lies in $\cale$ 
(i.e., $\cale$ is a full subcategory of $\calf$), and 
\item $S=QT$.
\end{enumi}
Then $Q\cap T$ is strongly closed in $\cale$ and $\calf/Q\cong\cale/(Q\cap T)$. 
\end{Lem}

\begin{proof} Since an intersection of strongly closed subgroups is 
strongly closed, $Q\cap T$ is strongly closed in $\calf$ and hence in 
$\cale$.

By (ii), the inclusion of $T$ into $S$ induces an isomorphism 
$\alpha\:T/(Q\cap T)\xto{~\cong~}S/Q$. Then $\9\alpha(\cale/(Q\cap 
T))\le\calf/Q$ as fusion systems over $S/Q$, and we will show that they are 
equal.

Assume $\varphi\in\homf(P,R)$ for some $P,R\le S$ containing $Q$, and set 
$\psi=\varphi|_{P\cap T}$ as a morphism from $P\cap T$ to $R\cap T$. Then 
$\psi\in\Hom_\cale(P\cap T,R\cap T)$ by (i), and $\9\alpha(\psi/(Q\cap 
T))=\varphi/Q$ as homomorphisms from $P/Q$ to $R/Q$. Thus 
$\varphi/Q\in\Mor(\9\alpha(\cale/(Q\cap T)))$. 
Since $\varphi/Q\in\Mor(\calf/Q)$ was arbitrary, this proves that 
$\calf/Q\le\9\alpha(C_\calf(Q)/Z(Q))$.
\end{proof}

\bigskip

\section{Automorphism groups and tameness} 
\label{s:aut}

The main aim of this section is to introduce the concept of tameness for 
fusion systems. This was originally defined in \cite{AOV1} and it is one of 
the main subjects of this article. 

\smallskip

\subsection{Automorphisms of fusion and linking systems} 
Before defining tameness, we must define automorphism and outer automorphism 
groups of fusion and linking systems.

\begin{Defi}\label{d:autF} 
Fix a saturated fusion system $\calf$ over a finite $p$-group $S$. Then 
	\begin{enuma}
	\item $\Aut(\calf) = \{ \alpha \in \Aut(S) \mid 
	\9\alpha\calf = \calf \}$: the group of automorphisms of $S$ that 
	send $\calf$ to itself; 
	\item $\Out(\calf)=\Aut(\calf)/\Aut_{\calf}(S)$ is the group of outer automorphisms 
		of $\calf$; and 
	\item for each $\alpha\in\Aut(\calf)$, we let 
	$\8\alpha\:\calf\too\calf$ denote the functor that sends an object 
	$P$ to $\alpha(P)$ and a morphism $\varphi\in\Hom_{\calf}(P,Q)$ to 
	$\alpha \varphi \alpha^{-1}\in \Hom_{\calf}(\alpha(P),\alpha(Q))$.
	\end{enuma}
\end{Defi}

Now that we have defined automorphisms, we can define characteristic subsystems:

\begin{Defi} \label{d:characteristic}
Fix a saturated fusion system $\calf$ over a finite $p$-group $S$. A fusion 
subsystem $\cale\leq\calf$ over $T\nsg S$ is \emph{characteristic} if 
$\cale$ is normal in $\calf$ and $\8\alpha(\cale)=\cale$ for all 
$\alpha\in\Aut(\calf)$. Likewise, a subgroup $P$ of $S$ is characteristic 
in $\calf$ if $P\nsg\calf$ and $\alpha(P)=P$ for all 
$\alpha\in\Aut(\calf)$; equivalently, if $\calf_P(P)$ is a characteristic 
subsystem of $\calf$.
\end{Defi}

For example, when $\calf$ is a saturated fusion system over a finite 
$p$-group $S$, then the subsystems  $O^p(\calf)$ and $O^{p'}(\calf)$ 
(see Definition \ref{d:reduced}(b,c)) and the subgroups $O_p(\calf)$ 
and $Z(\calf)$ are all characteristic in  $\calf$.

\begin{Lem} \label{l:E<|F}
Let $\cale\nsg\calf$ be a normal pair of fusion systems over finite 
$p$-groups $T\nsg S$. Then
\begin{enuma} 

\item if $\cald\nsg\cale$ is characteristic in $\cale$, then 
$\cald\nsg\calf$; and

\item $O_p(\cale) \leq O_p(\calf)$. 
\end{enuma}
\end{Lem}

\begin{proof} Point (a) is shown in \cite[7.4]{A-gfit}. Since $O_p(\cale)$ 
	is characteristic in $\cale$, $O_p(\cale)\nsg\calf$ by (a), and so 
	$O_p(\cale) \leq O_p(\calf)$, proving (b). 
\end{proof}

The following condition for a subnormal fusion system to be normal is due 
to Aschbacher.

\begin{Lem} [{\cite[7.4]{A-gfit}}] \label{l:D<E<F}
Let $\cald\nsg\cale\nsg\calf$ be saturated fusion systems over finite 
$p$-groups $U\nsg T\nsg S$ such that $c_\alpha(\cald)=\cald$ for 
each $\alpha\in\autf(T)$. Then $\cald\nsg\calf$. 
\end{Lem}

The following definitions of automorphism groups are taken from 
\cite[Definition 1.13 \& Lemma 1.14]{AOV1}. Recall that for each pair of 
objects $P\le Q$ in a linking system $(\call,\delta,\pi)$, we write 
$\iota_{P,Q}=\delta_{P,Q}(1)\in\Mor_\call(P,Q)$, and regard it as the 
``inclusion'' in $\call$ of $P$ into $Q$.

\begin{Defi} \label{d:aut(L)}
Let $\calf$ be a fusion system over a finite $p$-group $S$ and let 
$(\call,\delta,\pi)$  be an associated linking system. For each $P$ in 
$\call$, we call $\delta_P(P)\le\Aut_\call(P)$ the \emph{distinguished 
subgroup} of $\Aut_\call(P)$. 
\begin{enuma}

\item Let $\Aut(\call)$ be the group of automorphisms of the category 
$\call$  that send inclusions to inclusions and distinguished subgroups to 
distinguished subgroups. 
		
\item For $\gamma\in\Aut_\call(S)$, let $c_\gamma\in\Aut(\call)$ be the 
automorphism which sends an object $P$ to $c_\gamma(P)\defeq 
\pi(\gamma)(P)$, and sends $\psi\in\Mor_\call(P,Q)$ to $c_\gamma(\psi) 
\defeq \gamma|_{Q,c_\gamma(Q)} \circ \psi \circ (\gamma|_{P,c_\gamma 
(P)})^{-1}$. Set 
	\[ \Out(\call)
	=\Aut(\call)\big/\{c_\gamma\,|\,\gamma\in\Aut_\call(S)\}\,. \]
		
\end{enuma}
\end{Defi}

The notation in Definitions \ref{d:autF} and \ref{d:aut(L)} is 
slightly different from that used in \cite{AOV1} and \cite{AKO}, as 
described in the following table:

\begin{center}
		\renewcommand{\arraystretch}{1.3}
		\renewcommand{\arraycolsep}{3mm}
		\begin{tabular}{l|cccc}
%			\multicolumn{5}{c}{\textup{Comparison of notation used for automorphism
%			groups}} \\
			Notation used here & $\Aut(\calf)$ & $\Out(\calf)$ & $\Aut(\call)$ &
			$\Out(\call)$ \\ \hline
			Used in \cite{AOV1,AKO} & $\Aut(S,\calf)$ &
			$\Out(S,\calf)$ & $\Aut\typ^I(\call)$ & $\Out\typ(\call)$ 
		\end{tabular}
\end{center}

By \cite[Lemma 1.14]{AOV1}, the above definition of $\Out(\call)$  is 
equivalent to $\Out\typ(\call)$ in \cite{BLO2}, and by \cite[Lemma 
8.2]{BLO2}, both are equivalent to $\Out\typ(\call)$ in \cite{BLO1}.  So by 
\cite[Theorem 4.5(a)]{BLO1}, $\Out(\call_S^c(G))\cong\Out(BG\pcom)$: the 
group of homotopy classes of self homotopy equivalences of the space 
$BG\pcom$.

The next result shows how an automorphism of a linking system automatically 
preserves the structure functors. For use in the next section, we state 
this for certain full subcategories of a linking system that need not 
themselves be linking systems because their objects might not be closed 
under overgroups. (Compare with Proposition 6 in \cite{O-linkext}.)

For a group $G$, a set $\calh$ of subgroups of $G$, and $\beta\in\Aut(G)$ 
that permutes the members of $\calh$, let 
$\calt(\beta)\:\calt_\calh(G)\too\calt_\calh(G)$ denote the functor that 
sends $H\in\calh$ to $\beta(H)$, and sends $g\in T_G(H,K)$ to $\beta(g)\in 
T_G(\beta(H),\beta(K))$. 

\begin{Prop} \label{AutI}
Let $(\call,\delta,\pi)$ be a linking system associated to a fusion system 
$\calf$ over a finite $p$-group $S$, let $\call_0\subseteq\call$ be a 
full subcategory such that $\Ob(\call_0)\supseteq\calf^{cr}$, and let 
$\Aut(\call_0)$ be the group of automorphisms of the category $\call_0$ 
that send inclusions to inclusions and distinguished subgroups to 
distinguished subgroups. 
Fix $\alpha\in\Aut(\call_0)$, and let $\beta\in\Aut(S)$ be the unique 
automorphism such that 
$\alpha(\delta_S(g))=\delta_S(\beta(g))$ for all $g\in{}S$.  Then 
$\beta\in\Aut(\calf)$, $\alpha(P)=\beta(P)$ for each $P\in\Ob(\call_0)$, 
and the following diagram of functors 
	\begin{equation}\label{AutI-diag}
	\vcenter{\xymatrix@C=40pt@R=25pt{ 
	\calt_{\Ob(\call_0)}(S) \ar[r]^-{\delta} \ar[d]_{ \calt(\beta) }
	&  \call_0 \ar[r]^-{\pi} \ar[d]^\alpha &  \calf
	\ar[d]^{\8\beta}\\
	\calt_{\Ob(\call_0)}(S) \ar[r]^-{\delta}  
	& \call_0 \ar[r]^-{\pi}  & \calf }}
	\end{equation}
commutes. 
\end{Prop}

\begin{proof} Clearly, $\alpha(S)=S$, and hence $\alpha_S$ sends 
$\delta_S(S)\in\sylp{\Aut_\call(S)}$ to itself.  Thus $\beta$ is well 
defined.  Since $\alpha$ sends inclusions to inclusions, it commutes with 
restrictions.  So for $P,Q\in\Ob(\call_0)$ and $g\in T_S(P,Q)$ (the 
transporter set), we have 
	\beqq \alpha(\delta_{P,Q}(g))=\delta_{\alpha(P),\alpha(Q)}(\beta(g)) 
	\label{e:AutI-1} \eeqq
since $\alpha(\delta_S(g))=\delta_S(\beta(g))$. In particular, the 
left-hand square in \eqref{AutI-diag} commutes. 

When $Q=P$, \eqref{e:AutI-1} says that 
$\delta_{\alpha(P)}(\beta(P))=\alpha_P(\delta_P(P))$, and 
$\alpha_P(\delta_P(P))=\delta_{\alpha(P)}(\alpha(P))$ since $\alpha$ sends 
distinguished subgroups to distinguished subgroups. So $\alpha(P)=\beta(P)$ 
since $\delta_{\alpha(P)}$ is a monomorphism (Proposition \ref{L-prop}(a)).  

Fix $P,Q\in\Ob(\call)$ and $\psi\in\Mor_\call(P,Q)$, and set 
$\varphi=\pi(\psi)\in\homf(P,Q)$.  For each $g\in{}P$, consider the 
following three squares:
	\[ 
	\vcenter{\xymatrix@C=30pt{
	P \ar[r]^{\psi} \ar[d]^{\delta_P(g)} & Q 
	\ar[d]^{\delta_Q(\varphi(g))} \\
	P \ar[r]^{\psi} & Q }} 
	\quad 
	\vcenter{\xymatrix@C=40pt{
	\alpha(P) \ar[r]^{\alpha(\psi)} \ar[d]^{\delta_{\alpha(P)}(\beta(g))} 
	& \alpha(Q) \ar[d]^{\delta_{\alpha(Q)}(\beta(\varphi(g)))} \\
	\alpha(P) \ar[r]^{\alpha(\psi)} & \alpha(Q) } }
	\quad 
	\vcenter{\xymatrix@C=40pt{
	\alpha(P) \ar[r]^{\alpha(\psi)} \ar[d]^{\delta_{\alpha(P)}(\beta(g))} 
	& \alpha(Q) 
	\ar[d]^{\delta_{\alpha(Q)}(\pi(\alpha(\psi))(\beta(g)))} \\
	\alpha(P) \ar[r]^{\alpha(\psi)} & \alpha(Q)\rlap{\,.} } }  \]
The first and third of these squares commute by axiom (C) in Definition 
\ref{d:Linking}, and the second commutes since it is the image under 
$\alpha$ of the first. Since morphisms in $\call$ are epimorphisms and 
$\delta_{\alpha(Q)}$ is injective (Proposition \ref{L-prop}(a,b)), 
this implies $\beta(\varphi(g))=\pi(\alpha(\psi))(\beta(g))$. Thus 
$\pi(\alpha(\psi))=\beta\varphi\beta^{-1} = \8\beta(\pi(\psi))$, 
proving that the right-hand square in \eqref{AutI-diag} commutes.

In particular, since $\pi$ is surjective on morphism sets (axiom (A2) in 
Definition \ref{d:Linking}), 
$\beta\varphi\beta^{-1}\in\homf(\beta(P),\beta(Q))$ for each 
$P,Q\in\Ob(\call_0)$ and each $\varphi\in\homf(P,Q)$.  Since $\Ob(\call_0)$ 
includes all subgroups which are $\calf$-centric and $\calf$-radical, all 
morphisms in $\calf$ are composites of restrictions of morphisms between 
objects of $\call_0$ by Theorem \ref{t:AFT}. Hence 
$\9\beta\calf\le\calf$ with equality since $\calf$ is a finite category, 
and so $\beta\in\Aut(\calf)$.
\end{proof}

Proposition \ref{AutI} motivates the following definition.

\begin{Defi} \label{d:til.mu}
Let $(\call,\delta,\pi)$ be a linking system associated to a fusion system 
$\calf$ over a finite $p$-group $S$. Let 
$\til\mu_{\call}\:\Aut(\call)\too\Aut(\calf)$ denote the homomorphism that 
sends $\alpha\in\Aut(\call_0)$ to $\beta\in\Aut(\calf)$ such that diagram 
\eqref{AutI-diag} commutes. Let $\mu_\call\:\Out(\call)\too\Out(\calf)$ be 
the induced homomorphisms on the quotient groups.
\end{Defi}

That $\til\mu_\call$ is a homomorphism follows easily from its definition 
via diagram \eqref{AutI-diag}. For $\gamma\in\Aut_\call(S)$, we have 
$\til\mu_\call(c_\gamma)=\pi(\gamma)\in\autf(S)$ since $\pi$ is a 
functor, so $\mu_\call$ is well defined. 
 
\smallskip

\subsection{Tameness of fusion systems} 
We next define a homomorphism $\kappa_G$ that connects the 
automorphisms of a group to those of its linking system. We refer to 
\cite[\S\,2.2]{AOV1} for more details about $\kappa_G$ and the proof that 
it is well defined.

\begin{Defi}\label{d:kappaG}
	Let $G$ be a finite group and choose $S\in \sylp G$.
Let
	\begin{equation*}
	    \kappa_G\:\Out(G)\Right2{}\Out(\call_S^c(G))
	\end{equation*}
denote the homo\m\ that sends the class of 
$\alpha\in\Aut(G)$ such that $\alpha(S)=S$ to the class of the
automorphism of $\call_S^c(G)$ induced by $\alpha$.
\end{Defi}

In these terms, tameness can be defined as follows.

\begin{Defi} \label{d:tame} 
Let $\calf$ be a saturated fusion system over a finite $p$-group $S$. Then
\begin{enuma}
\item $\calf$ is \emph{tamely realized by} a finite group $G$ 
%if $S\in\sylp{G}$, $\calf =\calf_S(G)$, 
if $\calf\cong\calf_{S^*}(G)$ for some $S^*\in\sylp{G}$
and the homomorphism
$\kappa_G \colon 
\Out(G)\longrightarrow\Out(\call_{S^*}^c(G))$  is split surjective; 
and
\item $\calf$ is \emph{tame} if it is tamely realized by some finite group.
\end{enuma}
\end{Defi}

\smallskip

\subsection{Centric fusion and linking subsystems}

Some of the results in later sections need the hypothesis that a 
certain fusion or linking subsystem be \emph{centric}, which we now define.

\begin{Defi} \label{d:centric.f.s.}
Let $\cale\nsg\calf$ be a normal pair of saturated fusion systems over 
finite $p$-groups $T\nsg S$.
\begin{enuma}

\item Let $C_S(\cale)$ denote the unique largest subgroup $X\le C_S(T)$
such that $C_\calf(X)\ge\cale$. Such a largest subgroup exists by
\cite[6.7]{A-gfit}, or (via a different proof) by \cite[Theorem
1(a)]{Henke}.

\item The subsystem $\cale$ is \emph{centric} in $\calf$ if $C_S(\cale)\le 
T$; i.e., if there is no $x\in C_S(T)\sminus T$ such that 
$C_\calf(x)\ge\cale$.

\item If $\calm\nsg\call$ are linking systems associated to 
$\cale\nsg\calf$, set 
	\[ C_{\Aut_\call(T)}(\calm) = \bigl\{ \gamma\in\Aut_\call(T) 
	\,\big|\, c_\gamma=\Id_\calm \bigr\} \]
where $c_\gamma$ is a well defined element of $\Aut(\calm)$ by Definition 
\ref{d:Lnormal}(b). 

\item If $\calm\nsg\call$ are linking systems associated to 
$\cale\nsg\calf$, then $\calm$ is \emph{centric} in $\call$ if 
$C_{\Aut_\call(T)}(\calm) \le \Aut_\calm(T)$; equivalently, if 
$c_\psi\ne\Id_\calm$ for each $\psi\in\Aut_\call(T)\sminus\Aut_\calm(T)$. 

\end{enuma}
\end{Defi}

For pairs of linking systems, this is the definition used in \cite{AOV1} 
(Definition 1.27). The term ``centric fusion subsystem'' was not used in 
\cite{O-red}, but the condition in Definition \ref{d:centric.f.s.}(b) 
appears in Proposition 2.1 and Theorem 2.3 of that paper (and the term is 
used in \cite{O-red-corr}).

In the next lemma, we look at the relation between normal centric 
fusion subsystems and normal centric linking subsystems.

\begin{Lem} \label{centric<=>centric}
Let $\cale\nsg\calf$ be a normal pair of saturated fusion systems over 
finite $p$-groups
$T\nsg S$ with associated linking systems $\calm\nsg\call$, and set 
	\[ C_S(\calm) = \bigl\{ x\in S \,\big|\, 
	c_{\delta_T(x)}=\Id_\calm \bigr\} \nsg S. \]
Then 
\begin{enuma} 

\item $C_{\Aut_\call(T)}(\calm) = \delta_T(C_S(\calm))$ and 
$Z(\cale)Z(\calf)\le C_S(\calm)\le C_S(\cale)$; 

\item $\calm$ is centric in $\call$ if and only if $C_S(\calm)\le T$, and 
this holds if $\cale$ is centric in $\calf$; and 

\item the conjugation action of $\Aut_\call(T)$ on 
$C_S(\calm)\cong\delta_T(C_S(\calm))\nsg\Aut_\call(T)$ induces an action of 
the quotient group $\call/\calm=\Aut_\call(T)/\Aut_\calm(T)$ on 
$C_S(\calm)$, and $C_{C_S(\calm)}(\call/\calm)=Z(\calf)$.

\end{enuma}
\end{Lem}

\begin{proof} Throughout the proof, ``axiom (--)'' always refers to one of 
the axioms in the definition of a linking system (Definition 
\ref{d:Linking}). 

We first claim that 
	\beqq \forall x\in S, \qquad \delta_T(x)\in\Aut_\calm(T) \iff x\in T .
	\label{e:centric-1} \eeqq 
The implication ``$\Longleftarrow$'' is clear. To see the converse, fix 
$x\in S$ such that $\delta_T(x)\in\Aut_\calm(T)$. Then 
$c_x\in\Aut_\cale(T)$, and $c_x\in\Inn(T)\in\sylp{\Aut_\cale(T)}$ since it 
has $p$-power order. Thus there is $t\in T$ such that $xt^{-1}\in C_S(T)$, 
$\delta_T(xt^{-1})=\delta_T(x)\delta_T(t)^{-1}\in\Aut_\calm(T)$, and so 
$xt^{-1}\in Z(T)$ by axiom (A2) applied to $\calm$ and since 
$\delta_T$ 
is injective (Proposition \ref{L-prop}(a)). Hence $x\in T$.

We next claim that 
	\beqq C_S(\cale)\cap T = Z(\cale). \label{e:centric-1x} \eeqq
The inclusion $Z(\cale)\le C_S(\cale)\cap T$ is immediate from the 
definitions. If $x\in C_S(\cale)\cap T\le Z(T)$, then since $\cale\le 
C_\calf(x)$, each $\varphi\in\Hom_\cale(\gen{x},S)$ extends to a 
morphism in $\calf$ that sends $x$ to itself, and hence $\varphi(x)=x$. 
Thus $x^\cale=\{x\}$, so $x\in Z(\cale)$ by \cite[Lemma I.4.2]{AKO}, 
finishing the proof of \eqref{e:centric-1x}.

\smallskip

\noindent\textbf{(a) } Fix $\alpha\in C_{\Aut_\call(T)}(\calm)$. By axiom 
(C) for the linking system $\call$, for all 
$g\in T$, $\alpha\delta_T(g)\alpha^{-1}=\delta_T(\pi(\alpha)(g))$, so 
$\delta_T(g)=\delta_T(\pi(\alpha)(g))$ since $c_\alpha=\Id_\calm$, and 
$g=\pi(\alpha)(g)$ since $\delta_T$ is injective (see Proposition 
\ref{L-prop}(a)). So $\pi(\alpha)=\Id_T$, and $\alpha=\delta_T(x)$ for 
some $x\in C_S(T)$ by axiom (A2). Then $x\in C_S(\calm)$ by definition, and 
thus $C_{\Aut_\call(T)}(\calm)=\delta_T(C_S(\calm))$.

If $x\in Z(\cale)= C_S(\cale)\cap T$, then for each $P,Q\le T$ and 
$\psi\in\Mor_\calm(P,Q)$, $\pi(\psi)$ extends to some 
$\4\varphi\in\Hom_\cale(P\gen{x},Q\gen{x})$, and $\4\varphi=\pi(\5\psi)$ 
for some $\5\psi\in\Mor_\calm(P\gen{x},Q\gen{x})$. Set 
$\psi'=\5\psi|_{P,Q}\in\Mor_\calm(P,Q)$. Then 
$\psi'\delta_{P}(x)=\delta_{Q}(x)\psi'$ since 
$\5\psi\delta_{P\gen{x}}(x)=\delta_{Q\gen{x}}(x)\5\psi$ by axiom (C). Also, 
$\pi(\psi')=\pi(\psi)$, so if $P$ is fully centralized in $\cale$, then 
$\psi'=\psi\delta_P(y)$ for some $y\in C_T(P)$. Then $[x,y]=1$ since 
$x\in Z(T)$, so $c_{\delta_T(x)}(\delta_P(y))=\delta_P(y)$, and thus 
$c_{\delta_T(x)}(\psi)=\psi$. This holds for all $\psi\in\Mor(\calm)$ whose 
domain is fully centralized, and hence for all morphisms in $\calm$. So 
$x\in C_T(\calm)\le C_S(\calm)$. 

Thus $Z(\cale)\le C_S(\calm)$. By a similar argument but working in $\call$ 
and $\calf$ instead of $\calm$ and $\cale$, we also have $Z(\calf)\le 
C_S(\call)\le C_S(\calm)$, and so $Z(\cale)Z(\calf)\le C_S(\calm)$.

It remains to show that $C_S(\calm)\le C_S(\cale)$. Fix $x\in C_S(\calm)$; 
we must show that $\cale\le C_\calf(x)$. Fix $P,Q\le T$ and 
$\varphi\in\Hom_\cale(P,Q)$, and choose $\psi\in\Mor_\calm(P,Q)$ such that 
$\pi(\psi)=\varphi$. Set $\4P=P\gen{x}$ and $\4Q=Q\gen{x}$. Then 
$\psi\delta_P(x)=\delta_Q(x)\psi$ since 
$c_{\delta_T(x)}=\Id_\calm$, and 
$\psi\delta_P(g)=\delta_Q(\varphi(g))\psi$ for all $g\in P$ by axiom (C) 
applied to $\calm$. So for each $g\in\4P$, there is $h\in\4Q$ such 
that $\psi\delta_P(g)=\delta_{Q}(h)\psi$ in $\Mor_\call(P,Q)$, and by 
Proposition \ref{L-prop}(d), $\psi$ extends to a unique morphism 
$\4\psi\in\Mor_\call(\4P,\4Q)$. Set 
$\4\varphi=\pi(\4\psi)\in\homf(\4P,\4Q)$. By axiom (C) again (but applied 
to $\call$) we have 
$\4\psi\delta_{\4P}(x)=\delta_{\4Q}(\4\varphi(x))\4\psi$, and after 
restriction to $P$ and $Q$ this gives 
$\psi\delta_P(x)=\delta_Q(\4\varphi(x))\psi$. Hence 
$\delta_Q(\4\varphi(x))\psi=\delta_Q(x)\psi$, so 
$\delta_Q(\4\varphi(x))=\delta_Q(x)$ since $\psi$ is an epimorphism in the 
categorical sense (see Proposition \ref{L-prop}(b)), and 
$\4\varphi(x)=x$ by the injectivity of $\delta_Q$. Thus each 
morphism in $\cale$ extends to a morphism in $\calf$ that sends $x$ to 
itself, and hence $C_\calf(x)\ge\cale$ and $x\in C_S(\cale)$. 

\smallskip

\noindent\textbf{(b) } By (a) and Definition \ref{d:centric.f.s.}(c), 
$\calm$ is centric in $\call$ if and only if 
$\delta_T(C_S(\calm))\le\Aut_\calm(T)$, and this holds exactly when 
$C_S(\calm)\le T$ by \eqref{e:centric-1}. Since $C_S(\calm)\le C_S(\cale)$ 
by (a), this is the case whenever $\cale$ is centric in $\calf$.

\smallskip

\noindent\textbf{(c) } By (a), $\delta_T(C_S(\calm))$ is the kernel of the 
homomorphism from $\Aut_\call(T)$ to $\Aut(\calm)$ induced by conjugation. 
So $\delta_T(C_S(\calm))\nsg\Aut_\call(T)$, and $\Aut_\call(T)$ acts by 
conjugation on $C_S(\calm)\cong\delta_T(C_S(\calm))$. Since this subgroup 
centralizes $\calm$ by definition, $\Aut_\calm(T)$ acts trivially, and 
hence the action of $\Aut_\call(T)$ factors through an action of the 
quotient group $\call/\calm$.

By Definitions \ref{d:subgroups}(g) and \ref{d:centric.f.s.}(a), 
$Z(\calf)=C_S(\calf)$. So by (a) applied when $\calm=\call$,
	\beqq C_{\Aut_\call(S)}(\call) = \delta_S(Z(\calf)). 
	\label{e:d(Z(F))} \eeqq
Thus for each $x\in Z(\calf)$, $\delta_S(x)$ acts trivially on $\call$ and 
hence $\delta_T(x)$ acts trivially on $\Aut_\call(T)$. So $Z(\calf)\le 
C_{C_S(\calm)}(\Aut_\call(T)) = C_{C_S(\calm)}(\call/\calm)$, and it 
remains to prove the opposite inclusion.

Fix $x\in C_S(\calm)$ such that $\delta_T(x)\in Z(\Aut_\call(T))$. In 
particular, $\delta_T(x)\in Z(\delta_T(S))$, so $x\in Z(S)$ by 
the injectivity of $\delta_T$ (Proposition \ref{L-prop}(a)), and 
$c_{\delta_T(x)}(P)=\9xP=P$ for each $P\in\Ob(\call)$. We must show 
that $\delta_T(x)$ acts trivially on $\call$. Fix $P,Q\le S$ and 
$\psi\in\Mor_\call(P,Q)$, and set $P_0=P\cap T$, 
$Q_0=Q\cap T$, and $\psi_0=\psi|_{P,Q}$ (see Proposition 
\ref{L-prop}(c)). By the Frattini condition on a normal linking subsystem 
(Lemma \ref{l:Lnormal}), $\psi_0$ is the composite of the 
restriction of a morphism $\gamma\in\Aut_\call(T)$ followed by some 
$\chi\in\Mor(\calm)$. Since $\delta_T(x)$ commutes with $\gamma$ by 
assumption and commutes with $\chi$ by (a) ($Z(\calf)\le C_S(\calm)$), 
we have 
	\[ (\psi\delta_P(x))|_{P_0,Q_0} = \psi_0\delta_{P_0}(x) = 
	\delta_{Q_0}(x)\psi_0 = (\delta_Q(x)\psi)|_{P_0,Q_0}. \]
Then $\psi\delta_P(x)=\delta_Q(x)\psi$ by 
the uniqueness of extensions in a linking system (Proposition 
\ref{L-prop}(d)), and hence $c_{\delta_T(x)}(\psi)=\psi$. 

Thus $c_{\delta_S(x)}$ is the identity on all objects and morphisms 
in $\call$. So $x\in Z(\calf)$ by \eqref{e:d(Z(F))} and the injectivity of 
$\delta_S$ (Proposition \ref{L-prop}(a)).
\end{proof}

The following consequence of Lemma \ref{centric<=>centric} will be needed 
in Section \ref{s:comp}. 

\begin{Lem} \label{l:Ker(kappa)}
Let $H\nsg G$ be finite groups, choose $S\in\sylp{G}$, and set $T=S\cap H$. 
Set $\calf=\calf_S(G)$ and $\cale=\calf_T(H)\nsg\calf$. If $\Ker(\kappa_H)$ 
has order prime to $p$, then $Z(\calf)\le Z(\cale) C_S(H)$. 
\end{Lem}

\begin{proof} Set $G_0=SH$ and $\calf_0=\calf_S(G_0)$, and set 
$\calh=\{P\le S\,|\,P\cap T\in\cale^c\}$. For each $P\in\calh$, 
$C_H(P\cap H)=Z(P\cap H)\times O_{p'}(C_H(P\cap H))$ since $P\cap 
H\in\cale^c$, so $C_H(P)=(Z(P)\cap H)\times O_{p'}(C_H(P))$, and this has 
$p$-power index in $C_{G_0}(P)$. Thus $O^p(C_{G_0}(P))$ has order prime to 
$p$, so $P$ is $\calf_0$-quasicentric by \cite[Lemma III.4.6(e)]{AKO}. 

Thus $\calh\subseteq\calf_0^q$. Set 
$\call_0=\call_S^\calh(G_0)$ (see \cite[p. 146]{AKO}), and 
$\calm=\call_T^c(H)$. Then $\calm\nsg\call_0$ is a normal pair of linking 
systems associated to $\cale\nsg\calf_0$, so by Lemma 
\ref{centric<=>centric}(a), 
	\beqq Z(\calf_0) \le C_S(\calm) = \{ x\in S\,|\, 
	c_{\delta_T(x)}=\Id_{\calm} \}. \label{e:kappa} \eeqq

Fix $x\in Z(\calf)\le Z(\calf_0)$. By \eqref{e:kappa}, 
$[c_x]\in\Ker(\kappa_H)$, and so $c_x\in\Inn(H)$ since $\Ker(\kappa_H)$ has 
order prime to $p$. Thus $c_x$ is conjugation by some element $y\in 
C_H(T)=Z(T)\times O_{p'}(C_H(T))$, and since $c_x$ has $p$-power order, we can 
assume $y\in Z(T)$. Then $c_y$ induces the identity on $\calm$ since $c_x$ 
does by \eqref{e:kappa}, so $y\in Z(\cale)$ by the exact sequence in 
\cite[Lemma 1.14(a)]{AOV1}. Also, $y^{-1}x\in C_S(H)$, and so $x\in 
Z(\cale) C_S(H)$. 
\end{proof}

\bigskip

\section{Products of fusion and linking systems}
\label{s:prod}

In this section, we first define centric linking systems 
associated to products of two or more fusion systems (Lemma 
\ref{prod.link1}).  This is followed by a description of the group of 
automorphisms of such a product linking systems that leave the factors 
invariant up to permutation, as well as conditions on the fusion 
systems that guarantee that these are the only automorphisms of the linking 
system (Proposition \ref{p:prod.tame}(a,b)). As a consequence, we show 
that a product of tame fusion systems that satisfy these same conditions 
is always tame (Proposition \ref{p:prod.tame}(c)).

Recall that if $\calf_1$ and $\calf_2$ are fusion systems over finite 
$p$-groups $S_1$ and $S_2$, then $\calf_1\times\calf_2$ is the fusion 
system over $S_1\times S_2$ generated by all morphisms 
$\varphi_1\times\varphi_2\in\Hom(P_1\times P_2,Q_1\times Q_2)$ for 
$\varphi_i\in\Hom_{\calf_i}(P_i,Q_i)$ ($i=1,2$). See \cite[Definition 
I.6.5]{AKO} for more details. When $\calf_1$ and $\calf_2$ are both 
saturated, then so is $\calf_1\times\calf_2$ \cite[Theorem 
I.6.6]{AKO}. 

The following notation and hypotheses will be used throughout the 
section. 

\begin{Hyp} \label{hyp:prod.fus}
Let $\calf_1,\dots,\calf_k$ be saturated fusion systems over finite 
$p$-groups $S_1,\dots,S_k$ (some $k\ge2$), and set $S=\xxx{S}$ and 
$\calf=\xxx{\calf}$. For each $i$, let $\pr_i\:S\too S_i$ be the 
projection. For each $P\le S$, we write $P_i=\pr_i(P)$ (for $1\le i\le k$) 
and $\5P=\xxx{P}\le S$. Thus $\5P\ge P$ for each $P$.
\end{Hyp}

We first check which subgroups are centric in a product of fusion systems.

\begin{Lem} \label{prod.fus}
Assume Hypotheses \ref{hyp:prod.fus}. Then a subgroup $P\le S$ 
is $\calf$-centric if and only if $P_i$ is $\calf$-centric for 
all $i$ and $Z(P_1)\times\dots \times Z(P_k)\leq P$. 
\end{Lem}

\begin{proof} For each $P\le S$, we have 
$C_S(P)=C_{S_1}(P_1)\times\cdots\times C_{S_k}(P_k)=C_S(\5P)$. Hence 
\begin{itemize} 
\item $P$ is fully centralized in $\calf$ if and only if $P_i$ is fully 
centralized in $\calf_i$ for each $i$, and 
\item $C_S(P)\le P$ if and only if $C_{S_i}(P_i)\le P_i$ for each $i$ 
and $Z(\5P)\le P$.
\end{itemize}
The result now follows since $P$ is $\calf$-centric if and only if $P$ is 
fully centralized in $\calf$ and $C_S(P)\le P$.
\end{proof}

Note also that under Hypotheses \ref{hyp:prod.fus}, if $P\le S$ is 
$\calf$-centric and $\calf$-radical, then $P=\5P=\xxx{P}$ 
(see, e.g., \cite[Lemma 3.1]{AOV1}). But that will not be needed here.

The following easy consequence of the Krull-Remak-Schmidt theorem will be 
needed. Recall that a group is indecomposable if it is not 
the direct product of two of its proper subgroups.

\begin{Prop} \label{Krull-Schmidt}
Assume $G_1,\dots,G_k$ are finite, indecomposable groups, and set 
$G=\xxx{G}$. Then the following hold for each $\alpha\in\Aut(G)$.
\begin{enuma} 
\item There is $\sigma\in\Sigma_k$ such that $\alpha(G_iZ(G))=G_{\sigma(i)}Z(G)$ 
for each $1\le i\le k$. 
\item If $\bigl(|Z(G)|,|G/[G,G]|\bigr)=1$, then there is 
$\sigma\in\Sigma_k$ such that $\alpha(G_i)=G_{\sigma(i)}$ for each $i$.
\end{enuma} 
\end{Prop}

\begin{proof} An automorphism $\beta\in\Aut(G)$ is \emph{normal} if $\beta$ 
commutes with all inner automorphisms of $G$; equivalently, if 
$[\beta,G]\le Z(G)$. For such $\beta$, one can define a homomorphism $\delta\:G\too Z(G)$ by setting $\delta(g)= \beta(g)g^{-1}$, and $\delta$ factors through
$G/[G,G]$. In particular, if 
$\bigl(|Z(G)|,|G/[G,G]|\bigr)=1$, then the only normal automorphism of $G$ 
is the identity.

By the Krull-Remak-Schmidt theorem in the form stated in \cite[Theorem 
2.4.8]{Sz1} (and applied with $\Omega=1$ or $\Omega=\Inn(G)$), any two 
direct product decompositions of $G$ with indecomposable factors have the 
same number of factors, and there is always a normal automorphism of $G$ 
that sends one to the other up to a permutation of the factors . 
Points (a) and (b) follow immediately from this, applied to the 
decompositions of $G$ as the product of the $G_i$ and of the $\alpha(G_i)$. 
\end{proof}

One immediate consequence of Proposition \ref{Krull-Schmidt} is the 
following description of $\Out(G)$ when $G$ is a product of simple groups. 

\begin{Prop} \label{p:Out(G1x...xGk)}
Assume $G=\xxx{G}$, where $G_1,\dots,G_k$ are finite indecomposable 
groups and $\bigl(|Z(G)|,|G/[G,G]|\bigr)=1$. 
Let $\Gamma$ be the group of all $\gamma\in\Sigma_k$ such that 
$G_{\gamma(i)}\cong G_i$ for each $i$. Then there is an isomorphism 
	\[ \Phi_G\:\bigl(\Aut(G_1)\times\cdots\times\Aut(G_k)\bigr) 
	\rtimes\Gamma \Right5{\cong} \Aut(G) \] 
with the property that $\Phi_G(\beta_1,\dots,\beta_k)=\xxx{\beta}$ for 
each $k$-tuple of automorphisms $\beta_i\in\Aut(G_i)$. Also, $\Phi_G$ sends 
$\prod_{i=1}^k\Inn(G_i)$ isomorphically to $\Inn(G)$, and hence induces an 
isomorphism
	\[ \4\Phi_G\:\bigl(\Out(G_1)\times\cdots\times\Out(G_k)\bigr) 
	\rtimes\Gamma \Right5{\cong} \Out(G). \] 

To define $\Phi_G$ more precisely, fix isomorphisms $\lambda_{ij}\:G_i\too 
G_j$ for each $1\le i<j\le k$ such that $G_i\cong G_j$, chosen so that 
$\lambda_{i\ell}=\lambda_{j\ell}\lambda_{ij}$ whenever $G_i\cong G_j\cong 
G_\ell$, and set $\lambda_{ji}=\lambda_{ij}^{-1}$. Also, set 
$\lambda_{ii}=\Id_{G_i}$ for each $i$. Then $\Phi_G$ can be chosen so 
that for each $\gamma\in\Gamma$, 
	\[ \Phi_G(\gamma)(g_1,\dots,g_k) = 
	\bigl(\lambda_{\gamma^{-1}(1),1}(g_{\gamma^{-1}(1)}),\dots,
	\lambda_{\gamma^{-1}(k),k}(g_{\gamma^{-1}(k)})\bigr). \]
\end{Prop}

\begin{proof} Without loss of generality, we can assume that $G_i=G_j$ and 
$\lambda_{ij}=\Id_{G_i}$ for each $i,j$ such that $G_i\cong G_j$. Thus 
$\Phi_G(\gamma)(g_1,\dots,g_k)= 
(g_{\gamma^{-1}(1)},\dots,g_{\gamma^{-1}(k)})$ for each $\gamma\in\Gamma$. 
Then $\Phi_G$ is clearly an injective homomorphism, and it factors 
through a homomorphism $\4\Phi_G$ as above since
$\Inn(G)=\Phi_G\bigl(\prod_{i=1}^k\Inn(G_i)\bigr)$. Each 
automorphism of $G$ permutes the factors by Proposition \ref{Krull-Schmidt}(b), 
and hence $\Phi_G$ is surjective. 
\end{proof}

In the next proposition, we describe one way to construct linking systems 
associated to products of fusion systems.

\begin{Prop} \label{prod.link1}
Assume Hypotheses \ref{hyp:prod.fus}. For each $1\le i\le k$, let 
$\call_i$ be a centric linking system associated to 
$\calf_i$, with structure functors $\delta_i$ and $\pi_i$. 
Let $\call$ be the category whose objects are the $\calf$-centric subgroups 
of $S$, and where for each $P,Q\in\Ob(\call)$,
	\[ \Mor_\call(P,Q) = \bigl\{ (\varphi_1,\dots,\varphi_k)\in
	\textstyle\prod_{i=1}^k\Mor_{\call_i}(P_i,Q_i) \,\big|\, 
	(\pi_1(\varphi_1),\dots,\pi_k(\varphi_k))(P)\le Q \bigr\}. \]
Define 
	\[ \calt_{\Ob(\call)}(S) \Right4{\delta} \call \Right4{\pi} \calf 
	\]
by setting, for all $P,Q\in\calf^c=\Ob(\call)$, 
	\begin{align*} 
	\delta_{P,Q}(g) &= \bigl((\delta_1)_{P_1,Q_1}(g_1),\dots, 
	(\delta_k)_{P_k,Q_k}(g_k) \bigr) && 
	\textup{all $g=(g_1,\dots,g_k)\in T_S(P,Q)$} \\
	\pi_{P,Q}(\varphi) &= \bigl(
	\pi_1(\varphi_1),\dots,\pi_k(\varphi_k)\bigr) && \textup{all 
	$\varphi=(\varphi_1,\dots,\varphi_k)\in\Mor_\call(P,Q)$.} 
	\end{align*}
Then the following hold:
\begin{enuma} 

\item The functors $\delta$ and $\pi$ make $\call$ into a centric linking 
system associated to $\calf$.

\item Let $\xxx\call$ be the product of the \emph{categories} 
$\call_i$. Define $\xi_\call\:\xxx{\call}\too\call$ by setting 
	\[ \xi_\call(P_1,\dots,P_k) = \xxx{P} 
	\qquad\textup{and}\qquad
	\xi_\call(\varphi_1,\dots,\varphi_k) = 
	(\varphi_1,\dots,\varphi_k). \]
Then $\xi_\call$ is an isomorphism of categories from 
$\xxx{\call}$ to the full subcategory 
$\5\call\subseteq\call$ whose objects are those $P\in\calf^c$ such that 
$P=\5P$; equivalently, the products $\xxx{P}$ 
for $P_i\in\calf_i^c$. Also, the following square commutes
	\[ \xymatrix@C=40pt@R=25pt{ 	
	\calt_{\Ob(\call_1)}(S_1) \times\cdots\times 
	\calt_{\Ob(\call_k)}(S_k) 
	\ar[r]^-{\eta}_-{\cong} \ar[d]_{\xxx{\delta}} &
	\calt_{\Ob(\5\call)}(S) \ar[d]^{\delta} \\
	\xxx{\call} \ar[r]^-{\xi_\call} & \call 
	} \] 
where $\eta$ is the natural isomorphism that sends $(P_1,\dots,P_k)$ to 
$\Xxx{P}$.

\item Let $\rho_i\:\call_i\too\call$ be the functor that sends 
$P_i\in\Ob(\call_i)$ to its product with the $S_j$ for $j\ne i$, and sends 
$\varphi_i\in\Mor(\call_i)$ to its product with $\Id_{S_j}$ for $j\ne i$. 
Then $\rho_i$ is injective on objects and on morphism sets. 
If $\alpha\in\Aut(\call)$ is such that 
$\alpha_S(\delta_S(S_i))=\delta_S(S_i)$ for each $i$, then 
$\alpha(\rho_i(\call_i))=\rho_i(\call_i)$ for each $i$.

\end{enuma}
\end{Prop}

\begin{proof} By Lemma \ref{prod.fus}, for each $P\in\Ob(\call)=\calf^c$, 
$\pr_i(P)$ is $\calf_i$-centric for each $i$. So the definitions of 
$\Mor_\call(P,Q)$ and $\delta$ make sense.

\smallskip

\noindent\textbf{(a) } Axiom (A1) is clear. Fix $P,Q\in\Ob(\call)$ and 
set $P_i=\pr_i(P)$ and $Q_i=\pr_i(Q)$; then 
$C_S(P)=Z(P)=Z(P_1)\times\cdots\times Z(P_k)$. So by axiom (A2) for the 
$\call_i$, for $\varphi,\varphi'\in\Mor_\call(P,Q)$, 
$\pi(\varphi)=\pi(\varphi')$ if and only if 
$\varphi'=\varphi\circ\delta_P(z)$ for some $z\in Z(P)$. 
For each $P,Q\in\Ob(\call)$, each $\varphi\in\homf(P,Q)$ is the restriction 
of some morphism $\prod_{i=1}^k\varphi_i\in\homf(\5P,\5Q)$ 
(see \cite[Theorem I.6.6]{AKO}), and 
hence the surjectivity of $\pi$ on morphism sets follows from that of the 
$\pi_i$.
The rest of (A2) 
(the effect of $\varphi$ and $\pi$ on objects) is clear. Likewise, axioms 
(B) and (C) for $\call$ follow immediately from the corresponding axioms 
for the $\call_i$. Thus $\call$ is a centric linking system with structure 
functors $\delta$ and $\pi$.

\smallskip

\noindent\textbf{(b) } Both statements ($\xi_\call$ is an isomorphism of 
categories and the diagram commutes) are immediate from the definitions 
and since $\xxx{P}$ is $\calf$-centric if $P_i$ is 
$\calf_i$-centric for each $i$ (Lemma \ref{prod.fus}).

\smallskip

\noindent\textbf{(c) } Let $\alpha\in\Aut(\call)$ be such that 
$\alpha(\delta_S(S_i))=\delta_S(S_i)\le\Aut_\call(S)$ for each $i$. We must 
show that $\alpha(\rho_i(\call_i))=\rho_i(\call_i)$ for each $i$. Fix some 
$i$, let $\4S_i\le S$ be the product of the $S_j$ for $j\ne i$, and 
identify $S_i\times\4S_i$ with $S$. Thus $\rho_i(P_i)=P_i\times\4S_i$ and 
$\rho_i(\varphi_i)=(\varphi_i\times\Id_{\4S_i})$ for $P_i\in\Ob(\call_i)$ 
and $\varphi_i\in\Mor(\call_i)$. 

Set $\beta=\til\mu_\call(\alpha)\in\Aut(\calf)$ (see Definition 
\ref{d:til.mu}). By Proposition \ref{AutI}, $\alpha(P)=\beta(P)$ for 
$P\in\Ob(\call)$, and $\pi\circ\alpha=\8\beta\circ\pi$ as functors from 
$\call$ to $\calf$. By assumption, $\beta(S_i)=S_i$ and 
$\beta(\4S_i)=\4S_i$. So for each $P_i\in\Ob(\call_i)$, we have 
$\alpha(\rho_i(P_i))=\beta(P_i\times\4S_i)=P_i^*\times\4S_i$ for some 
$P_i^*\le S_i$. Thus $\alpha$ permutes the objects in $\rho_i(\call_i)$.

Now fix a morphism $\varphi_i\in\Mor_{\call_i}(P_i,Q_i)$. Since 
$\pi\circ\alpha=\8\beta\circ\pi$, we have 
	\[ \pi(\alpha(\rho_i(\varphi_i))) = 
	\8\beta(\pi(\varphi_i\times\Id_{\4S_i})) = 
	\8\beta(\pi(\varphi_i)\times\Id_{\4S_i}) = \pi(\psi)\times\Id_{\4S_i} 
	= \pi(\rho_i(\psi)) \]
for some $\psi\in\Mor(\call_i)$. So by axiom (A2) in Definition 
\ref{d:Linking}, $\alpha(\rho_i(\varphi_i))=\rho_i(\psi)\delta_S(z,z')$ 
for some $z\in Z(P_i)$ and $z'\in Z(\4S_i)$. Since $z'$ has $p$-power 
order, this shows that $z'=1$ and 
$\alpha(\rho_i(\varphi_i))\in\Mor(\rho_i(\call_i))$ if $\varphi_i$ is an 
automorphism of order prime to $p$. Also, 
	\[ 
	\alpha(\rho_i(\delta_{S_i}(g_i))) = \alpha(\delta_S(g_i))= 
	\delta_S(\beta(g_i))\in\Mor(\rho_i(\call_i))  \]
for all $g\in S_i$ since $\beta(g_i)\in\beta(S_i)=S_i$. 

By \cite[Theorem 1.12]{AOV1}, each morphism in $\call_i$ is a composite of 
restrictions of elements of $\Aut_{\call_i}(P)$ for fully normalized 
subgroups $P\in\calf_i^{cr}$. Also, when $P$ is fully normalized, 
$\delta_P(N_S(P))\in\sylp{\Aut_{\call_i}(P)}$ (see \cite[Proposition 
III.4.2(c)]{AKO}), and so $\Aut_{\call_i}(P)$ 
is generated by $\delta_P(N_S(P))$ and elements of order prime to $P$. Thus 
each morphism in $\call_i$ is a composite of restrictions of 
automorphisms of order prime to $p$ and elements of $\delta_i(S_i)$, and so 
$\alpha(\rho_i(\call_i))=\rho_i(\call_i)$. 
\end{proof}

As one example, if $G_1,\dots,G_k$ are finite groups, $S_i\in\sylp{G_i}$, 
and $\call_i=\call_{S_i}^c(G_i)$, then it is an easy exercise to show 
that the linking system $\call$ defined in Proposition \ref{prod.link1} is 
the centric linking system of $\xxx{G}$.

The subcategory $\5\call\subseteq\call$ defined in Proposition 
\ref{prod.link1}(b) is not a linking system, since $\Ob(\5\call)$ is not 
closed under overgroups. However, 
	\[ \Ob(\5\call) = \bigl\{ P=\xxx{P} \,\big|\, P_i\le S_i,~ 
	P\in\calf^c \bigr\} \]
does include all subgroups of $S$ that are $\calf$-centric and 
$\calf$-radical: this is shown in \cite[Lemma 3.1]{AOV1} when $k=2$ and 
follows in the general case by iteration. So Proposition \ref{AutI} 
applies to the automorphism group
	\[ \Aut(\5\call) = \bigl\{ \alpha\in\Aut_{\textup{cat}}(\5\call) 
	\,\big|\, \alpha(\delta_{P,S}(1))=\delta_{\alpha(P),S}(1),~ 
	\alpha(\delta_P(P))=\delta_{\alpha(P)}(\alpha(P))~ 
	\forall~ P\in\Ob(\5\call) \bigr\} . \]

By analogy with finite groups, a saturated fusion system is 
\emph{indecomposable} if it is not the direct product of two proper fusion 
subsystems. 

\begin{Lem} \label{prod.link2}
Assume Hypotheses \ref{hyp:prod.fus}. Let $\call_1,\dots,\call_k$ be 
centric linking systems associated to $\calf_1,\dots,\calf_k$, 
respectively, and let $\call$ be the centric linking system associated to 
$\calf$ defined as in Proposition \ref{prod.link1}. Let 
$\5\call\subseteq\call$ be the full subcategory with objects the subgroups 
$\xxx{P}\le S$ for $P_i\in\calf_i^c=\Ob(\call_i)$, and let 
$\xi_\call\:\prod_{i=1}^k\call_i\xto{~\cong~}\5\call\le\call$ be as in Proposition 
\ref{prod.link1}(b). Then 
\begin{enuma} 
\item $\xi_\call$ induces a homomorphism 
	\[ c_\xi\: \Aut(\call_1)\times\cdots\times\Aut(\call_k) 
	\Right5{} \Aut(\5\call) \]
that sends $(\alpha_1,\dots,\alpha_k)$ to 
$\xi_\call(\xxx{\alpha})\xi_\call^{-1}$;

\item each $\5\alpha\in \Aut(\5\call)$ has a unique extension 
$E_\call(\5\alpha)$ to an automorphism of $\call$, in this way defining 
an injective homomorphism $E_\call\:\Aut(\5\call)\too\Aut(\call)$; and

\item if $Z(\calf_i)=1$ and $\calf_i$ is indecomposable for each 
$i$, then $E_\call$ is an isomorphism.

\end{enuma}
\end{Lem}

\begin{proof} \textbf{(a) } This formula clearly defines a homomorphism to the 
group $\Aut_{\textup{cat}}(\5\call)$ of all automorphisms of $\5\call$ as a 
category. That $\xi_\call(\xxx{\alpha})\xi_\call^{-1}$ (for 
$\alpha_i\in\Aut(\call_i)$) sends inclusions in $\5\call$ to inclusions and 
sends distinguished subgroups to distinguished subgroups 
follows from the commutativity of the square in 
Proposition \ref{prod.link1}(b).

\smallskip

\noindent\textbf{(b) } We will show that each 
$\5\alpha\in\Aut(\5\call)$ extends to some $\alpha\in\Aut(\call)$. By the 
definition in Proposition \ref{prod.link1}, each morphism in $\call$ is a 
restriction of a morphism in $\5\call$, and hence there is at most one such 
extension. So upon setting $E_\call(\5\alpha)=\alpha$, we get a well 
defined injective homomorphism from $\Aut(\5\call)$ to $\Aut(\call)$.

Fix $\5\alpha\in\Aut(\5\call)$, and let 
$\beta=\til\mu_{\call}(\5\alpha)\in\Aut(\calf)$ be the automorphism of 
Proposition \ref{AutI} and Definition \ref{d:til.mu}. Thus 
$\pi\circ\5\alpha=\8\beta\circ\pi$, and $\5\alpha(P)=\beta(P)$ for all 
$P\in\Ob(\5\call)$. In terms of $\call$ and $\5\call$, the definition 
of morphisms in Proposition 3.5 takes the form 
	\beqq \Mor_\call(P,Q) = \bigl\{ \psi\in\Mor_{\5\call}(\5P,\5Q) 
	\,\big|\, \pi(\psi)(P)\le Q \bigr\} \label{e:MorL1xLk} \eeqq
for all $P,Q\in\Ob(\call)$,
where $\5P$ and $\5Q$ are as in Hypotheses \ref{hyp:prod.fus}. Also, 
$\beta(\5P)=\5{\beta(P)}$ for each $P\in\Ob(\call)$, since $\5P$ and 
$\5{\beta(P)}$ are the unique minimal objects of $\5\call$ containing $P$ 
and $\beta(P)$, and since $\beta$ permutes the objects of $\5\call$ and of 
$\call$. (We are \emph{not} assuming here that $\beta$ permutes the factors 
$S_i$.)

For all $P,Q\in\Ob(\call)$ and 
$\psi\in\Mor_\call(P,Q)\subseteq\Mor_{\5\call}(\5P,\5Q)$, we have 
	\[ \pi(\5\alpha(\psi))(\beta(P)) = \8\beta(\pi(\psi))(\beta(P)) 
	= \beta(\pi(\psi)(P)) \le \beta(Q): \]
the first equality since $\pi\circ\5\alpha=\8\beta\circ\pi$ and the 
inequality since $\pi(\psi)(P)\le Q$ by \eqref{e:MorL1xLk}.
So $\5\alpha(\psi)\in\Mor_\call(\beta(P),\beta(Q))$ by 
\eqref{e:MorL1xLk} again. We can thus define $\alpha\in\Aut(\call)$ 
extending $\5\alpha$ by setting $\alpha(P)=\beta(P)$ for all $P$, and 
letting $\alpha_{P,Q}$ be the restriction of $\5\alpha_{\5P,\5Q}$ for 
$P,Q\in\Ob(\call)$.

\smallskip

\noindent\textbf{(c) } Now assume that $Z(\calf_i)=1$ and $\calf_i$ is 
indecomposable for each $i$. By \cite[Corollary 5.3]{O-KRS}, $\calf$ has a 
unique factorization as a product of indecomposable fusion systems. 

Fix $\alpha\in\Aut(\call)$, and set 
$\beta=\til\mu_\call(\alpha)\in\Aut(\calf)$ (Definition 
\ref{d:til.mu}). 
By the uniqueness of the 
factorization, there is $\gamma\in\Sigma_k$ such that 
$\8\beta(\calf_i)=\calf_{\gamma(i)}$ for each $1\le i\le k$. 
In particular, $\beta(S_i)=S_{\gamma(i)}$ for each $i$. So for each object 
$P=\xxx{P}$ in $\5\call$, $\beta(P)=\prod_{i=1}^k\beta(P_i)$ 
is also an object in $\5\call$. Hence $\alpha(\5\call)=\5\call$ and 
$\alpha=E_\call(\alpha|_{\5\call})$. Since $\alpha\in\Aut(\call)$ was 
arbitrary, $E_\call$ is onto.
\end{proof}

We are now ready to prove our main results concerning the automorphism 
group of a product of linking systems. 

\begin{Prop} \label{p:prod.tame}
Assume Hypotheses \ref{hyp:prod.fus}. Let $\call_i$ be a centric linking 
system associated to $\calf_i$ for each $1\le i\le k$, and let $\call$ 
be the centric linking system associated to $\calf$ defined as in 
Proposition \ref{prod.link1}. Set 
	\[ \Gamma = \bigl\{ \sigma\in\Sigma_k \,\big|\, 
	\calf_{\sigma(i)}\cong\calf_i ~\textup{for each $1\le i\le k$} 
	\bigr\}. \]
\begin{enuma} 

\item There is an injective homomorphism 
	\[ \Phi_\call\:\bigl(\Aut(\call_1)\times\cdots\times\Aut(\call_k)\bigr)
	\rtimes\Gamma \Right5{} \Aut(\call) \]
with the property that for each 
$(\alpha_1,\dots,\alpha_k)\in\prod_{i=1}^k\Aut(\call_i)$, 
$(P_1,\dots,P_k)\in\prod_{i=1}^k\Ob(\call_i)$, and 
$(\varphi_1,\dots,\varphi_k)\in\textstyle\prod_{i=1}^k\Mor(\call_i)$, 
we have 
	\beqq \begin{split} 
	\Phi_\call\bigl(\alpha_1,\dots,\alpha_k\bigr)
	\bigl(\xxx{P}\bigr) &= 
	\alpha_1(P_1)\times\cdots\times\alpha_k(P_k) \\
	\Phi_\call\bigl(\alpha_1,\dots,\alpha_k\bigr)
	\bigl(\varphi_1,\dots,\varphi_k\bigr) &= 
	\bigl(\alpha_1(\varphi_1),\dots,\alpha_k(\varphi_k)\bigr) 
	\end{split} \label{e:Phi0-def} \eeqq 
for $\alpha_i\in\Aut(\call_i)$. Furthermore, 
	\[ \Phi_\call\bigl(\textstyle\prod_{i=1}^k\Aut(\call_i)\bigr)
	=\Aut^0(\call) \defeq \bigl\{ \alpha\in\Aut(\call) \,\big|\, 
	\alpha_S(\delta_S(S_i))=\delta_S(S_i) ~\textup{for each $1\le i\le 
	k$} \bigr\}. \]

\item If $Z(\calf)=1$, and $\calf_i$ is indecomposable for each $i$, 
then $\Phi_\call$ is an isomorphism, and induces an isomorphism  
	\[ \4\Phi_\call\:\bigl(\Out(\call_1)\times\cdots\times
	\Out(\call_k)\bigr) \rtimes\Gamma \Right5{\cong} \Out(\call). \]

\item Assume that $Z(\calf)=1$, and that $\calf_i$ is indecomposable and 
tame for each $i$. Then $\calf$ is tame. If $G_1,\dots,G_k$ are finite 
groups such that $O_{p'}(G_i)=1$ and $\calf_i$ is tamely realized by $G_i$ 
for each $i$, and such that $\calf_i\cong\calf_j$ implies $G_i\cong G_j$, 
then $\calf$ is tamely realized by the product $\xxx{G}$.

\end{enuma}
\end{Prop}

\begin{proof} Let $\5\call\subseteq\call$ be the full subcategory defined 
in Proposition \ref{prod.link1}(b). Thus $\Ob(\5\call)$ is the set of all 
$P\in\Ob(\call)=\calf^c$ such that $P=\5P$.

Without loss of generality, for each pair of indices $i,j$ such that 
$\calf_i\cong\calf_j$, we can assume that $\calf_i=\calf_j$ and $S_i=S_j$. 
Then $\call_i\cong\call_j$ by Theorem \ref{t:unique.l.s.}, and so we can 
also assume that $\call_i=\call_j$.

\smallskip

\noindent\textbf{(a) } Define $\Phi_\call$ to be the composite
	\[ \Phi_\call \: \bigl( 
	\Aut(\call_1)\times\cdots\times\Aut(\call_k) \bigr) \rtimes\Gamma 
	\Right4{\5c_\xi} \Aut(\5\call)  \Right4{E_\call} \Aut(\call) 
	 \]
where $E_\call$ is the homomorphism of Lemma \ref{prod.link2}(b), where 
the restriction of $\5c_\xi$ to $\prod_{i=1}^k\Aut(\call_i)$ is the 
homomorphism $c_\xi$ of Lemma \ref{prod.link2}(a), and where 
	\beqq \5c_\xi(\gamma)(\varphi_1,\dots,\varphi_k)= 
	(\varphi_{\gamma^{-1}(1)},\dots,\varphi_{\gamma^{-1}(k)}) 
	\label{e:Phi1-def} \eeqq
for $\gamma\in\Gamma$ and $\varphi_i\in\Mor(\call_i)$. Then 
\eqref{e:Phi0-def} holds by the definition of $c_\xi$. One easily checks 
using \eqref{e:Phi0-def} and \eqref{e:Phi1-def} that $\Phi_\call$ is an 
injective homomorphism.

It remains to check that $\Phi_\call\bigl(\prod_{i=1}^k\Aut(\call_i)\bigr) 
=\Aut^0(\call)$: the subgroup of those $\alpha\in\Aut(\call)$ such that 
$\alpha_S(\delta_S(S_i))=\delta_S(S_i)$ for each $i$. The inclusion of the 
first group in $\Aut^0(\call)$ is clear. By Proposition \ref{prod.link1}(c), 
there are embeddings of categories $\rho_i\:\call_i\too\call$ sending 
$P_i\le S_i$ to its product with the $S_j$ for all $j\ne i$, and 
$\alpha(\rho_i(\call_i))=\rho_i(\call_i)$ for each 
$\alpha\in\Aut^0(\call)$. We can thus define 
$\Psi_{\call,i}\:\Aut^0(\call)\too\Aut(\call_i)$ by sending $\alpha$ to 
$\rho_i^{-1}\alpha\rho_i$. Set 
$\Psi^0_\call=(\Psi_{\call,1},\dots,\Psi_{\call,k})$; then 
$\Phi_\call\circ\Psi^0_\call$ is the inclusion of $\Aut^0(\call)$ into 
$\Aut(\call)$, and so $\Aut^0(\call)\le\Im(\Phi_\call)$.

\smallskip

\noindent\textbf{(b) } Fix $\alpha\in\Aut(\call)$, and set 
$\beta=\til\mu_\call(\alpha)\in\Aut(\calf)$ (see Definition 
\ref{d:til.mu}). By \cite[Proposition 3.6]{AOV1} and since $Z(\calf)=1$ 
and the $\calf_i$ are indecomposable, $\8{\beta}$ permutes the factors 
$\calf_i$.

Let $\gamma\in\Sigma_k$ be such that $\8{\beta}(\calf_i)=\calf_{\gamma(i)}$ 
for each $i$, and hence also $\beta(S_i)=S_{\gamma(i)}$ and 
$\alpha_S(\delta_S(S_i))=\delta_S(S_{\gamma(i)})$. In particular,  
$\gamma\in\Gamma$. Then 
$\Phi_\call(\gamma)^{-1}\circ\alpha\in\Aut^0(\call)$, and since 
$\Aut^0(\call)\le\Im(\Phi_\call)$ by (a), this shows that 
$\Aut(\call)\le\Im(\Phi_\call)$ and hence that $\Phi_\call$ is onto.

Since $\Aut_\call(S)=\Aut_{\call_1}(S_1)\times\cdots 
\times\Aut_{\call_k}(S_k)$, $\Phi_\call$ induces an isomorphism of quotient 
groups 
	\[ \4\Phi_\call\: \bigl(\Out(\call_1)\times\cdots 
	\times\Out(\call_k)\bigr) \rtimes \Gamma \Right5{\cong} 
	\Out(\call). \]

\smallskip

\noindent\textbf{(c) } Assume now that $Z(\calf)=1$, and that $\calf_i$ is 
indecomposable and tame for each $i$. Let $G_1,\dots,G_k$ be such that 
$O_{p'}(G_i)=1$ and $\calf_i$ is tamely realized by $G_i$ for each $i$, and 
such that $\calf_i\cong\calf_j$ implies $G_i\cong G_j$. For each $i$, 
$Z(G_i)$ is a $p$-group and hence $Z(G_i)\le Z(\calf_i)=1$ by Lemma 
\ref{l:NG(Q)}(b). Without loss of generality, we can 
assume that $G_i=G_j$ whenever $G_i\cong G_j$ and also (by the uniqueness 
of linking systems again) that $\calf_i=\calf_{S_i}(G_i)$ and 
$\call_i=\call_{S_i}^c(G_i)$ for each $i$. Note that each $G_i$ is 
indecomposable: since $O_{p'}(G_i)=1$, a nontrivial factorization of $G_i$ 
would induce a nontrivial factorization of $\calf_i$.

Set $\kappa_i=\kappa_{G_i}\:\Out(G_i)\too\Out(\call_i)$ for short. Fix 
splittings $s_i\:\Out(\call_i)\too\Out(G_i)$ for all $i$, chosen so 
that $s_i=s_j$ if $G_i=G_j$. 

Consider the following diagram
	\[ \xymatrix@C=40pt@R=25pt{ 	
	\bigl(\Out(\call_1)\times\cdots\times\Out(\call_k)\bigr) 
	\rtimes\Gamma \ar[d]^{(s_1,\dots,s_k)\rtimes\Id_\Gamma} 
	\ar[r]^-{\4\Phi_\call}_-{\cong} \ar@<-25mm>@/_2pc/[dd]_{\Id} 
	& \Out(\call) \ar[d]^{s} \\
	\bigl(\Out(G_1)\times\cdots\times\Out(G_k)\bigr) 
	\rtimes\Gamma \ar[d]^{(\kappa_1,\dots,\kappa_k)\rtimes\Id_\Gamma} 
	\ar[r]^-{\4\Phi_G}_-{\cong} & \Out(G)  \ar[d]^{\kappa_G} \\ 
	\bigl(\Out(\call_1)\times\cdots\times\Out(\call_k)\bigr) 
	\rtimes\Gamma \ar[r]^-{\4\Phi_\call}_-{\cong} & \Out(\call) 
	} \] 
where $\4\Phi_G$ is the isomorphism of Proposition \ref{p:Out(G1x...xGk)} 
(as defined when taking $\lambda_{ij}=\Id_{G_i}$ for each $i<j$ such that 
$G_i=G_j$), where $s$ is defined to make the top square commute, and where 
the commutativity of the bottom square is immediate from the definitions. 
Thus $\kappa_G\circ s=\Id_{\Out(\call)}$, so $\kappa_G$ is split 
surjective, and $\calf$ is tamely realized by $G$. 
\end{proof}

\bigskip

\section{Components of groups and of fusion systems}
\label{s:comp}

\newcommand{\snsg}{\nsg\,\nsg}
\newcommand{\KC}{{\scrk\scrc}}

In this section, we set up some tools that will be used later when proving 
inductively that all realizable fusion systems are tame. The starting point 
for the inductive procedure is Theorem \ref{OR-ThA}, which summarizes the 
main results in \cite{pprime}. 
Note that if $\calf$ is a saturated fusion system such that $O^{p'}(\calf)$ is 
simple, then for each finite group $G$ with $O_{p'}(G)=1$ that realizes 
$\calf$, $O^{p'}(G)$ is simple (so $G$ is almost simple), and $O^{p'}(G)$ 
realizes $\calf$ if $\calf$ is simple.

Let $\Comp(G)$ denote the set of components of a finite group $G$; 
i.e., the set of subnormal subgroups of $G$ that are quasisimple. 
(Recall that a subgroup $H$ of $G$ is subnormal, denoted $H\snsg 
G$, if there is a sequence $H=H_0\nsg H_1\nsg\cdots\nsg H_k=G$ with 
each subgroup normal in the following one, and $H$ is quasisimple if 
$H$ is perfect and $H/Z(H)$ is simple.) The components of $G$ commute 
with each other pairwise (see \cite[\S\,31]{A-FGT} or \cite[Lemma 
A.12]{AKO}). In particular, when $O_q(G)=1$ for all primes $q$, they are all 
simple groups, and the subgroup $E(G)\defeq\gen{\Comp(G)}$ is their 
direct product. 

Similarly, the components of a saturated fusion system $\calf$ over a 
finite $p$-group $S$ are its subnormal fusion subsystems 
$\calc\snsg\calf$ that are quasisimple (i.e., 
$O^p(\calc)=\calc$ and $\calc/Z(\calc)$ is simple). The set of 
components of $\calf$ will be denoted $\Comp(\calf)$.

By analogy with the case for groups, a \emph{central product} of fusion 
systems $\cale_1,\dots,\cale_k$ is a fusion system 
$\cale\cong(\xxx\cale)/Z$, for some central subgroup 
$Z\le\prod_{i=1}^kZ(\cale_i)$ that intersects trivially with each factor 
$Z(\cale_i)$. More precisely, if $\calf$ is a fusion system over $S$ and 
$\cale_1,\dots,\cale_k\le\calf$ are fusion subsystems over 
$T_1,\dots,T_k\le S$, then the subsystems commute in $\calf$ if 
the $T_i$ commute pairwise, and for each $k$-tuple of morphisms 
$(\varphi_1,\dots,\varphi_k)$, where $\varphi_i\in\Hom_{\cale_i}(P_i,Q_i)$, 
there is a morphism $\4\varphi\in\homf(P_1\cdots P_k,Q_1\cdots Q_k)$ that 
extends each of the $\varphi_i$. Note in particular that 
	\[ \cale_1,\dots,\cale_k ~\textup{commute}~ \implies 
	\cale_i\le C_\calf(T_j) ~\textup{for each $i\ne j$.} \] 
In this situation, the (internal) central 
product of the $\cale_i$ is the fusion subsystem 
	\begin{multline*} 
	\cale_1\cdots\cale_k = \Gen{ \4\varphi\in\homf(P_1\cdots P_k, 
	Q_1\cdots Q_k) \,\big|\, P_i,Q_i\le T_i,~ \\ 
	\4\varphi|_{P_i}\in \Hom_{\cale_i}(P_i,Q_i) ~ 
	\forall\,1\le i\le k } \le \calf
	\end{multline*}
over $T_1\cdots T_k\le S$. See Definition 2.4 and Lemma 2.8 in \cite{O-KRS} 
for some more details, and see \cite[Proposition 3.3]{Henke} for a slightly 
different approach to defining central products of fusion subsystems.

By \cite[9.8--9.9]{A-gfit}, the components of a saturated fusion system 
$\calf$ commute, and also commute with $O_p(\calf)$. So by analogy with 
finite groups, when $\calf$ is a saturated fusion system over a finite 
$p$-group and $\Comp(\calf)=\{\calc_1,\dots,\calc_k\}$, one defines 
	\[ E(\calf) = \calc_1\cdots\calc_k \qquad\textup{and}\qquad 
	F^*(\calf)=E(\calf)O_p(\calf) \]
(central products). In particular, $F^*(\calf)$ is the \emph{generalized 
Fitting subsystem} of $\calf$.

Note that when $O_p(\calf)=1$, the components of $\calf$ are 
all simple, and $F^*(\calf)=E(\calf)$ is their direct product. 

\begin{Lem} \label{l:A2-Th6}
Let $\calf$ be a saturated fusion system over a finite $p$-group $S$. Then 
\begin{enuma} 

\item $E(\calf)$ is characteristic in $\calf$;

\item $F^*(\calf)$ is characteristic and centric in $\calf$; and

\item if $\cale\nsg\calf$ and $\calc\in\Comp(\calf)\sminus\Comp(\cale)$, 
then $\calf$ contains a central product of $\calc$ and $\cale$. 

\end{enuma}
\end{Lem}

\begin{proof} These are all shown in Chapter 9 of \cite{A-gfit}: point (a) 
in 9.8.1 and 9.8.2, point (b) in 9.9 and 9.11, and (c) in 9.13. 
More precisely, $F^*(\calf)$ is centric in $\calf$ since 
$C_S(F^*(\calf))=Z(F^*(\calf))$ by \cite[9.11]{A-gfit}.
\end{proof}

\begin{Lem} \label{l:same.comps}
Let $\cale\nsg\calf$ be a normal pair of saturated fusion systems over 
finite $p$-groups. Then 
\begin{enuma} 

\item $\Comp(\cale)$ is equal to the set of all $\calc\in\Comp(\calf)$ such 
that $\calc\le\cale$; 

\item if $\cale$ is centric in $\calf$ or has $p$-power index in 
$\calf$, then $\Comp(\cale)=\Comp(\calf)$; and  

\item if $O_p(\calf)=1$ and $\Comp(\cale)=\Comp(\calf)$, then $\cale$ 
is centric in $\calf$. 

\end{enuma}
\end{Lem}

\begin{proof} \textbf{(a,b) } In all cases, each fusion subsystem 
subnormal in $\cale$ is subnormal in $\calf$, and hence 
$\Comp(\cale)\subseteq\Comp(\calf)$. If 
$\calc\in\Comp(\calf)\sminus\Comp(\cale)$, then by Lemma \ref{l:A2-Th6}(c), 
$\calf$ contains a central product of $\calc$ and $\cale$, and in 
particular, $\calc\nleq\cale$. This proves (a), and also shows that 
$\cale$ is not centric in $\calf$ in this case, proving the first part 
of (b). If $\cale$ has $p$-power index in $\calf$, then $\calf$ cannot 
be a central product of $\cale$ with a quasisimple system, so 
$\Comp(\cale)=\Comp(\calf)$ also in this case.

\smallskip

\noindent\textbf{(c) } If $O_p(\calf)=1$ and 
$\Comp(\cale)=\Comp(\calf)$, then $\cale\ge E(\calf)=F^*(\calf)$ is the 
generalized Fitting subsystem of $\calf$. Since $F^*(\calf)$ is centric 
in $\calf$ by Lemma \ref{l:A2-Th6}(b), so is $\cale$. 
\end{proof}

In the proof of the next lemma, we need to work with the centralizer 
fusion subsystem $C_\calf(\cale)$ of a normal fusion subsystem 
$\cale\nsg\calf$ over $T\nsg S$. This was defined by Henke \cite{Henke} to 
be the unique fusion subsystem over $C_S(\cale)$ of $p$-power index in 
$C_\calf(T)$. (There is such a subsystem by \cite[Theorem I.7.4]{AKO} and 
since $C_S(\cale)\ge\foc(C_\calf(T))$ by \cite[Proposition 1]{Henke}.) By 
\cite[Proposition 6.3]{Henke}, it is equal to the subsystem 
$C_\calf(\cale)$ defined by Aschbacher in \cite[Chapter 6]{A-gfit}.

\begin{Lem} \label{l:ZC/Z}
Let $\calf$ be a saturated fusion system over a finite $p$-group $S$. 
Set $\Comp(\calf)=\{\calc_1,\dots,\calc_k\}$ where $\calc_i$ is a 
fusion subsystem over $U_i$ for each $1\le i\le k$. Let $Z\le Z(\calf)$ 
be a central subgroup. Then $\Comp(\calf/Z) = \{Z\calc_1/Z,\dots,Z\calc_k/Z\}$.
\end{Lem}

\begin{proof} Set $\Comp_0(\calf/Z)=\{Z\calc_1/Z,\dots,Z\calc_k/Z\}$ 
for short. For each $i$, $Z\calc_i/Z\snsg\calf/Z$ by Lemmas 
\ref{l:E/Z<|F/Z} and \ref{l:F/Q} and since $\calc_i\snsg\calf$. Also, 
$Z\calc_i/Z\cong\calc_i/(Z\cap U_i)$ by Lemma \ref{l:F/Q=E/Q0}
and hence is quasisimple. Thus 
$Z\calc_i/Z\in\Comp(\calf/Z)$ for each $i$, and $\Comp(\calf/Z) 
\supseteq \Comp_0(\calf/Z)$.

It remains to prove the opposite inclusion. Set $\cale=E(\calf)$: the 
central product of the $\calc_i$. It is a saturated fusion system over 
$U=U_1\cdots U_k$, and is normal in $\calf$ by Lemma \ref{l:A2-Th6}(a). Set 
$K=\{\alpha\in\autf(ZU)\,|\,[\alpha,U]\le Z\}$: a $p$-group of 
automorphisms by \cite[Corollary 5.3.3]{Gorenstein}. Each 
$\varphi\in\Hom_{C_{\calf/Z}(ZU/Z)}(P/Z,Q/Z)$ (for $Z\le P,Q\le 
N_S^K(ZU)$) extends to $\4\varphi\in\Hom_{\calf/Z}(PU/Z,QU/Z)$ such 
that $\4\varphi|_{ZU/Z}=\Id$, and this in turn lifts to 
$\psi\in\homf(PU,QU)$ with $\psi|_{ZU}\in K$. Thus 
$N_\calf^K(ZU)/Z=C_{\calf/Z}(ZU/Z)$.

Recall that $K$ is a $p$-group. Hence by Lemma \ref{l:C(Q)<|N(Q)}, the 
centralizer $C_\calf(ZU)$ is normal of $p$-power index in 
$N_\calf^K(ZU)$. Also, $C_\calf(\cale)$ has $p$-power index in 
$C_\calf(U)=C_\calf(ZU)$ by Henke's definition in \cite{Henke}, and so we 
have inclusions 
	\[ C_\calf(\cale)/Z \le C_\calf(ZU)/Z \le N_\calf^K(ZU)/Z = 
	C_{\calf/Z}(ZU/Z), \]
each of $p$-power index in the next. Each component of $\calf/Z$ not in 
$\Comp_0(\calf/Z)$ commutes with the $Z\calc_i/Z$, hence is contained 
in $C_{\calf/Z}(ZU/Z)$, hence is contained in $C_\calf(\cale)/Z$ by Lemma 
\ref{l:same.comps}(b), and hence lies in $\Comp(C_\calf(\cale)/Z)$ by 
Lemma \ref{l:same.comps}(a) and since $C_\calf(\cale)/Z\nsg\calf/Z$ by 
Lemma \ref{l:F/Q}. 

By \cite[9.12.3]{A-gfit} and since $\cale=E(\calf)$, the centralizer 
subsystem $C_\calf(\cale)$ is constrained. Hence $C_\calf(\cale)/Z$ is also 
constrained by \cite[Lemma 2.10]{Henke-subcentric}, so 
$\Comp(C_\calf(\cale)/Z)=\emptyset$, finishing the proof that
$\Comp(\calf/Z)=\Comp_0(\calf/Z)$. 
\end{proof}

In the next lemma, we use the following notation to describe certain 
automorphism groups. For a fixed prime $p$ and integers $k\mid m\mid(p-1)$ and 
$n,\ell\ge1$, let $G(m,k,n)\le\GL_n(\Z/p^\ell)$ be the subgroup 
	\[ G(m,k,n) = \bigl\{ \diag(u_1,\dots,u_n)\in\GL_n(\Z/p^\ell) 
	\,\big|\, u_i^m=1~\forall i,~ (u_1\cdots u_n)^{m/k}=1 \bigr\} \cdot 
	\mathfrak{Perm}(n), \]
where $\mathfrak{Perm}(n)\simeq\Sigma_n$ is the group of all permutation 
matrices. Thus $G(m,1,n)\cong C_m\wr\Sigma_n$, the 
group of all monomial matrices whose nonzero entries are $m$-th roots of 
unity in $\Z/p^\ell$, and $G(m,k,n)$ has index $k$ in $G(m,1,n)$. 

The following is a version of \cite[Lemma 4.7]{pprime} that has been 
reformulated so as to not depend on the classification of finite simple 
groups.

\begin{Lem} \label{l:exotic}
Let $\calf$ be a saturated fusion system over a finite $p$-group $S$, 
for some prime $p\ge5$, and assume $A\nsg S$ is abelian and 
$\calf$-centric. Assume also, for some $\ell\ge1$, $\kappa\ge p$, and $2< 
m\mid(p-1)$, that $A$ is homocyclic of rank $\kappa$ and exponent $p^\ell$, 
and that with respect to some basis $\{a_1,\dots,a_\kappa\}$ for $A$ as a 
$\Z/p^\ell$-module, $\autf(A)$ contains $G(m,m,\kappa)$ with index prime to 
$p$, and 
	\[ \autf(A)\cap G(m,1,\kappa)=G(m,r,\kappa) \le 
	\GL_\kappa(\Z/p^\ell) \qquad\textup{for some $2<r\mid m$.} \]
Then either $A\nsg\calf$, or $O^{p'}(\calf)$ is simple and $\calf$ is 
not realized by any known finite almost simple group.
\end{Lem}

\begin{proof} Since $\autf(A)$ contains $G(m,m,\kappa)$ with index prime to 
$p$, some subgroup conjugate to $\Aut_S(A)$ is contained in 
$G(1,1,\kappa)\cong\Sigma_\kappa$, and hence $\Aut_S(A)$ permutes some 
basis of $\Omega_1(A)$. Also, $G(m,m,\kappa)$ acts faithfully on 
$\Omega_1(A)$, as does each subgroup of $\autf(A)$ of order prime to $p$ 
(see \cite[Theorem 5.2.4]{Gorenstein}). So by the assumptions on 
$\autf(A)$, 
	\beqq \parbox{95mm}{$\autf(A)$ acts faithfully on $\Omega_1(A)$, 
	$\Aut_S(A)$ permutes a basis of $\Omega_1(A)$, and 
	$C_S(\Omega_1(A))=A$.} \label{e:CS(O1(A))}
	\eeqq

We next claim that
	\beqq \textup{$\Omega_1(A)$ is the only elementary abelian subgroup 
	of $S$ of rank $\kappa$.} \label{e:A-uniq} \eeqq
This is well known, but the proof is simple enough that we give it here. 
Set $V=\Omega_1(A)$ for short, let $W\le S$ be another elementary abelian 
subgroup, and set $\4W=\Aut_W(V)$ and $r=\rk(\4W)$. Then 
$r=\rk(W/C_W(V))$ where $C_W(V)=W\cap A=W\cap V$ by \eqref{e:CS(O1(A))}. 
Let $\calb$ be a basis for $V$ permuted by $\4W$, and assume $\4W$ acts on 
$\calb$ with $s$ orbits (including fixed orbits) of lengths 
$p^{m_1},\dots,p^{m_s}$. Then $p^r=|\4W|\le p^{m_1}\cdots p^{m_s}$, 
and hence $m_1+\dots+m_s\ge r$. So 
	\[ \rk(W) = r+\rk(W\cap V) \le r+\rk(C_V(\4W)) 
	= r + s \le \sum_{i=1}^s(m_i+1) < 
	\sum_{i=1}^sp^{m_i}=\rk(V), \]
proving \eqref{e:A-uniq}. In particular, $\Omega_1(A)$ and 
$A=C_S(\Omega_1(A))$ are weakly closed in $\calf$.

Set $G_0(m,m,\kappa)=O^{p'}(G(m,m,\kappa))\cong(C_m)^{\kappa-1}\rtimes 
A_\kappa$: the unique subgroup of index $2$ in $G(m,m,\kappa)$. There are 
exactly $\kappa$ 1-dimensional subspaces of $\Omega_1(A)$ invariant under 
the action of $O_{p'}(G_0(m,m,\kappa))\cong(C_m)^{\kappa-1}$, and they are 
permuted transitively by the alternating group $A_\kappa$. Hence 
	\beqq \textup{$\Omega_1(A)$ is a simple 
	$\F_p[G_0(m,m,\kappa)]$-module.} \label{e:A-simple} \eeqq

Set $\calf_0=O^{p'}(\calf)$. Set $\Gamma=\autf(A)$ and 
$\Gamma_0=\Aut_{\calf_0}(A)$. Thus $\Gamma_0\ge G_0(m,m,\kappa)$ since 
$\Gamma\ge G(m,m,\kappa)$.

\smallskip

\noindent\textbf{Step 1: } Assume $\calf_0$ is not simple, and let 
$\cale\nsg\calf_0$ be a proper nontrivial normal subsystem over $1\ne T\nsg 
S$. Then $T$ is strongly closed in $\calf_0$, so $T\cap A$ is normalized by 
the action of $\Gamma_0$ on $A$, and $T\cap A=\Omega_k(A)$ for some $1\le 
k\le\ell$ since $\Omega_1(A)$ is simple by \eqref{e:A-simple}. Also, 
$T/\Omega_k(A)$ is normal in $S/\Omega_k(A)$, so if $k<\ell$ and 
$T>\Omega_k(A)$, then $T/\Omega_k(A)\cap Z(S/\Omega_k(A))\ne1$. Since 
$Z(S/\Omega_k(A))\le A/\Omega_k(A)$, this implies that 
$T\cap A>\Omega_k(A)$, contradicting the choice of $k$. Thus either 
$T=\Omega_k(A)$ for some $1\le k\le\ell$, or $T>A$.

If $T=\Omega_k(A)$ for some $k\le\ell$, then since $T$ is abelian and 
strongly closed, $T=\Omega_k(A)\nsg\calf$ by \cite[Corollary 
I.4.7(a)]{AKO}. Hence for each $a\in A$ and each $x\in a^\calf$, there is 
$\varphi\in\homf(\Omega_k(A)\gen{a},S)$ such that $\varphi(a)=x$ and 
$\varphi(\Omega_k(A))=\Omega_k(A)$. Then $x\in 
C_S(\Omega_1(A))=A$ by \eqref{e:CS(O1(A))}, so $A$ is strongly closed in 
this case, and $A\nsg\calf$ by \cite[Corollary I.4.7(a)]{AKO} again.

Thus if $\calf_0$ is not simple and $A\nnsg\calf$, then there is a proper 
normal subsystem $\cale\nsg\calf_0$ over $T\nsg S$ such that $T>A$. Set 
$\Delta=\Aut_\cale(A)\nsg\Gamma_0$. Then $\Delta\ge\Aut_T(A)\ne1$ since 
$T>A$, so $p\mid|\Delta|$. Since $\Gamma_0$ contains $G_0(m,m,\kappa)$ with 
index prime to $p$, we have 
	\[ \Delta\cap G_0(m,m,\kappa)\nsg 
	G_0(m,m,\kappa)\cong(C_m)^{\kappa-1}\rtimes A_\kappa \]
where $\kappa\ge p\ge5$ and $p\mid|\Delta\cap G_0(m,m,\kappa)|$. Since 
$A_\kappa$ and hence $G_0(m,m,\kappa)$ have no proper normal 
subgroups of order a multiple of $p$, it follows that $\Delta\ge 
G_0(m,m,\kappa)$, and hence that $\Delta$ has index prime to $p$ in 
$\autf(A)$. But then $\Aut_T(A)=\Aut_S(A)$, so $T=S$ since $A$ is 
$\calf$-centric, and $\cale$ has index prime to $p$ in $\calf_0$ and 
$\calf$ by \cite[Lemma 1.26]{AOV1}. Thus $\cale=\calf_0=O^{p'}(\calf)$, 
contradicting our assumption that $\cale$ is proper.

\smallskip

\noindent\textbf{Step 2: } It remains to show, when $A\nnsg\calf$, 
that $\calf$ is not realized by any known finite almost simple group. 
Assume otherwise: assume $\calf=\calf_S(G)$ where $G$ is almost simple, 
and set $G_0=O^{p'}(G)$. Since $O^{p'}(\calf)$ is simple, $G_0$ is a  
known simple group.

We claim this is impossible. Note that $A$ is a radical $p$-subgroup 
of $G_0$, since $O_p(\Aut_{G_0}(A))=1$ and $p\nmid|C_{G_0}(A)/A|$ (i.e., 
$A$ is $\calf_0$-centric). Although we do not know $\Aut_{G_0}(A)$ 
precisely, we know that it is contained in $\autf(A)$ and contains 
$G_0(m,m,\kappa)\cong(C_m)^{\kappa-1}\rtimes A_\kappa$.

Since $p\ge5$ and $\rk_p(G_0)\ge p$, $G_0$ cannot be a sporadic group by 
\cite[Table 5.6.1]{GLS3}. 

By \cite[\S\,2]{AF}, for each abelian radical $p$-subgroup 
$B\le\Sigma_\kappa$, $\Aut_{\Sigma_\kappa}(B)$ is a product of wreath 
products of the form $\GL_c(p)\wr\Sigma_\kappa$ for $c\ge1$ and 
$\kappa\ge1$. Thus $\Aut_{A_\kappa}(B)$ can have index $2$ in 
$C_{p-1}\wr\Sigma_\kappa$ for some $\kappa$, but not index larger than $2$. 
So $G_0$ cannot be an alternating group.

If $G_0\in\Lie(p)$, then $N_{G_0}(A)$ is a parabolic subgroup by the 
Borel-Tits theorem \cite[Corollary 3.1.5]{GLS3} and since $A$ is centric 
and radical. So in the notation of \cite[\S\,2.6]{GLS3}, $A=U_J$ and 
$N_{G_0}(A)=P_J$ (up to conjugacy) for some set $J$ of primitive roots for 
$G_0$. Hence by \cite[Theorem 2.6.5(f,g)]{GLS3}, 
$O^{p'}(N_{G_0}(A)/A)\cong O^{p'}(L_J)$ is a central product of groups in 
$\Lie(p)$, contradicting the assumption that $O^{p'}(\Aut_G(A))\cong 
G_0(m,m,\kappa)$. 

Now assume that $G_0\in\Lie(q_0)$ for some prime $q_0\ne p$. By 
\cite[10-2]{GL} (and since $p\ge5$), $S\in\sylp{G_0}$ contains a unique 
elementary abelian $p$-subgroup of maximal rank, and by 
\eqref{e:A-uniq}, it must be equal to $\Omega_1(A)$. Hence 
$\autf(A)$ must be as in one of the entries in Table 4.2 or 4.3 in 
\cite{pprime}. \\
$\bullet$ If $G_0$ is a classical group and hence 
$\autf(A)\cong G(\5m,\5r,\kappa)$ for $\5m=\mu$ or $2\mu$ and $\5r\le2$ 
(see the next-to-last column in \cite[Table 4.2]{pprime} 
and recall that $G(\5m,1,\kappa)\cong C_{\5m}\wr\Sigma_\kappa$), 
then the identifications $\autf(A)\cong G(\5m,\5r,\kappa)$ and $\autf(A)\ge 
G(m,r,\kappa)$ are based on the same decompositions of $A$ as a direct sum 
of cyclic subgroups, and hence we have $m\mid\5m$ and $r\le2$, 
contradicting our original assumption. \\
$\bullet$ If $G_0$ is an 
exceptional group, then by \cite[Table 4.3]{pprime}, either 
$\kappa=\rk(A)<p$, or $p=3$, or (in case (b)) $m^{\kappa-1}\cdot\kappa!$ 
does not divide $|\autf(A)|$ for any $m>2$ and hence $\autf(A)$ cannot 
contain any such $G(m,r,\kappa)$. 
\end{proof}

As noted at the beginning of the section, if $\calf$ is a realizable 
fusion system such that $O^{p'}(\calf)$ is simple, then $\calf$ is realized 
by a finite almost simple group, and by a simple group if $\calf$ is 
simple. We now consider the converse, by determining which fusion 
systems of known finite simple groups are simple or almost simple.

\begin{Thm} \label{OR-ThA}
Fix a prime $p$ and a known finite quasisimple group $L$ such that 
$Z(L)$ is a $p$-group and $p\mid|L|$. Fix $T\in\sylp{L}$, and set 
$\cale=\calf_T(L)$ and $\calc=O^{p'}(\cale)$. Then $T>Z(L)$, 
and either 
\begin{enuma} 

\item $T\nsg\cale$ and hence $\calc$ is not quasisimple; or 

\item $p=3$ and $L\cong G_2(q)$ for some $q\equiv\pm1$ (mod $9$), in which 
case $|O_3(\cale)|=3$, $Z(\cale)=1$, and $\calc<\cale$ is 
quasisimple and is realized by $\SL_3^\pm(q)$; or 

\item $p\ge5$, $L$ is one of the simple classical groups $\PSL_n^\pm(q)$, 
$\PSp_{2n}(q)$, $\varOmega_{2n+1}(q)$, or $\POmega_{2n+2}^\pm(q)$ where 
$n\ge2$ and $q\not\equiv0,\pm1$ (mod $p$), in which case $\calc$ is simple 
and is not realized by any known finite simple group; or 

\item $\calc$ is quasisimple, $Z(\calc)=Z(L)$, and $\calc$ is 
realized by a known finite quasisimple group with center $Z(\calc)$. 

\end{enuma}
In cases (b), (c), and (d), there is a normal fusion subsystem 
$\calc^*\nsg\cale$ over $T$ containing $\calc$ that is realized by a known 
finite quasisimple group, and is such that for each saturated fusion system 
$\cale'$ over $T$ such that $O^{p'}(\cale')=\calc$, $\cale'$ is realized by 
a known finite quasisimple group only if it contains $\calc^*$. Thus 
$\calc^*=\calc$ in cases (b) and (d), while $\calc^*>\calc$ in case 
(c). 
\end{Thm}

\begin{proof} In all cases, $T\ne1$ since $p\mid|L|$ by assumption, and 
$T>Z(L)$ since otherwise $p\nmid|L/Z(L)|$ while $p\mid|Z(L)|$, 
contradicting the assumption that $L$ is perfect. 

When $L$ is simple, this is essentially \cite[Theorem 4.8]{pprime}, but 
restated to make its proof independent of the classification of finite 
simple groups. The only difference between the proof of this version and 
that of Theorem 4.8 in \cite{pprime} is that we replace \cite[Lemma 
4.7]{pprime} by the above Lemma \ref{l:exotic}. Note in case (a) that 
$\calc=\calf_T(T)$ is not quasisimple since $O^p(\calc)=1$, and in case (d) 
that $\calc$ is simple and is realized by a known finite simple group (thus 
with center $Z(L)=L(\calc)=1$). In case (b), $Z(\cale)=1$ since $\Aut_L(T)$ 
acts nontrivially on $O_3(\cale)=O_3(\calf_T(L))\cong C_3$.

Now assume that $Z(L)\ne1$. Then $Z(L)\le T$ since $Z(L)$ is a $p$-group by 
assumption, and $\calf_{T/Z(L)}(L/Z(L)) \cong \cale/Z(L)$. Also, $L/Z(L)$ 
is not isomorphic to $G_2(q)$ for any prime power $q$ and is not one of the 
groups in case (c) (see \cite[6-1]{GL}, or Theorem 6.1.4 and Table 
6.1.2 in \cite{GLS3}, for the description of Schur multipliers of groups of 
Lie type in cross characteristic). So either $T/Z(L)\nsg\cale/Z(L)$ (case 
(a)), in which case $T\nsg\cale$; or else $O^{p'}(\cale/Z(L))=\calc/Z(L)$ 
is simple and is realized by a known simple group (case (d)). 

If $\calc/Z(L)$ is realized by a known simple group $H$, then $\calc$ is 
realized by a central extension $\5H$ of $Z(L)$ by $H$. This follows from 
the proof of \cite[Corollary 6.14]{BCGLO2} (the statement itself only says 
that $\calc$ is realizable). We have $\foc(\calc/Z(L))=T/Z(L)$ since $H$ is 
simple, and hence $T=Z(L)\foc(\calc)$. Also, 
	\[ T = \foc(\cale) = \foc(\calc)[\Aut_\cale(T),T] = \foc(\calc), 
	\]
where the first equality holds since $\cale=\calf_T(L)$ where $L$ is 
perfect, the second holds since $\cale$ is generated by 
$\calc=O^{p'}(\cale)$ and $\Aut_\cale(T)$ by the Frattini condition for 
$O^{p'}(\cale)\nsg\cale$ (see \cite[Theorem I.7.7]{AKO}), and the last 
equality holds since $T=Z(L)\foc(\calc)$ where $\Aut_\cale(T)$ acts 
trivially on $Z(L)$ and sends $\foc(\calc)$ to itself. Thus $\calc$ and 
$\5H$ are quasisimple by the focal subgroup theorems (Lemma 
\ref{l:foc(F)}(a,b)), and $\5H$ is a known quasisimple group. So the last 
statement in the theorem holds in this case with $\calc^*=\calc$.
\end{proof}

The following terminology will be useful when stating many of the results 
throughout the rest of the paper. 

\begin{Defi} \label{d:KC-gps}
Let $G$ be a finite group. We say that 
\begin{enuma} 

\item $G$ is \emph{$p'$-reduced} if $O_{p'}(G)=1$; and 

\item $G$ is a \emph{$\KC$-group} if all of its 
components are known quasisimple groups.

\end{enuma}
\end{Defi}

Most of our statements from now on about groups will be formulated in terms 
of ``finite $p'$-reduced $\KC$-groups'', a restriction that allows us to 
avoid assuming the classification of finite simple groups in our proofs. 
Note that when working with the $p$-local structure of a finite group $G$, 
there's not much point in assuming that all components of $G$ are known 
without also assuming that $O_{p'}(G)=1$.

\begin{Lem} \label{l:CG(QT)}
Let $G$ be a finite $p'$-reduced $\KC$-group. Then for $U\in\sylp{F^*(G)}$, 
the centralizer $C_G(U)$ is $p$-solvable.
\end{Lem}

\begin{proof} Let $L_1,\dots,L_k$ be the components of $G$, and fix 
$T_i\in\sylp{L_i}$. Set $L=L_1\cdots L_k$ and $T=T_1\cdots T_k$, and set 
$Q=O_p(G)$. Then $F^*(G)=QL$ and $QT\in\sylp{QL}$, so we 
can assume $U=QT$. 

Conjugation by each element of $G$ permutes the subgroups $L_i$. Since 
$O_{p'}(G)=1$, and since the Schur multiplier of a group of order prime to 
$p$ has order prime to $p$, we have $T_i>Z(L_i)$ for each $i$. So each 
element of $C_G(U)$ normalizes each of the $L_i$. 

Consider the homomorphism $\gamma\:C_G(U)\too\Out(L)$ that sends $g\in 
C_G(U)$ to the class of $c_g|_L$. We just showed that 
$\Im(\gamma)\le\prod_{i=1}^k\Out(L_i)$, when identified with a subgroup of 
$\Out(L)$ in the obvious way. Moreover, $\Out(L_i)\le\Out(L_i/Z(L_i))$ is 
solvable for each $i$ by the Schreier conjecture and since $L_i/Z(L_i)$ is 
a known simple group (see \cite[Theorem 7.1.1]{GLS3}), so $\Im(\gamma)$ is 
also solvable. 

If $g\in\Ker(\gamma)$, then $c_g|_L=c_h|_L$ for some $h\in L$, so 
$gh^{-1}\in C_G(L)$. Then $gh^{-1}\in C_G(QL)$ since $h\in L\le C_G(Q)$, 
and $h\in C_L(T)$ since $g$ and $gh^{-1}$ both commute with $T$. Thus 
$\Ker(\gamma)\le C_G(QL)C_L(T)$ (and the opposite inclusion is obvious). 
Also, $C_G(QL)\nsg G$ since $QL\nsg G$, and so $C_G(QL)\nsg\Ker(\gamma)$. 
By \cite[31.13]{A-FGT} or \cite[Theorem A.13(c)]{AKO}, $C_G(QL)\le QL$, and 
hence $C_G(QL)=Z(QL)=Z(Q)$: an abelian $p$-group. Also, $C_L(T)/Z(T)$ has 
order prime to $p$ since $T\in\sylp{L}$. So $\Ker(\gamma)$, 
and hence $C_G(U)$, are $p$-solvable.
\end{proof}

The following notation will be used in several of the results in the 
rest of the section. 

\begin{Not} \label{h:Comp(F)}
Let $G$ be a finite $p'$-reduced $\KC$-group, fix $S\in\sylp{G}$, and set 
$\calf=\calf_S(G)$. Set $\Comp(G)=\{L_1,\dots,L_k\}$; thus each $L_i$ is a 
known quasisimple group. For each $1\le i\le k$, set 
	\[ T_i=S\cap L_i, \quad \cale_i=\calf_{T_i}(L_i)\le\calf, 
	\quad\textup{and}\quad \calc_i=O^{p'}(\cale_i). \]
Assume the $L_i$ were ordered so that for some $0\le m\le k$, 
$\calc_i$ is quasisimple if and only if $i\le m$. 
\end{Not}

\begin{Prop} \label{p:Comp(F)-1}
Let $G$ be a finite $p'$-reduced $\KC$-group, and assume Notation 
\ref{h:Comp(F)}. Then 
	\[ \Comp(\calf)=\{\calc_1,\dots,\calc_m\} \qquad\textup{and}\qquad
	O_p(\calf)\ge O_p(G)T_{m+1}\cdots T_k. \]
\end{Prop}

\begin{proof} Set $L=L_1\cdots L_k$ and $T=T_1\cdots T_k$. For each 
$1\le i\le k$, we have $T_i\in\sylp{L_i}$ since $L_i\nsg L\nsg G$, and 
also
	\[ \calc_i = O^{p'}(\calf_{T_i}(L_i)) \nsg \calf_{T_i}(L_i) 
	\nsg \calf_T(L) \nsg \calf_S(G) = \calf: \]
the first normality relation by \cite[Theorem I.7.7]{AKO} and the other 
two by \cite[Proposition I.6.2]{AKO} and since $L_i\nsg L\nsg G$. Thus 
$\calc_i\snsg\calf$ for each $i$. If $i\le m$, then $\calc_i$ is 
quasisimple and hence is a component of $\calf$. If $i\ge m+1$, then 
$T_i\nsg\cale_i$ by Theorem \ref{OR-ThA}, and hence $T_i\le 
O_p(\calf_T(L))\le O_p(\calf)$ by Lemma \ref{l:E<|F}(b).

Set $Q=O_p(G)$ for short. Thus $QT_{m+1}\cdots T_k\le O_p(\calf)$, and 
$T_1\cdots T_m$ is the Sylow of the central product 
$\calc_1\cdots\calc_m$. 

Assume $\cald\snsg\calf$ is another component. By 
\cite[9.8--9.9]{A-gfit}, the components of $\calf$ commute with each 
other and with $O_p(\calf)$. Hence $\cald\le C_\calf(QT)$, where 
$C_\calf(QT)$ is the fusion system of $C_G(QT)$ by Lemma \ref{l:NG(Q)}. 
(Note that $QT$ is fully centralized in $\calf$ since it is normal in $S$.)

By Lemma \ref{l:CG(QT)}, the centralizer $C_G(QT)$ is $p$-solvable. 
Hence $C_\calf(QT)=\calf_{C_S(QT)}(C_G(QT))$ is solvable in the sense of 
\cite[Definition II.12.1]{AKO}, and its saturated fusion subsystems are 
all solvable by \cite[Lemma II.12.8]{AKO}. So $\cald$ is solvable, and 
hence is constrained by \cite[Lemma II.12.5(b)]{AKO} (see also 
\cite[Definition I.4.8]{AKO}), which is impossible since $\cald$ was 
assumed to be quasisimple. We conclude that $\calc_1,\dots,\calc_m$ are 
the only components of $\calf$.
\end{proof}

In the next section, we will be working mostly with fusion systems that are 
(tamely) realized by finite $p'$-reduced $\KC$-groups. It will be important 
to know that in such situations, the fusion system and the group always 
have the same center. The following technical lemma is needed to prove 
that.

\begin{Lem} \label{l:ZF=ZG}
Fix an odd prime $p$. Let $G$ be a central product of known $p'$-reduced 
quasisimple groups, choose $S\in\sylp{G}$, and set $\calf=\calf_S(G)$. Then 
$Z(\calf)=Z(G)$, and $\Ker(\kappa_G)$ (see Definition \ref{d:kappaG}) has 
order prime to $p$.
\end{Lem}

\begin{proof} Assume the lemma holds for $G/Z(G)$. Then 
$Z(\calf/Z(G))=Z(G/Z(G))=1$, so $Z(\calf)\le Z(G)$, while the opposite 
inclusion holds since $G$ is $p'$-reduced. Also, $\Ker(\kappa_G)$ has order 
prime to $p$ by \cite[Lemma 2.17]{AOV1}. 

It thus suffices to prove the lemma when $G$ is a product of known simple 
groups. 

\smallskip

\noindent\boldd{$p\nmid|\Ker(\kappa_G)|$: } We first claim that if $G$ is a 
finite group such that $Z(G)=1$, and $U\in\sylp{\Aut(G)}$ and 
$S=U\cap\Inn(G)\in\sylp{\Inn(G)}$, then 
	\beqq C_U(S)\le S \quad\implies\quad p\nmid|\Ker(\kappa_G)|. 
	\label{e:CU(S)} \eeqq
To simplify notation, we identify $G$ with $\Inn(G)$ (recall $Z(G)=1$), and 
thus identify $S$ with a Sylow $p$-subgroup of $G$. Assume \eqref{e:CU(S)} 
does not hold: thus $C_U(S)\le S$ and $p\mid|\Ker(\kappa_G)|$. Let 
$\alpha\in\Aut(G)$ be such that $[\alpha]$ has order $p$ in $\Out(G)$ and 
$\kappa_G([\alpha])=1$. Since $\Aut(G)=\Inn(G)N_{\Aut(G)}(S)$ by the 
Frattini argument, we can assume that $\alpha(S)=S$ without changing the 
class $[\alpha]$. Then $\kappa_G([\alpha])=1$ implies in particular that 
$\alpha|_S\in\Aut_G(S)$, and thus $\alpha|_S=c_x|_S$ for some $x\in 
N_G(S)$. So upon replacing $\alpha$ by $c_x^{-1}\alpha$, we can assume that 
$\alpha$ centralizes $S$, and upon replacing $\alpha$ by $\alpha^k$ for some 
appropriate $k$, we can also arrange that $\alpha$ have $p$-power 
order in $\Aut(G)$. Also, $\alpha\in N_{\Aut(G)}(S)$ and 
$U\in\sylp{N_{\Aut(G)}(S)}$ imply that $\beta\alpha\beta^{-1}\in 
C_U(S)\sminus S$ for some $\beta\in N_{\Aut(G)}(S)$. This 
contradicts our original assumption, and finishes the proof of \eqref{e:CU(S)}. 

By \cite[Theorem B]{Gross}, $C_U(S)\le S$ whenever $G$ is a known simple 
group. The corresponding relation for products of known simple groups then 
follows from the description in Proposition \ref{p:Out(G1x...xGk)} of the 
automorphism group of a product. So $p\nmid|\Ker(\kappa_G)|$ for all such 
$G$ by \eqref{e:CU(S)}.

\smallskip

\noindent\boldd{$Z(\calf)=Z(G)$: } If $S\nnsg\calf$, then by Theorem 
\ref{OR-ThA}, either $G$ is as in case (b) and $Z(\calf)=Z(G)=1$; or it is 
as in case (c) or (d) and $Z(O^{p'}(\calf))=Z(G)=1$. Since $Z(\calf)\le 
Z(O^{p'}(\calf))$ for every saturated fusion system over a finite 
$p$-group, this proves that $Z(\calf)=Z(G)=1$.

Now assume that $S\nsg\calf$; i.e., that $G$ is $p$-Goldschmidt in the 
terminology of Aschbacher. Then $Z(\calf)=C_{Z(S)}(N_G(S)/S)$, and we must 
show this is trivial in all cases (recall $G$ is simple). If $S$ is 
abelian, then since $\Aut_G(S)\cong N_G(S)/S$ has order prime to $p$, we 
have $S=C_S(\Aut_G(S))\times[\Aut_G(S),S]$ (see \cite[Theorem 
5.2.3]{Gorenstein}), where the second factor is the focal subgroup 
$\foc(\calf)$. Also, $\foc(\calf)=S$ since $G$ is simple, and thus 
$Z(\calf)=C_S(\Aut_G(S))=1$. 

If $S\nsg\calf$ and $S$ is nonabelian, then by 
\cite[Theorem 15.6]{A-gfit}, $(G,p)$ is one of pairs listed in Table 
\ref{tbl:NL(T)}.
\begin{table}[ht]
\[ \renewcommand{\arraystretch}{1.2}
\newcommand{\OK}{\textup{OK}}
\begin{array}{|c|c||c|c|c|c|} 
p & G & S & N_G(S)/S \\\hline
p & \PSU_3(q) ~(q=p^k) & p^{k+2k} & C_{(q^2-1)[/3]} \\ 
3 & \lie2G2(q)~(q=3^k) & 3^{k+k+k} & C_{q-1} \\
3 & G_2(q) ~(q\equiv\pm2,\pm4\!\!\pmod9) & 3^{1+2}_+ & \SD_{16} \\
3 & J_2 & 3^{1+2}_+ & C_8 \\
3 & J_3 & \textup{order $3^5$} & C_8 \\
5 & \McL & 5^{1+2}_+ & C_3\rtimes C_8 \\
5 & \HS & 5^{1+2}_+ & \textup{order 16} \\
5 & \Co_2 & 5^{1+2}_+ & 4S_4 \\
5 & \Co_3 & 5^{1+2}_+ & \textup{order 48} \\
11 & J_4 & 11^{1+2}_+ & 5\times2S_4 \\
\end{array} \]
\caption{} 
\label{tbl:NL(T)} 
\end{table}
In all cases, $C_S(N_G(S))\le[S,S]$ by an argument similar to that used 
when $S$ is abelian. If $G\cong\PSU_3(q)$ or $\lie2G2(q)$, then the explicit 
description of $S$ and $N_G(S)/S$ in \cite[II.10.12(b)]{Huppert} or 
\cite[\S XI.3]{HB3}, respectively, shows that $C_S(\Aut_G(S))=1$. 
When $p=3$ and $G\cong G_2(q)$ for $q\equiv\pm2,\pm4$ (mod $9$), the 
action of the Weyl group $W\cong D_{12}$ on $Z(S)\cong C_3$ is nontrivial. 
In all other cases where $S\cong p^{1+2}_+$, the generators of $Z(S)$ are 
conjugate (by \cite[\S5]{GL} or \cite[Tables 5.3]{GLS3}), and so $N(S)/S$ 
acts nontrivially on $Z(S)=[S,S]$. 

This leaves the case $G\cong J_3$ and $p=3$. By \cite[Lemma 5.4]{Janko} 
(where $S$ is denoted $W_1$), $Z(S)\cong E_9$, and there is a subgroup 
$W\nsg N_G(S)$ such that $Z(S)\le W\cong E_{27}$ and $S\nsg N_G(W)$. Also, 
$N_G(W)/S\cong C_8$; and all elements in $Z(S)^\#$ and all those in 
$W\sminus Z(S)$ are conjugate in $N_G(W)$. Thus $W\ge[S,S]$, and 
$C_S(N_G(S))=1$ in this case.
\end{proof}

By \cite[Theorem 5.1]{GLynd2} and Glauberman's $Z^*$-theorem, Lemma 
\ref{l:ZF=ZG} also holds when $p=2$. More generally, in the same theorem, 
Glauberman and Lynd showed that for each prime $p$ for which the 
$Z^*$-theorem holds for all almost simple groups, one also has that 
$\Ker(\kappa_G)$ has order prime to $p$ for all simple groups. 

We are now ready to prove the $Z^*$-theorem at odd primes for all finite 
$p'$-reduced $\KC$-groups. 

\begin{Prop} \label{p:ZF=ZG}
Fix a prime $p$, let $G$ be a $p'$-reduced $\KC$-group, and choose 
$S\in\sylp{G}$. Then $Z(\calf_S(G))=Z(G)$. 
\end{Prop}

\begin{proof} When $p=2$, this is just Glauberman's $Z^*$-theorem 
\cite[Corollary 1]{Glauberman}, and holds for all finite $2'$-reduced 
groups. So it remains to prove the proposition when $p$ is odd. 

Set $H=E(G)$: the central product of the components of $G$. Set $T=S\cap 
H$, and set $\calf=\calf_S(G)$ and $\cale=\calf_T(H)$. By Lemma 
\ref{l:ZF=ZG}, $\Ker(\kappa_{H})$ has order prime to $p$ and $Z(\cale)=Z(H)$. 
Also, $C_G(O_p(G)H)=Z(O_p(G)H)=Z(O_p(G))$ (see \cite[31.13]{A-FGT}), so 
$O_p(G)\nsg C_G(H)$ is a centric subgroup. In particular, $Z(C_G(H))\le 
Z(O_p(G))$.

Assume $x\in Z(\calf)$. By Lemma \ref{l:Ker(kappa)}, there are $y\in 
C_S(H)$ and $z\in Z(\cale)=Z(H)$ such that $x=yz$. If $g\in C_G(H)$ is such 
that $\9gy\in C_S(H)$, then $\9gx=(\9gy)z\in S$, so $\9gx=x$ since $x\in 
Z(\calf)$ and hence $\9gy=y$. Thus 
	\[ y\in Z(C_G(H)) \le Z(O_p(G)), \] 
and since $z\in Z(H)\le Z(O_p(G))$, we have $x\in Z(O_p(G))$. Since $x\in 
Z(\calf)$, it must be invariant under the action of $\autf(Z(O_p(G)))$, and 
so $x\in C_{Z(O_p(G))}(G)=Z(G)$. 

This proves that $Z(\calf)\le Z(G)$, and the opposite inclusion 
holds since $G$ is $p'$-reduced.
\end{proof}

A very similar proof of the $Z^*$-theorem at odd primes was given by 
Guralnick and Robinson \cite[Theorem 4.1]{GR}. If we give our own proof 
here, it is mostly to make our assumptions completely clear: we assume only 
that all components of $G$ are known quasisimple groups. Other proofs of 
the $Z^*$-theorem, where the analogous assumptions are less restrictive or 
less clear, are given in \cite[Theorem 1]{Xiao} and in \cite[Remark 
7.8.3]{GLS3}.

The next proposition is somewhat more technical. 

\begin{Prop} \label{p:Comp(F)-2}
Let $G$ be a finite $p'$-reduced $\KC$-group, assume Notation 
\ref{h:Comp(F)}, and set 
	\[ L=L_1\cdots L_m,\quad T=T_1\cdots T_m, \quad 
	\cale=\calf_T(L), \quad\textup{and}\quad \calc=O^{p'}(\cale). \]
Then $\calc\nsg\cale$ are the central products of the subsystems 
$\calc_i\nsg\cale_i$ for $1\le i\le m$. 
\begin{enuma} 

\item There is a unique minimal normal fusion subsystem $\calc^*\nsg\calf$ 
among those subsystems containing $\calc=E(\calf)$ that are realized by 
finite $p'$-reduced $\KC$-groups, and $\calc^*$ itself is realized by a 
central product of known finite quasisimple groups. More precisely, 
$\calc^*=\calc_1^*\cdots\calc_m^*$, where for each $1\le i\le m$, 
$\calc_i^*\le\cale_i$ is a fusion system over $T_i$ such that 
$O^{p'}(\calc_i^*)=\calc_i$, and $\calc_i^*$ is the fusion system of a 
known finite quasisimple group. 

\item We have $\calc=O^{p'}(\calc^*)$ and $\Aut(\calc^*)=\Aut(\calc)$, and 
$\calc$ and $\calc^*$ are both characteristic in $\calf$.

\end{enuma}
\end{Prop}

\begin{proof} Let $\varphi\:L_1\times\cdots\times L_m \too G$ be the 
homomorphism induced by the inclusions $L_i\le G$, and let 
	\[ \5\varphi\: \cale_1\times\cdots\times\cale_m = 
	\calf_{T_1\times\cdots\times T_m}(L_1\times\cdots\times L_m) 
	\Right3{} \calf_S(G) = \calf \]
be the induced functor between the fusion systems. Then 
$\Im(\5\varphi)=\cale_1\cdots\cale_m$, the central product of the 
$\cale_i$, and is equal to $\calf_{T_1\cdots T_m}(L_1\cdots L_m)$. Set 
$Z=\Ker(\varphi)\le\prod_{i=1}^mZ(L_i)$. Then 
$\cale_1\cdots\cale_m\cong(\prod_{i=1}^m\cale_i)/Z$, and hence 
	\[ O^{p'}(\cale_1\cdots\cale_m) \cong 
	O^{p'}(\textstyle\prod\nolimits_{i=1}^m\cale_i)/Z \cong 
	(\textstyle\prod\nolimits_{i=1}^mO^{p'}(\cale_i))/Z \cong 
	O^{p'}(\cale_1)\cdots O^{p'}(\cale_m): \]
the first isomorphism by Lemma \ref{l:Op'(F/Z)} and the second by 
\cite[Proposition 3.4]{AOV1}. Since these are finite categories and 
$O^{p'}(\cale_1\cdots\cale_m)\le O^{p'}(\cale_1)\cdots O^{p'}(\cale_m)$, 
the two are equal.

\smallskip

\noindent\textbf{(a) } For each $1\le i\le m$, let $\calc_i^*\le\cale_i$ be 
the minimal realizable fusion subsystem containing $\calc_i$ of Theorem 
\ref{OR-ThA}. More precisely, $\calc_i^*$ is realized by a known finite 
quasisimple group, and each saturated fusion subsystem of $\cale_i$ that is 
realized by a known finite quasisimple group and contains $\calc_i$ also 
contains $\calc_i^*$. Since $\calc_i^*\le\cale_i$ for each $i$ and the 
$\cale_i$ commute, the $\calc_i^*$ also commute.

Set $\calc^*=\calc_1^*\cdots\calc_m^*\le\calf$, the central product of 
these subsystems. Thus $\calc\le\calc^*\le\cale$ where all three fusion 
systems are over $T=T_1\cdots T_m$, and where $\calc=E(\calf)\nsg\calf$, 
and $\cale\nsg\calf$ since $L_1\cdots L_m\nsg G$. 

It remains to show that each normal fusion subsystem of $\calf$ containing 
$\calc$ and realized by a finite $p'$-reduced $\KC$-group 
also contains $\calc^*$. Assume that $H$ is such a 
group. In particular, there is $R\in\sylp{H}$ such that $T\le R\nsg S$ and 
$\calc\le\calf_R(H)\nsg\calf$. Let $H_1,\dots,H_\ell\snsg H$ be the 
components of $H$ (by assumption, known quasisimple groups), and set 
$R_i=R\cap H_i$. By 
Proposition \ref{p:Comp(F)-1}, the components of $\calf_R(H)$ are those 
subsystems $O^{p'}(\calf_{R_i}(H_i))$ for $1\le i\le \ell$ that are 
quasisimple. Each component of $\calf_R(H)$ is subnormal in $\calf$ and 
hence a component of $\calf$ (recall $\calf_R(H)\nsg\calf$), and each 
component of $\calf$ is contained in and hence a component of $\calf_R(H)$ 
by assumption. Thus $m\le\ell$, and we can assume that the indices are 
chosen so that $T_i=R_i$ and $O^{p'}(\calf_{T_i}(H_i))=\calc_i$ for each 
$i\le m$.

For each $i$, $\calf_{T_i}(H_i)$ and $\calc_i^*$ are both realizable fusion 
systems over $T_i$ where 
$O^{p'}(\calf_{T_i}(H_i))=O^{p'}(\calc_i^*)=\calc_i$, and so 
$\calf_{T_i}(H_i)\ge\calc_i^*$ by Theorem \ref{OR-ThA} (the last statement) 
and the minimality assumption. So $\calf_R(H)$ contains 
$\calc^*$, and thus $\calc^*$ is minimal among normal subsystems of $\calf$ 
containing $\calc$ and realized by finite $p'$-reduced $\KC$-groups.

\smallskip

\noindent\textbf{(b) } Since $\calc\le\calc^*\le\cale$ and 
$\calc=O^{p'}(\cale)$, we also have $\calc=O^{p'}(\calc^*)$. It remains to 
prove that $\Aut(\calc^*)=\Aut(\calc)$, and that $\calc$ and $\calc^*$ are 
characteristic in $\calf$. 

Assume the central factors $\calc_i$ are ordered so 
that for some $\ell\le m$, we have $\calc_i<\calc_i^*$ if and only if $1\le 
i\le\ell$. For each $1\le i\le\ell$, $L_i$ falls under case (c) in Theorem 
\ref{OR-ThA}, so $\calc_i$ is simple, and $Z(\calc_i^*)=1$. So if we set 
$T_0=T_{\ell+1}\cdots T_m$ and 
$\calc_0=\calc_{\ell+1}\cdots\calc_m$, the central products of the remaining 
factors, we get direct product decompositions 
	\[ T=T_0\times\cdots\times T_\ell, \qquad
	\calc=\calc_0\times\cdots\times\calc_\ell 
	\qquad\textup{and}\qquad 
	\calc^*=\calc_0^*\times\cdots\times\calc_\ell^*. \]

Since $\calc=O^{p'}(\calc^*)$, we have $\Aut(\calc^*)\le\Aut(\calc)$, and 
it remains to prove the opposite inclusion.
Fix $\alpha\in\Aut(\calc)\le\Aut(T)$. Then $\alpha(T_0)=T_0$ and 
$\9\alpha\calc_0=\calc_0$, and $\alpha$ permutes the factors $T_i$ and 
$\calc_i$ for $1\le i\le\ell$. Thus there is $\sigma\in\Sigma_\ell$ such 
that $\9\alpha\calc_i^*\ge\9\alpha\calc_i=\calc_{\sigma(i)}$ for each $1\le 
i\le\ell$. For each $i$, $\calc_i^*$ and $\calc_{\sigma(i)}^*$ were chosen 
to be the unique smallest saturated fusion systems over $T_i$ and 
$T_{\sigma(i)}$ containing $\calc_i$ and $\calc_{\sigma(i)}$, respectively, 
that are realized by known finite quasisimple groups. Since 
$\9\alpha\calc_i^*$ is also realized by a known quasisimple group, we have 
$\9\alpha\calc_i^*=\calc_{\sigma(i)}^*$. 
Upon taking the direct product of these systems, it now follows that 
$\9\alpha\calc^*=\calc^*$. Hence $\alpha\in\Aut(\calc^*)$, finishing the 
proof that $\Aut(\calc)=\Aut(\calc^*)$. 

Now, $\calc=E(\calf)$ since $\Comp(\calf)=\{\calc_1,\dots,\calc_m\}$ by 
Proposition \ref{p:Comp(F)-1}, and hence $\calc$ is characteristic in 
$\calf$ by Lemma \ref{l:A2-Th6}(a). The Frattini condition for 
$\calc^*\le\calf$ holds since it holds for $\calc\nsg\calf$, and the 
extension condition holds since it holds for $\cale\nsg\calf$. For each 
$\alpha\in\Aut(\calf)$, $\alpha|_T\in\Aut(\calc)=\Aut(\calc^*)$ since 
$\calc$ is characteristic in $\calf$, and hence $\9\alpha\calc^*=\calc^*$. 
Thus the invariance condition holds (so $\calc^*\nsg\calf$), and 
$\calc^*$ is characteristic in $\calf$. 
\end{proof}

We need to understand the role played by the components of $G$ and of 
$\calf_S(G)$ when determining automorphisms of the linking system 
$\call_S^c(G)$ and tameness. The next proposition is a first step towards 
that. 

\begin{Prop} \label{p:Comp(F)-3}
Let $G$ be a finite $p'$-reduced $\KC$-group, and assume Notation 
\ref{h:Comp(F)}. Assume also that $O_p(G)=1=O_p(\calf)$. 
Thus for each $1\le i\le k$, $L_i$ is a known simple group and $\calc_i$ is 
a simple fusion system. Set 
	\[ L=\xxx{L}\nsg G \qquad\textup{and}\qquad 
	T=\xxx{T}=S\cap L\in\sylp{L}; \]
and also $\cale=\calf_T(L)$ and $\calc=O^{p'}(\cale)$. 
\begin{enuma} 

\item There is a unique minimal normal fusion subsystem $\calc^*\le\cale$ 
containing $\calc$ that is realized by a product of known finite simple 
groups. Also, $\calc^*=\calc_1^*\times\dots\times\calc_k^*$, where for each 
$1\le i\le k$, the subsystem $\calc_i^*$ is realized by a known finite 
simple group and $\calc_i\le\calc_i^*\le\cale_i$. 

\item The fusion subsystems $\calc^*$ and $\calc$ are both centric and 
characteristic in $\calf$.

\end{enuma}
\end{Prop}

\begin{proof} \textbf{(a) } This is the special case of Proposition 
\ref{p:Comp(F)-2} when $O_p(G)=1=O_p(\calf)$.

\smallskip

\noindent\textbf{(b) } The subsystems $\calc$ and $\calc^*$ are 
characteristic in $\calf$ by Proposition \ref{p:Comp(F)-2}. 

Since $O_p(\calf)=1=O_p(G)$, each of the subsystems $\calc_i$ is simple, 
and hence $\Comp(\calf)=\{\calc_1,\dots,\calc_k\}$ by Proposition 
\ref{p:Comp(F)-1}. So $F^*(\calf)=E(\calf)=\calc$, and $\calc$ is centric 
in $\calf$ by Lemma \ref{l:A2-Th6}(b). So 
$\calc^*\ge\calc$ is also centric in $\calf$. 
\end{proof}

We finish the section with two more specialized results. The first shows 
that in a saturated fusion system $\calf$ where $E(\calf)$ is ``almost 
realizable'' in the sense that there is a minimal realizable fusion 
subsystem $\calc^*\le\calf$ with $O^{p'}(\calc^*)=E(\calf)$, the subsystem 
$\calc^*$ is always contained and normal in $C_\calf(O_p(\calf))$.

\begin{Lem} \label{l:C*nsgCF(Q)}
Let $\calf$ be a saturated fusion system over a finite $p$-group $S$. Let 
$\calc_1,\dots,\calc_m$ be its components, where each 
$\calc_i\nsg\nsg\calf$ is a fusion system over $T_i\nsg\nsg S$. Set 
$T=T_1\cdots T_m$ and $\calc=\calc_1\cdots\calc_m$. Set $Q=O_p(\calf)$. 

Assume, for each $i=1,\dots,m$, that $\calc_i^*\le\calf$ is a saturated 
fusion subsystem over $T_i$ containing $\calc_i$ such that 
$\calc_i=O^{p'}(\calc_i^*)$, such that $\calc_i^*$ is realized by a known 
finite quasisimple group, and such that $\calc_i^*=\calc_i$ for each $i$ 
such that $\calc_i$ is realized by a known finite quasisimple group. Assume 
also that the subsystems $\calc_1^*,\dots,\calc_m^*$ commute in $\calf$, 
and that their central product $\calc^*=\calc_1^*\cdots\calc_m^*$ is normal 
in $\calf$. Then $\calc^*$ is contained in $C_\calf(Q)$, and is normal in 
$C_\calf(Q)$ and in $N_\calf^{\Inn(Q)}(Q)$. 
\end{Lem}

\begin{proof} Assume that the indices are chosen so that for some 
$0\le\ell\le m$, $\calc_i^*=\calc_i$ for each $1\le i\le\ell$ and 
$\calc_i^*>\calc_i$ for each $\ell+1\le i\le m$. For each $1\le i\le m$, 
let $L_i$ be a known finite quasisimple group with $T_i\in\sylp{L_i}$ such 
that $\calc_i^*=\calf_{T_i}(L_i)$ (there is such a group by assumption). By 
Theorem \ref{OR-ThA} applied to $L_i$, we are in case (d) of the theorem 
whenever $i\le\ell$ and in case (c) whenever $i\ge\ell+1$. (Case (b) cannot 
occur since we assume $\calc_i=O^{p'}(\calc^*_i)=\calc_i^*$ whenever 
$\calc_i$ is realizable by a known finite quasisimple group.) In 
particular, the $\calc_i$ and $L_i$ are all simple for $\ell+1\le i\le m$ 
by Theorem \ref{OR-ThA}(c).

Set 
	\[ \calc_{I}=\calc_1\cdots\calc_\ell=\calc_1^*\cdots\calc_\ell^*, 
	\qquad \calc_{II}=\calc_{\ell+1}\cdots\calc_m, 
	\qquad\textup{and}\qquad
	\calc_{II}^*=\calc_{\ell+1}^*\cdots\calc_m^*, \]
and also $T_{I}=T_1\cdots T_\ell$ and $T_{II}=T_{\ell+1}\cdots T_m$. Since 
$\calc_i$ is simple for $i\ge\ell+1$, we have 
	\[ Z(\calc_{II})=1, \qquad \calc=\calc_{I}\times\calc_{II}, \qquad
	\calc^*=\calc_{I}\times\calc_{II}^*,\qquad\textup{and}\qquad 
	T=T_{I}\times T_{II}. \] 

By \cite[9.9]{A-gfit} or \cite[Theorem 7.10(e)]{CH}, we have $\calc\nsg 
C_\calf(Q)$. By the Frattini condition for $\calc\nsg\calc^*$, the fusion 
system $\calc^*$ is 
generated by $\calc$ and automorphisms $\alpha\in\Aut_{\calc^*}(T)$ of 
order prime to $p$ that are the identity on $T_{I}$. By the extension 
condition for $\calc^*\nsg\calf$, each such $\alpha$ extends to an element 
$\4\alpha\in\autf(TC_S(T))$ such that $[\4\alpha,C_S(T)]\le Z(T)$, and since 
$Q\nsg\calf$, this implies that $[\4\alpha,Q]\le Q\cap Z(T)=Z(\calc)\le 
T_{I}$. Upon replacing $\4\alpha$ by $\4\alpha^k$ for some appropriate $k$, 
we can arrange that $\4\alpha$ have order prime to $p$. Then 
	\[ Q=C_{Q}(\4\alpha)[\4\alpha,Q] \le C_Q(\4\alpha) T_{I}\le 
	C_{QT}(\4\alpha) \]
(see \cite[Theorem 5.3.5]{Gorenstein} for the first equality), and so 
$\4\alpha|_Q=\Id_Q$. Thus $\alpha$ is a morphism in 
$C_\calf(Q)$, finishing the proof that $\calc^*\le C_\calf(Q)$. 

It remains to prove that $\calc^*$ is normal in  $C_\calf(Q)$ and in 
$N_\calf^{\Inn(Q)}(Q)$. The Frattini condition holds since it holds for 
$\calc\nsg C_\calf(Q)$ and $\calc\nsg N_\calf^{\Inn(Q)}(Q)$, and the 
invariance condition holds for both inclusions since it holds for 
$\calc^*\nsg\calf$. We just showed that $\calc^*$ is generated by morphisms 
in $\calc$ and morphisms $\alpha\in\Aut_{\calc^*}(T)$ that extend to 
$\4\alpha\in\autf(TC_S(T))$ such that $\4\alpha|_Q=\Id_Q$ and 
$[\4\alpha,C_S(T)]\le Z(T)$, and hence the extension condition holds for 
both inclusions since $\calc\nsg C_\calf(Q)$ and $\calc\nsg 
N_\calf^{\Inn(Q)}(Q)$. 
\end{proof}

In the following proposition, we show that each saturated fusion system 
$\calf$ has a maximal characteristic subsystem that normalizes all 
components of $\calf$.

\begin{Prop} \label{p:NF(EJ)}
Let $\calf$ be a saturated fusion system over a finite group $S$. 
Let $\calc_1,\dots,\calc_k$ be the components of $\calf$, where 
$\calc_i$ is a fusion system over $U_i$. Assume, for each $1\le i\le k$, 
that $Z(\calc_i)=1$ (i.e., that $\calc_i$ is simple). Set 
	\begin{align*} 
	U &= \xxx{U} \le S \,, \qquad 
	N = \bigcap\nolimits_{i=1}^k N_S(U_i) \,, \\
	\calh &= \{ P\le N \,|\, P\cap U_i\ne1 ~\textup{for each 
	$1\le i\le k$} \}\,, \qquad\textup{and} \\
	\caln &= \Gen{ \varphi\in\homf(P,Q) \,\big|\, 
	P,Q\in\calh,~ \varphi(P\cap U_i)\le Q\cap U_i ~\textup{for 
	each $1\le i\le k$}}_{N}.
	\end{align*}
Thus $E(\calf)$ is the direct product of the $\calc_i$, and is a fusion 
subsystem over $U$ by Lemma \ref{l:A2-Th6}(a). 
Then the fusion subsystem $\caln$ is saturated and characteristic in 
$\calf$, $E(\calf)\nsg \caln$, and $\calc_i\nsg\caln$ for each $1\le 
i\le k$. 
\end{Prop}

\begin{proof} By Lemma \ref{l:A2-Th6}(a) and 
since $Z(\calc_i)=1$ for each $i$, $E(\calf)$ is the direct product of the 
$\calc_i$ and is characteristic in $\calf$. In particular, 
$E(\calf)\nsg\calf$, and $U$ is strongly closed in $\calf$.

Set 
	\[ \Delta=\{\delta\in\Sigma_k \,|\, \exists\,\alpha\in\autf(U) 
	~\textup{such that}~ \alpha(U_i)=U_{\delta(i)}~ \forall\,1\le i\le 
	k \}. \] 
We first claim that 
	\beqq \parbox{130mm}{for each $P\in\calh$ and each 
	$\varphi\in\homf(P,S)$, there is $\delta\in\Delta$ such that 
	$\varphi(P\cap U_i)\le U_{\delta(i)}$ for all $1\le i\le k$.}
	\label{e:yy1} \eeqq
It suffices to prove this when $P\le U$ (hence $\varphi(P)\le U$). 
Since $E(\calf)\nsg\calf$ by Lemma \ref{l:A2-Th6}(a), the Frattini condition 
implies that $\varphi=\alpha\varphi'$ for some $\alpha\in\autf(U)$ and 
some $\varphi'\in\Hom_{E(\calf)}(P,U)$. Since $\alpha$ permutes the 
components of $\calf$, there is $\delta\in\Delta$ such that 
$\alpha(U_i)=U_{\delta(i)}$ for each $i$. Since $\varphi'$ is in 
$E(\calf)$, we have $\varphi(P\cap U_i)\le \alpha(U_i) = U_{\delta(i)}$ 
for each $1\le i\le k$.

We next claim that 
	\beqq \textup{for each $\delta\in\Delta$, there is 
	$\alpha\in\autf(N)$ such that $\alpha(U_i)=U_{\delta(i)}$ for 
	each $1\le i\le k$.} \label{e:yy2} \eeqq
To see this, fix $\delta\in\Delta$, and choose $\beta\in\autf(U)$ such 
that $\beta(U_i)=U_{\delta(i)}$ for each $1\le i\le k$. Since 
$\Aut_S(U)\in\sylp{\autf(U)}$ and $\Aut_{\caln}(U)\nsg\autf(U)$, 
we have that $\Aut_{N}(U)$ and 
$\9\beta\Aut_{N}(U)$ are both Sylow $p$-subgroups of 
$\Aut_{\caln}(U)$. So there is $\gamma\in\Aut_{\caln}(U)$ 
such that $\gamma\beta$ normalizes $\Aut_{N}(U)$, and hence by the 
extension axiom extends to $\alpha\in\autf(N)$. By construction, 
$\alpha(U_i)=U_{\delta(i)}$ for each $1\le i\le k$.

Fix $\alpha\in\autf(N)$, and let $\delta\in\Sigma_k$ be such that 
$\alpha(U_i)=U_{\delta(i)}$ for all $1\le i\le k$. For each $P,Q\in\calh$ 
and $\varphi\in\Hom_{\caln}(P,Q)$, $\varphi(P\cap U_i)\le 
Q\cap U_i$ for each $1\le i\le k$, and hence 
$\alpha\varphi\alpha^{-1}(\alpha(P)\cap U_i)\le \alpha(Q)\cap U_i$ for 
each $1\le i\le k$. Thus $\9\alpha\varphi\in\Mor(\caln)$. So 
$\alpha$ normalizes the subsystem $\caln$, and we have shown
	\beqq \textup{for all $\alpha\in\autf(N)$,\quad 
	$\9\alpha\caln=\caln$.} \label{e:yy3} \eeqq

We show in Step 1 that $\caln$ is saturated, in Step 2 that $\caln$ 
is characteristic, and in Step 3 that $E(\calf)\nsg \caln$ and 
$\calc_i\nsg\caln$ for $1\le i\le k$. 

\smallskip

\noindent\textbf{Step 1: } For each $1\le i\le k$, since $U_i\nsg N$, we 
have $Z(N)\cap U_i\ne1$. So for each $P\in \caln^c$, we have 
$P\cap U_i\ge Z(N)\cap U_i\ne1$ for each $1\le i\le k$, and hence 
$P\in\calh$. Thus $\caln^c\subseteq\calh$.

By definition, $\caln$ is $\calh$-generated. So by \cite[Theorem 
I.3.10]{AKO}, to prove that $\caln$ is saturated, it suffices to 
prove that it is $\calh$-saturated; i.e., that each $P\in\calh$ is 
$\caln$-conjugate to a subgroup that is fully automized and 
receptive in $\caln$ (see \cite[Definition I.3.9]{AKO}).

If $P\in\calh$ is receptive in $\calf$ and 
$\varphi\in\Iso_{\caln}(Q,P)$ for some $Q\le N$, then $\varphi$ 
extends to some $\4\varphi\in\homf(N_\varphi^\calf,S)$, and 
$\4\varphi(N_\varphi^\calf\cap U_i)\le U_i$ for each $1\le i\le k$ by 
\eqref{e:yy1}. Hence $\4\varphi$ restricts to an element of 
$\Hom_{\caln}(N_\varphi^{\caln},N)$. Thus $P$ is receptive 
in $\caln$. 

Assume $P\in\calh$ is fully automized in $\calf$. By \eqref{e:yy1}, each 
$\beta\in\autf(P)$ permutes the subgroups $P\cap U_i$ for $1\le i\le k$, 
while $\beta\in\Aut_{\caln}(P)$ if and only if it sends each $P\cap U_i$ 
to itself. So $\Aut_{\caln}(P)$ is normal in $\autf(P)$. Also, 
$\Aut_{N}(P)=\Aut_S(P)\cap\Aut_{\caln}(P)$: if 
$c_x\in\Aut_{\caln}(P)$ for $x\in N_S(P)$, then $\9x(P\cap U_i)\le U_i$ 
for each $1\le i\le k$ and hence $x\in N$. So 
$\Aut_{N}(P)\in\sylp{\Aut_{\caln}(P)}$ since 
$\Aut_S(P)\in\sylp{\autf(P)}$, and we conclude that $P$ is fully automized 
in $\caln$.

Now fix $P\in\calh$, and let $\chi\in\homf(P,N)$ be such that $\chi(P)$ is 
fully normalized in $\calf$. Then $\chi(P)\in\calh$, and we just showed 
that $\chi(P)$ is fully automized and receptive in $\caln$. By 
\eqref{e:yy1} and \eqref{e:yy2}, there is $\alpha\in\autf(N)$ such that 
$\alpha\chi\in\Hom_{\caln}(P,N)$. Since $\9\alpha\caln=\caln$ by 
\eqref{e:yy3}, the subgroup $\alpha\chi(P)$ is also fully automized and 
receptive in $\caln$, and is $\caln$-conjugate to $P$. Since $P\in\calh$ 
was arbitrary, this proves that $\caln$ is $\calh$-saturated, and finishes 
the proof that it is saturated. 

\smallskip

\noindent\textbf{Step 2: } We first check that $N$ is 
strongly closed in $\calf$. Set 
	\[ K=\{\alpha\in\autf(U)\,|\,\alpha(U_i)=U_i,~\textup{all $1\le i\le 
	k$}\}\le\autf(U). \]
Thus $N=N_S^K(U)$. 
Let $x,y\in S$ be such that $x\in N$ and $y\in x^\calf$; we claim 
that $y\in N$.

Let $P,Q\le S$ and $\varphi\in\homf(P,Q)$ be such that $x\in P$, $y\in 
Q$, and $\varphi(x)=y$. Then $\varphi(P\cap U)\le Q\cap U$ since $U$ is 
strongly closed in $\calf$ as noted above, and so $\varphi$ induces a 
homomorphism $\4\varphi$ from $PU/U\cong P/(P\cap U)$ to $QU/U\cong 
Q/(Q\cap U)$. By a theorem of Puig (see \cite[Theorem 5.14]{Craven}), 
$\4\varphi\in\Hom_{\calf/U}(PU/U,QU/U)$. In other words, there is 
$\psi\in\homf(PU,QU)$ such that for each $g\in P$, 
$\psi(g)\in\varphi(g)U$. 

Now, $c_{\psi(x)}=(\psi|_U)c_x(\psi|_U)^{-1}$ where 
$\psi|_U\in\autf(U)$. Also, $c_x\in K$ since $x\in N=N_S^K(U)$, and 
$K$ is normal in $\autf(U)$ since each $\alpha\in\autf(U)$ permutes the 
$U_i$. So $c_{\psi(x)}\in K$, and hence $\psi(x)\in N$. Hence $y\in 
\psi(x)U\subseteq N$, finishing the proof that $N$ is strongly 
closed.

If $x\in C_S(N)\le C_S(U)$, then $\9xU_i=U_i$ for each $1\le i\le k$ 
(hence for each $1\le i\le k$), and so $x\in N$. Thus $C_S(N)\le N$, 
so the extension condition holds for $\caln\le\calf$. 

By \eqref{e:yy1} and \eqref{e:yy2}, for each $P,Q\le N$ and 
$\varphi\in\homf(P,Q)$, there is $\delta\in\Delta$ such that $\varphi(P\cap 
U_i)\le Q\cap U_{\delta(i)}$ for all $1\le i\le k$, and $\alpha\in\autf(N)$ 
such that $\alpha(U_i)=U_{\delta(i)}$ for each $1\le i\le k$. So 
$\alpha^{-1}\varphi\in\Hom_{\caln}(P,\alpha^{-1}(Q))$, and the Frattini 
condition for normality holds. The invariance condition holds by 
\eqref{e:yy3}, and thus $\caln\nsg\calf$.

For each $\beta\in\Aut(\calf)$, $\beta$ permutes the components of 
$\calf$, and hence permutes the subgroups $U_i$ and the members of the set 
$\calh$. So $c_\beta(\caln)=\caln$ by the above definition of 
$\caln$, and $\caln$ is characteristic in $\calf$.

\smallskip

\noindent\textbf{Step 3: } By \cite[9.8.3]{A-gfit} and since 
$\caln\nsg\calf$, we have $E(\calf)=E(\caln)\nsg\caln$. For each 
$1\le i\le k$, we have $\calc_i\nsg E(\calf)$ by \cite[9.8.2]{A-gfit}, and 
$c_\alpha(\calc_i)=\calc_i$ for each $\alpha\in\Aut_\caln(U)$ by 
definition of $\caln$. So $\calc_i\nsg\caln$ by Lemma 
\ref{l:D<E<F}.
\end{proof}

By construction, $\caln_{}$ is the largest saturated subsystem of 
$\calf$ that contains each of the $\calc_i$ for $1\le i\le k$ as a normal 
subsystem.

\bigskip

\section{Tameness of realizable fusion systems}
\label{s:tame-real}

We are now ready to show that realizable fusion systems are tame, assuming 
the classification of finite simple groups. This has already been shown in 
earlier papers for fusion systems of known simple groups (see Proposition 
\ref{p:simple=>tame}). When $\calf$ is the fusion system of an arbitrary 
finite $p'$-reduced $\KC$-group $G$, we will show that it is tame via a 
series of reductions based on an examination of the components of $G$. 

We first restrict attention to tameness of fusion systems of finite simple 
groups. This was shown in most cases in earlier papers, and will be 
summarized below, but there were two cases whose proofs assumed 
earlier results that were in error: 

\begin{Lem} \label{l:FS(L).tame}
Let $(G,p)$ be one of the pairs $(\He,3)$ or $(\Co_1,5)$, choose 
$S\in\sylp{G}$, and set $\calf=\calf_S(G)$ and $\call=\call_S^c(G)$. Then 
$\Out(\call)=1$, and so $\calf$ is tamely realized by $G$.
\end{Lem}

\begin{proof} Since $p$ is odd, $\Out(\calf)\cong\Out(\call)$ in both cases. 
The simplest proof of this is given in \cite[Theorem 
C]{limz-odd} (and the sporadic groups are handled in Proposition 4.4 of that 
paper). A more general result is shown in \cite[Theorem C]{O-Ch} and 
\cite[Theorem 1.1]{GLynd}. 

When $G=\He$ and $p=3$, the argument in \cite[p. 
139]{sportame} claimed (wrongly) that $\calf$ is simple, but 
did not actually use this. Since $S$ is extraspecial of order $27$ and 
exponent $3$ and $\Out_G(S)\cong D_8$, we have 
	\[ D_8 \cong \Out_G(S) \le \Aut(\calf)/\Inn(S) \le 
	N_{\Out(S)}(\Out_G(S)) \cong \SD_{16}. \]
Elements in $N_{\Aut(S)}(\Aut_G(S))\sminus\Aut_G(S)$ exchange subgroups of 
$S$ of order $9$ with non-isomorphic automizers, and hence do not normalize 
$\calf$. So $\Aut(\calf)=\Aut_G(S)$, and $\Out(\calf)=1$.

When $G=\Co_1$ and $p=5$, the proof that $\Out(\calf)=1$ in \cite[p. 
138]{sportame} used the incorrect claim that $\calf$ has a normal subsystem 
of index $2$. So we replace that argument with the following one. By 
\cite[Theorem 5.1]{Curtis} and the correction in \cite[p. 145]{Wilson-Co1}, 
$S$ contains a unique elementary abelian subgroup $Q$ of order $5^3$ and 
index $5$, and 
$N_G(Q)/Q\cong\Aut_G(Q)\cong C_4\times\Sigma_5$. Set $H=N_G(Q)$. By 
\cite[Lemma 1.2(b)]{sportame} and since $C_H(Q)=Q$, we have 
$|\Out(\calf)|\le|\Out(H)|$. By \cite[Lemma 1.2]{OV2}, there is an exact 
sequence 
	\[ 0 \Right2{} H^1(H/Q;Q) \Right4{} \Out(H) \Right4{} 
	N_{\Out(Q)}(\Out_H(Q))/\Out_H(Q), \]
and by \cite[p. 110]{Benson2}, there is a 5-term exact sequence for the 
homology of $H/Q$ as an extension of $C_4$ by $\Sigma_5$ that begins with 
	\[ 0 \Right2{} H^1(\Sigma_5;H^0(C_4;Q)) \Right4{} H^1(H/Q;Q) \Right4{} 
	H^0(\Sigma_5;H^1(C_4;Q)). \]
Since $H^0(C_4;Q)=H^1(C_4;Q)=0$ ($C_4$ acts on $Q\cong C_5\times C_5\times 
C_5$ via multiplication by scalars), this proves that $H^1(H/Q;Q)=0$. Also, 
	\[ N_{\Out(Q)}(\Out_H(Q))/\Out_H(Q) \cong 
	N_{\GL_3(5)}(C_4\times\Sigma_5)/(C_4\times\Sigma_5) \]
is trivial since $\GL_3(5)\cong C_4\times\PSL_3(5)$ and $\GO_3(5)\cong\Sigma_5$ is 
a maximal subgroup of $\PSL_3(5)$ (see, e.g., \cite[Theorem 
6.5.3]{GLS3}). So $\Out(\calf)=\Out(H)=1$. 
\end{proof}

We now summarize what we need to know here about tameness of fusion 
systems of finite simple groups.

\begin{Prop} \label{p:simple=>tame}
Fix a known simple group $G$, choose $S\in\sylp{G}$, 
and assume that $S\nnsg\calf_S(G)$. Then $\calf_S(G)$ is tamely realized by 
some known simple group $G^*$. 
\end{Prop}

\begin{proof} Set $\calf=\calf_S(G)$ and $\call=\call_S^c(G)$ for short. 
Note that $G$ is nonabelian since $S\nnsg\calf_S(G)$.

Assume first that $G\cong A_n$ for some $n\ge5$. By \cite[Proposition 
4.8]{AOV1}, if $p=2$ and $n\ge8$ or if $p$ is odd and $p^2\le n\equiv0,1$ (mod 
$p$), then $\kappa_G$ is an isomorphism. If $p$ is odd and $p^2<n\equiv k$ (mod 
$p$) where $2\le k\le p-1$, then $\calf$ is still tamely realized by $A_n$: 
$\Out(\call)=1$ since $\calf$ is isomorphic to 
the fusion system of $\Sigma_n$ and also that of $\Sigma_{n-k}$. If $p=2$ 
and $n=6,7$, then $\calf$ is tamely realized by $A_6\cong\PSL_2(9)$ (and 
$\kappa_{A_6}$ is an isomorphism). In all other cases, $S$ is abelian and 
hence $S\nsg\calf$. 

If $G$ is of Lie type in defining characteristic $p$, or if $p=2$ and 
$G\cong\lie2F4(2)'$, then by \cite[Theorems A and D]{BMO2}, $\kappa_G$ is 
an isomorphism except when $p=2$ and $G\cong\SL_3(2)$. In this exceptional 
case, $\calf$ is tamely realized by $A_6$ again. 

If $G$ is of Lie type in defining characteristic $q_0$ for some prime 
$q_0\ne p$, then by \cite[Theorem B]{BMO2}, $\calf$ is tamely realized by 
some other simple group $G^*$ of Lie type. See also Tables 0.1--0.3 in 
\cite{BMO2} for a list of which groups of Lie type do tamely realize their 
fusion system, and when they do not, which other groups they can be 
replaced by.

If $G$ is a sporadic simple group (and $S\nnsg\calf$), then by 
\cite[Theorem A]{sportame} and Lemma \ref{l:FS(L).tame}, $\kappa_G$ is an 
isomorphism except when $(G,p)$ is one of the pairs $(M_{11},2)$ or 
$(\He,3)$. If $(G,p)=(\He,3)$, then by the same theorem, $|\Out(G)|=2$ and 
$\Out(\call)=1$, so $\calf$ is still tamely realized by $G$. If 
$(G,p)=(M_{11},2)$, then $\calf$ is the unique simple fusion system over 
$\SD_{16}$, and is tamely realized by $G^*=\PSU_3(5)$ (and $\kappa_{G^*}$ 
is an isomorphism) by \cite[Proposition 4.4]{AOV1}. 
\end{proof}

The statements in the next proposition are very similar to results proven 
in \cite{AOV1}, but except for part (a) are not stated there explicitly. 
Their proof consists mostly of repeating those arguments. Since 
many of the results referred to in \cite[\S2]{AOV1} require considering 
linking systems that are not centric, they depend in a crucial way on 
\cite[Lemma 1.17]{AOV1}, which states that 
$\Out(\call_0)\cong\Out(\call)$ whenever $\call_0\le\call$ are linking 
systems associated to the same fusion system $\calf$ and 
$\Ob(\call_0)\subseteq\Ob(\call)$ are both $\Aut(\calf)$-invariant.

\begin{Prop} \label{p:tame=>tame}
Let $\calf$ be a saturated fusion system over a finite $p$-group $S$.
\begin{enuma} 

\item If $\calf/Z(\calf)$ is tamely realized by the finite $p'$-reduced 
$\KC$-group $\4G$, then $\calf$ is tamely realized by a finite $p'$-reduced 
$\KC$-group $G$ such that $G/Z(G)\cong\4G$. 

\item Assume $\calf_0\nsg\calf$ is a characteristic subsystem over $S_0\nsg 
S$, with the property that $\calf_0^{cr}\subseteq\calf^c$. If $\calf_0$ is 
tamely realized by a finite $p'$-reduced $\KC$-group $G_0$, then $\calf$ is 
tamely realized by a finite $p'$-reduced $\KC$-group $G$ such that $G_0\nsg 
G$. 

\item If $\calf_0\nsg\calf$ is a characteristic subsystem of index prime to 
$p$, and $\calf_0$ is tamely realized by a finite $p'$-reduced $\KC$-group 
$G_0$, then $\calf$ is tamely realized by a finite $p'$-reduced $\KC$-group 
$G$ such that $G_0\nsg G$. 

\item If $\calf_0\nsg\calf$ is a characteristic subsystem of $p$-power 
index, $\calf_0$ is tamely realized by a finite $p'$-reduced $\KC$-group 
$G_0$, and $Z(\calf)=1$, then $\calf$ is tamely realized by a finite 
$p'$-reduced $\KC$-group $G$ such that $G_0\nsg G$.

\end{enuma}
\end{Prop}

\begin{proof} \textbf{(a) } Assume $\calf/Z(\calf)$ is tamely realized by a 
finite $p'$-reduced $\KC$-group $\4G$. Then $Z(\4G)=Z(\calf/Z(\calf))$ by 
Proposition \ref{p:ZF=ZG}. So by \cite[Proposition 2.18]{AOV1}, $\calf$ is 
tamely realized by a finite $p'$-reduced group $G$ such that 
$G/Z(G)\cong\4G$. For each component $C$ of $G$, the subgroup 
$CZ(G)/Z(G)\cong C/(Z(C)\cap Z(G))$ is a component of $\4G$, and so $G$ is 
also a $\KC$-group.

\smallskip

\noindent\textbf{(b) } Assume $\calf_0\nsg\calf$ is characteristic over 
$S_0\nsg S$ with $\calf_0^{cr}\subseteq\calf^c$. Let $\calh_0$ be the set 
of all $P\in\calf^c$ such that $P\le S_0$, and let $\calh$ be the set of 
all $P\le S$ such that $P\cap S_0\in\calh_0$. For each $P\in\calh$, $P\cap 
S_0\in\calf^c$ by assumption, and hence $P\in\calf^c$. Thus 
$\calh\subseteq\calf^c$, and $\calf_0^{cr}\subseteq\calh_0$ since 
$\calf_0^{cr}\subseteq\calf^c$. So by \cite[Lemma 1.30]{AOV1} and the 
existence of a centric linking system associated to $\calf$ (Theorem 
\ref{t:unique.l.s.}), there is a normal pair of linking systems 
$\call_0\nsg\call$ associated to $\calf_0\nsg\calf$ with object sets 
$\calh_0$ and $\calh$. Furthermore, 
	\[ C_{\Aut_\call(S_0)}(\call_0) = \delta_{S_0}(C_S(\call_0)) 
	\leq\delta_{S_0}(C_S(\calf_0)) \leq\delta_{S_0}(C_S(S_0)) 
	\leq\delta_{S_0}(S_0) \leq \Aut_{\call_0}(S_0): \]
the equality and first inequality by Lemma \ref{centric<=>centric}(a), the 
third inequality since $S_0\in\calf_0^{cr}\subseteq\calf^c$, and the other 
two by definition. So $\call_0$ is centric in $\call$. 

Assume $\calf_0$ is tamely realized by a finite $p'$-reduced 
$\KC$-group 
$G_0$ with $S_0\in\sylp{G_0}$. By Proposition \ref{p:ZF=ZG}, we have 
$Z(G_0)=Z(\calf_0)$. Then $\call_0\cong\call_{S_0}^{\calh_0}(G_0)$ (the 
full subcategory of $\call_{S_0}^c(G_0)$ with objects the set $\calh_0$) by 
the uniqueness of linking systems. By definition, the sets of objects 
$\calh_0$ and $\calh$ are invariant under the actions of $\Aut(\calf_0)$ 
and $\Aut(\calf)$, respectively. Also, $\call_0$ is 
$\Aut(\call)$-invariant, since $\calf_0$ is characteristic in $\calf$, and 
$\call_0=\pi^{-1}(\calf_0)$ by \cite[Lemma 1.30]{AOV1} (where 
$\pi\:\call\too\calf$ is the structure functor for $\call$). All hypotheses 
in \cite[Proposition 2.16]{AOV1} are thus satisfied, and so $\calf$ is 
tamely realized by some finite group $G$ such that $G_0\nsg G$ and 
$G/G_0\cong\Aut_\call(S_0)/\Aut_{\call_0}(S_0)$. 

By the Frattini argument, $G=G_0N_G(S_0)$, and hence 
	\[ N_G(S_0)/N_{G_0}(S_0) \cong G/G_0 \cong 
	\Aut_\call(S_0)/\Aut_{\call_0}(S_0). \]
Since $\Aut_\call(S_0)=N_G(S_0)/O^p(C_G(S_0))$ and similarly for $\call_0$, 
this proves that $O^p(C_G(S_0))=O^p(C_{G_0}(S_0))$. Also, 
	\[ C_{G_0}(S_0)=Z(S_0)\times O^p(C_{G_0}(S_0)) 
	\qquad\textup{and}\qquad 
	C_{G}(S_0)=Z(S_0)\times O^p(C_{G}(S_0)): \]
the last equality since $S_0\in\calf_0^{cr}\subseteq\calf^c$. So 
$C_G(S_0)=C_{G_0}(S_0)\le G_0$. In particular, since $[O_{p'}(G),G_0]\le 
O_{p'}(G_0)=1$, this implies that $O_{p'}(G)\le C_G(S_0)\le G_0$. So 
$O_{p'}(G)\le O_{p'}(G_0)=1$, and hence $G$ is $p'$-reduced. 

If $C$ is a component of $G$, then since $G_0\nsg G$, either $C$ is a 
component of $G_0$ or $[C,G_0]=1$ (see \cite[31.4]{A-FGT} or \cite[Lemma 
A.12]{AKO}). Since $C_G(G_0)\le C_G(S_0)\le G_0$, and since all components 
of $G$ contained in $G_0$ are subnormal in $G_0$ and hence in $\Comp(G_0)$, 
this shows that $\Comp(G)\subseteq\Comp(G_0)$, and hence that $G$ is also a 
$\KC$-group.

\smallskip

\noindent\textbf{(c) } By \cite[Lemma I.7.6(a)]{AKO}, we have 
$\calf_0^c=\calf^c$. (See also Definition \ref{d:reduced}(c).) Hence (c) 
is a special case of (b).

\smallskip

\noindent\textbf{(d) } Assume $\calf_0$ has $p$-power index in $\calf$ and 
$Z(\calf)=1$. By \cite[Proposition 3.8(b)]{BCGLO2}, a subgroup $P\le S_0$ 
is $\calf$-quasicentric if and only if it is $\calf_0$-quasicentric. Let 
$\calh$ be the set of all $P\le S$ such that $P\cap S_0$ is 
$\calf_0$-quasicentric. Then $\calh\subseteq\calf^q$ since overgroups of 
quasicentric subgroups are quasicentric, and $\calh\supseteq\calf^{rc}$ by 
\cite[Lemma 1.20(d)]{AOV1}. By Theorem \ref{t:unique.l.s.}, there is a 
unique linking system $\call$ associated to $\calf$ with 
$\Ob(\call)=\calh$, and by \cite[Theorem 4.4]{BCGLO2}, there is a unique 
linking system $\call_0\le\call$ associated to $\calf_0$ with 
$\Ob(\call_0)=(\calf_0)^q$. Then $\call_0\nsg\call$: the condition on 
objects (Definition \ref{d:Lnormal}(a)) holds by construction, and the 
invariance condition (\ref{d:Lnormal}(b)) holds by the uniqueness of 
$\call_0$.

By Lemma \ref{centric<=>centric}(c), there is an action of $\call/\call_0$ 
on $C_S(\call_0)$ such that $C_{C_S(\calf_0)}(\call/\call_0)=Z(\calf)$, 
where $Z(\calf)=1$ by assumption. Also, 
$\call/\call_0=\Aut_\call(S_0)/\Aut_{\call_0}(S_0)$ is a $p$-group since 
$\calf_0$ has $p$-power index in $\calf$, so $C_S(\call_0)=1$. Hence 
$Z(\calf_0)=1$ and $C_{\Aut_\call(S_0)}(\call_0)=1$ by Lemma 
\ref{centric<=>centric}(a), and in particular, $\call_0$ is centric in 
$\call$. 

If in addition, $\calf_0$ is characteristic in $\calf$, then each 
$\alpha\in\Aut(\call)$ induces an element 
$\beta=\til\mu_\call(\alpha)\in\Aut(\calf)$ (Definition 
\ref{d:til.mu}), and $\8\beta(\calf_0)=\calf_0$ by assumption. Hence 
$\alpha(\call_0)=\call_0$ by Proposition \ref{AutI} and the 
uniqueness of the linking systems in \cite[Theorem 4.4]{BCGLO2}. Also, by 
construction, $\Ob(\call_0)$ and $\Ob(\call)$ are invariant under the 
actions of $\Aut(\calf_0)$ and $\Aut(\calf)$.

Assume $\calf_0$ is tamely realized by a finite $p'$-reduced 
$\KC$-group $G_0$. By Proposition \ref{p:ZF=ZG}, we have 
$Z(G_0)=Z(\calf_0)$. By Theorem \ref{t:unique.l.s.} (the uniqueness of 
centric linking systems), $\call_0\cong\call_{S_0}^q(G_0)$. The hypotheses 
of \cite[Proposition 2.16]{AOV1} thus hold, and so $\calf$ is tamely 
realized by some group $G$ such that $G_0\nsg G$ and 
$G/G_0\cong\Aut_\call(S_0)/\Aut_{\call_0}(S_0)$. In particular, $G/G_0$ is 
a $p$-group, so all components of $G$ are in $G_0$, and 
$\Comp(G)=\Comp(G_0)$. Also, $O_{p'}(G)=O_{p'}(G_0)=1$, and so $G$ is 
a $p'$-reduced $\KC$-group. 
\end{proof}

We are now ready to prove our main theorem. As explained in the 
introduction, Theorem \ref{t:E.real}, as well as Theorems \ref{t:comp.real} 
and \ref{t:all.tame}, have been formulated so that their proofs are 
independent of the classification of finite simple groups. 

Recall that by Lemma \ref{l:same.comps}(b,c), the condition 
$\Comp(\cale)=\Comp(\calf)$ in the statement of Theorem \ref{t:E.real} is 
satisfied whenever $\cale$ is centric in $\calf$, and these two conditions 
are in fact equivalent if $O_p(\calf)=1$. 

\begin{Thm} \label{t:E.real}
Let $\cale\nsg\calf$ be a normal pair of fusion systems over $T\nsg S$ such 
that $\Comp(\cale)=\Comp(\calf)$. Assume that $\cale$ is realized by a 
finite $p'$-reduced group all of whose components are known quasisimple 
groups. Then $\calf$ is tamely realized by a finite $p'$-reduced group all 
of whose components are known quasisimple groups.
\end{Thm}

\begin{proof} Let $\scrs$ be the set of all triples $(\calf,\cale,H)$ such 
that 
\begin{itemize} 

\item $\cale\nsg\calf$ are saturated fusion 
systems over finite $p$-groups $T\nsg S$ such that 
$\Comp(\cale)=\Comp(\calf)$; 

\item $H$ is a $p'$-reduced $\KC$-group such that $T\in\sylp{H}$ and 
$\cale=\calf_T(H)$; and 

\item $\calf$ is not tamely realized by any finite $p'$-reduced 
$\KC$-group. 

\end{itemize}
Assume the theorem does not hold; i.e., that $\scrs\ne\emptyset$. Let 
$(\calf,\cale,H)\in\scrs$ be such that 
$(|\Mor(\cale)|,|\Mor(\calf)|)\in\N^2$ is the smallest possible under the 
lexicographic ordering. In other words, there are no triples 
$(\calf^*,\cale^*,H^*)$ in $\scrs$ where $|\Mor(\cale^*)|<|\Mor(\cale)|$; 
and among those where $|\Mor(\cale^*)|=|\Mor(\cale)|$, there are none where 
$|\Mor(\calf^*)|<|\Mor(\calf)|$. 

We show in Step 1 that $O_p(\calf)=1$, and that $H$ can be chosen to 
be a product of known finite simple groups. We then show in Step 2 that 
the components of $\calf$ are all normal in $\calf$, and reduce this to a 
contradiction in Step 3.

\smallskip

\noindent\textbf{Step 1: } Let $L_1,\dots,L_k$ be the components of $H$, 
and set $U_i=T\cap L_i\in\sylp{L_i}$. Set $\cald_i=\calf_{U_i}(L_i)$ 
and $\calc_i=O^{p'}(\cald_i)$ for each $i$. Assume that the $L_i$ 
are ordered so that for some $m$, $\calc_i$ is quasisimple if and 
only if $i\le m$. We are thus in the situation of Proposition 
\ref{p:Comp(F)-1}, with $H$, $T$, $U_i$, and $\cald_i$ in the roles of $G$, 
$S$, $T_i$, and $\cale_i$. So 
$\Comp(\calf)=\Comp(\cale)=\{\calc_1,\dots,\calc_m\}$ by that proposition. 

Set $U=U_1\cdots U_m$ and $\cale_0=\calf_U(L_1\cdots L_m)\le\cale$. By 
Proposition \ref{p:Comp(F)-2}(a), there is a unique minimal subsystem 
$\calc^*\le\cale_0$ over $U$ containing $O^{p'}(\cale_0)$ that is realized 
by a $p'$-reduced $\KC$-group, and $\calc^*$ is realized by a central 
product $H_0$ of known finite $p'$-reduced quasisimple groups. By 
Proposition \ref{p:Comp(F)-2}(b), $\calc^*$ is characteristic in $\cale$ 
and hence normal in $\calf$. Hence $(\calf,\calc^*,H_0)\in\scrs$, so 
$\cale=\calc^*$ by the minimality of $|\Mor(\cale)|$, and we can take 
$H=H_0$. In particular, $m=k$, and $H$ is a central product of known 
finite $p'$-reduced quasisimple groups.

Set $Q=O_p(\calf)$, and set $S_0=N_S^{\Inn(Q)}(Q)=QC_S(Q)$ and 
$\calf_0=N_\calf^{\Inn(Q)}(Q)$. Thus $\calf_0$ is a fusion subsystem over 
$S_0$. Since $Q^\calf=\{Q\}$, the subgroup $Q$ is 
fully $\Inn(Q)$-normalized in $\calf$ and hence $\calf_0$ is saturated (see 
Definition I.5.1 and Theorem I.5.5 in \cite{AKO}). Also, $\calf_0$ is 
weakly normal in $\calf$ by \cite[Proposition 1.25(c)]{AOV1}, and is normal 
since $C_S(S_0)\le S_0$ (the extension condition holds). For each 
$\alpha\in\Aut(\calf)$, $\alpha(Q)=Q$, so $c_\alpha(\calf_0)=\calf_0$, and 
hence $\calf_0$ is characteristic in $\calf$.

If $P\in\calf_0^{cr}$, then $P\ge Q$ since $Q\nsg\calf_0$ (see 
\cite[Proposition I.4.5(a$\Rightarrow$b)]{AKO}). So for each $P^*\in 
P^\calf$, $P^*\ge Q$, and $C_S(P^*)\le C_S(Q)\le S_0$. Thus 
$C_S(P^*)=C_{S_0}(P^*)\le P^*$ since $P^*\in\calf_0^c$, so $P\in\calf^c$. 
Thus $\calf_0^{cr}\subseteq\calf^c$. 

By Proposition \ref{p:tame=>tame}(b) and since $\calf$ is not tamely 
realized by any $p'$-reduced $\KC$-group, the subsystem $\calf_0$ is not 
tamely realized by any $p'$-reduced $\KC$-group. Also, $\cale\nsg 
N_\calf^{\Inn(Q)}(Q)=\calf_0$ by Lemma \ref{l:C*nsgCF(Q)} (with $\cale$ in 
the role of $\calc^*$). Thus $(\calf_0,\cale,H)\in\scrs$, and by the 
minimality assumption, we have $\calf=\calf_0$. So $\autf(Q)=\Inn(Q)$. 

Let $1=Z_0(Q)\le Z_1(Q)\le Z_2(Q)\le \cdots\le Q$ be the upper 
central series of $Q=O_p(\calf)$. Thus for each $i$, 
$Z_{i+1}(Q)/Z_i(Q)=Z(Q/Z_i(Q))=Z(\calf/Z_i(Q))$ since 
$\autf(Q)=\Inn(Q)$. By Proposition 
\ref{p:tame=>tame}(a), if $\calf/Z_{i+1}(Q)$ is tamely realized by a finite 
$p'$-reduced $\KC$-group $G_{i+1}$, then $\calf/Z_i(Q)$ is tamely 
realized by a finite $p'$-reduced $\KC$-group $G_i$. Since $\calf$ is 
not tamely realized by any $p'$-reduced $\KC$-group by assumption, we 
conclude that $\calf/Q$ is not tamely realized by any $p'$-reduced 
$\KC$-group either.

For each $P,R\le C_S(Q)$ and each $\varphi\in\homf(P,R)$, the morphism 
$\varphi$ extends to some $\4\varphi\in\homf(PQ,RQ)$ since $Q\nsg\calf$, 
and $\varphi|_Q=c_g|_Q$ for some $g\in Q$ since $\autf(Q)=\Inn(Q)$. Hence 
$c_g^{-1}\4\varphi|_P=\varphi$ since $c_g|_{C_S(Q)}=\Id$. Thus each 
morphism in $\calf$ between subgroups of $C_S(Q)$ lies in $C_\calf(Q)$, and 
so $\calf/Q\cong C_\calf(Q)/Z(Q)$ by Lemma \ref{l:F/Q=E/Q0}, applied with 
$C_\calf(Q)$ in the role of $\cale$.

Set $Z=Z(Q)$. We have now shown that $C_\calf(Q)/Z$ is not realized by 
any $p'$-reduced $\KC$-group. Also, $\cale\nsg C_\calf(Q)$ by Lemma 
\ref{l:C*nsgCF(Q)} (applied with $\cale$ in the role of $\calc^*$), so 
$Z\cale\nsg C_\calf(Q)$ by Lemma \ref{l:E/Z<|F/Z}, and $Z\cale/Z\nsg 
C_\calf(Q)/Z$ by Lemma \ref{l:F/Q}. By Lemma \ref{l:F/Q=E/Q0}, 
$Z\cale/Z\cong\cale/(Z\cap T)$, where $\cale/(Z\cap T)$ is realized by 
$H/(Z\cap T)$ (see \cite[Theorem 5.20]{Craven}). So $Z\cale/Z$ is realized 
by a $p'$-reduced $\KC$-group $H_0\cong H/(Z\cap T)$. Also, 
	\[ \Comp(Z\cale/Z) = \{Z\calc_1/Z,\dots,Z\calc_k/Z\} 
	= \Comp(C_\calf(Q)/Z)\]
by Lemma \ref{l:ZC/Z}. So $(C_\calf(Q)/Z,Z\cale/Z,H_0)\in\scrs$, and by 
the minimality assumption on $(\calf,\cale,H)$, we have 
$\cale\cong Z\cale/Z$ and $\calf\cong C_\calf(Q)/Z(Q)$ 
and thus $O_p(\calf)=Q=1$. 

To summarize, we have reduced to the case where 
$(\calf,\cale,H)\in\scrs$ satisfies:
	\beqq \parbox{135mm}{$O_p(\calf)=O_p(\cale)=1$, 
	$\cale=\calf_T(H)$ where $H=\xxx{L}$, and each $L_i$ is a 
	known finite simple group. Also, 
	$\Comp(\calf)=\Comp(\cale)=\{\calc_1,\dots,\calc_k\}$ where 
	$T=\xxx{U}$, $U_i\in\sylp{L_i}$, and 
	$\calc_i=O^{p'}(\calf_{U_i}(L_i))$.}
	\label{e:Step1-red} \eeqq

\smallskip

\noindent\textbf{Step 2: } Let $\caln\nsg\calf$ be the characteristic 
subsystem constructed in Proposition \ref{p:NF(EJ)}: a subsystem over 
$N=\bigcap_{i=1}^kN_S(U_i)$ normal in $\calf$ and containing each component 
$\calc_i$ as a normal subsystem. In particular, for each 
$\alpha\in\Aut_\caln(T)$, $\alpha(U_i)=U_i$ for all $1\le i\le k$.

For each $P\in\caln^c$ and each $Q\in P^\calf$, $Q\in\caln^c$ since 
$\caln\nsg\calf$, and so $Q\ge Z(N)$. For each $1\le i\le k$, we have 
$U_i\nsg N$ by definition of $N$, and hence $Q\cap U_i\ne1$. Each $x\in 
C_S(Q)$ centralizes each $Q\cap U_i$ and hence normalizes each subgroup 
$U_i$ (recall that each element of $S$ permutes the $U_i$). So 
$C_S(Q)=C_{N}(Q)\le Q$, and $P\in\calf^c$. Thus $\caln^c\subseteq\calf^c$. 
By Proposition \ref{p:tame=>tame}(b) and since $\calf$ is not tamely 
realized by a $p'$-reduced $\KC$-group, $\caln$ is not tamely 
realized by a $p'$-reduced $\KC$-group either.

Now, $\cale$ is weakly normal in $\caln$ by \cite[Proposition 8.17]{Craven} 
and since $\cale\nsg\calf$ and $\caln\le\calf$. By the definition of 
$\caln$ in Proposition \ref{p:NF(EJ)}, if $T\le P\le N$, and 
$\alpha\in\autf(P)$ is such that $\alpha|_T\in\Aut_\caln(T)$, then 
$\alpha\in\Aut_\caln(P)$. So the extension condition for $\cale\le\caln$ 
follows from that for $\cale\nsg\calf$, and hence $\cale\nsg\caln$. 

Thus $(\caln,\cale,H)\in\scrs$, and $\calf=\caln$ by the minimality 
assumption. In other words, $(\calf,\cale,H)\in\scrs$ satisfies:
	\beqq \parbox{135mm}{\eqref{e:Step1-red} holds, and 
	$\calc_i\nsg\calf$ for each $1\le i\le k$.}
	\label{e:Step2-red} \eeqq

\smallskip

\noindent\textbf{Step 3: } Set $\calc=E(\calf)=\xxx\calc$. By 
Proposition \ref{p:Comp(F)-3} (applied with $H$ and $\cale$ 
in the roles of $G$ and $\calf$), there is a unique 
minimal normal fusion subsystem $\calc^*=\xxx{\calc^*}\nsg\cale$ containing 
$\calc$ that is realized by a product of known finite simple groups. 
Furthermore (by the same proposition), $\calc=O^{p'}(\calc^*)$, $\calc^*$ 
is centric and characteristic in $\cale$, and for each $i$, 
\begin{itemize} 
\item $\calc_i\le\calc_i^*\le\calf_{U_i}(L_i)$, 
\item $\calc_i=O^{p'}(\calc_i^*)$, and 
\item $\calc_i^*=\calf_{U_i}(H_i^*)$ where $H_i^*$ is a known finite simple 
group. 
\end{itemize}
Then $\calc^*\nsg\calf_T(H)=\cale\nsg\calf$, and so $\calc^*\nsg\calf$ by 
Lemma \ref{l:D<E<F} and since $\calc^*$ is characteristic in $\cale$. 
Set $H^*=\xxx{H^*}$, so that $\calc^*=\calf_T(H^*)$. Thus 
$(\calf,\calc^*,H^*)\in\scrs$, and $\cale=\calc^*$ by the 
minimality assumption.

Let $\Aut^0(\calc^*)\le\Aut(\calc^*)$ be the subgroup of all automorphisms 
that send each $U_i$ to itself. Then $\Aut_\calf(T)\le\Aut^0(\calc^*)$ 
by \eqref{e:Step2-red} and since $\calc^*\nsg\calf$. Each factor 
$\calc_i^*$ is a full subcategory of $\calc^*$ (contains all morphisms in 
$\calc^*$ between subgroups of $U_i$), and hence each 
$\alpha\in\Aut^0(\calc^*)$ sends each $\calc_i^*$ to itself. So 
	\[ \Aut_\calf(T)/\Aut_{\calc^*}(T)\le\Aut^0(\calc^*)/\Aut_{\calc^*}(T)
	\cong\prod_{i=1}^k\Out(\calc_i^*). \] 
By assumption, each $\calc^*_i$ is realized by a known finite simple group, 
hence is tamely realized by a known finite simple group by Proposition 
\ref{p:simple=>tame}. Since $\Out(K)$ is solvable for each known finite 
simple group $K$ (see \cite[Theorem 7.1.1(a)]{GLS3}), the groups 
$\Out(\calc^*_i)$ are also solvable. So $\Aut_\calf(T)/\Aut_{\calc^*}(T)$ is 
solvable. 

The hypotheses of \cite[Theorem 5(b)]{O-red-corr} thus hold for the pair 
$\calc^*\nsg\calf$. By that theorem, there is a sequence 
$\calc^*=\calf_0\nsg\calf_1\nsg\cdots\nsg\calf_m=\calf$ of saturated fusion 
subsystems, for some $m\ge0$, such that 
\begin{enumi} 
\item for each $0\le j<m$, $\calf_j$ is normal of $p$-power index or index 
prime to $p$ in $\calf_{j+1}$ and $\calc^*\nsg\calf_j\nsg\calf$; and 
\item for each $1\le j\le m$ and each $\alpha\in\Aut(\calf_{j})$ with 
$\8\alpha(\calc^*)=\calc^*$, we have $\8\alpha(\calf_{j'})=\calf_{j'}$ for all 
$0\le j'<j$.
\end{enumi}

Recall that $\Comp(\calf)=\{\calc_1,\dots,\calc_k\}$. 
For each $0\le j\le m-1$, 
$\Comp(\calf_{j})\subseteq\Comp(\calf)$ since $\calf_{j}\snsg\calf$, 
and the opposite inclusion holds since 
$\calc_i\nsg\calc\nsg\calc^*\nsg\calf_{j}$ for each $i$. Hence $\calc$ 
is characteristic in $\calf_{j}$. For each $\alpha\in\Aut(\calf_j)$, 
$\alpha|_T\in\Aut(\calc)$ since $\calc$ is characteristic, and 
$\Aut(\calc)=\Aut(\calc^*)$ by Proposition \ref{p:Comp(F)-2}(a). So 
$\8\alpha(\calc^*)=\calc^*$, and hence 
$\8\alpha(\calf_{j'})=\calf_{j'}$ for all $0<j'<j$ by condition (ii) above. 

In particular, this shows that $\calf_j$ is characteristic in $\calf_{j+1}$ 
for each $j$. Also, $Z(\calf_j)\le O_p(\calf_j)=1$ for each $j$ by Lemma 
\ref{l:E<|F}(b) and since $O_p(\calf)=1$ and $\calf_j\snsg\calf$. So by 
Proposition \ref{p:tame=>tame}(c,d), and since $\calf_j$ has index prime to 
$p$ or $p$-power index in $\calf_{j+1}$ and $Z(\calf_j)=1$, if $\calf_j$ is 
tamely realized by a finite $p'$-reduced $\KC$-group $G_j$, then 
$\calf_{j+1}$ is tamely realized by a finite $p'$-reduced $\KC$-group 
$G_{j+1}\ge G_j$. 

For each $1\le i\le k$, $\calc_i^*$ is tamely realized by some known finite 
simple group by Proposition \ref{p:simple=>tame}, and so $\calc^*=\calf_0$ 
is tamely realized by a product of known finite simple groups by 
Proposition \ref{p:prod.tame}(c). Hence $\calf_i$ is tamely realized 
by a $p'$-reduced $\KC$-group for each $1\le i\le m$. This contradicts our 
assumption on $\calf=\calf_m$, and we conclude that $\scrs=\emptyset$. 
\end{proof}

Note that Theorem \ref{ThA} is just Theorem \ref{t:E.real} without 
mentioning tameness.

We now list some special cases of Theorem \ref{t:E.real}.

\begin{Thm} \label{t:comp.real}
Let $\calf$ be a saturated fusion system over a finite $p$-group $S$. If 
all components of $\calf$ are realized by known finite quasisimple groups, 
then $\calf$ is tamely realized by a finite $p'$-reduced group 
all of whose components are known quasisimple groups. 
\end{Thm}

\begin{proof} This is the special case of Theorem \ref{t:E.real} where 
$\cale$ is the generalized Fitting subsystem of $\calf$. Note that 
$\cale$ is realizable since it is the central product of its components 
(which are realizable by assumption) and a $p$-group. 
\end{proof}

Our third theorem is the special case of Theorem \ref{t:E.real} where 
$\cale=\calf$.

\begin{Thm} \label{t:all.tame}
Let $p$ be a prime, and let $\calf$ be a fusion system over a finite 
$p$-group that is realized by a finite $p'$-reduced group all of whose 
components are known quasisimple groups. Then $\calf$ is tamely realized by 
a finite $p'$-reduced group all of whose components are known quasisimple 
groups. 
\end{Thm}

\end{document}